\documentclass[a4paper,11pt]{article} 

\usepackage[utf8]{inputenc} 
\usepackage{amsfonts}
\usepackage{amssymb}
\usepackage{amsthm}
\usepackage{amsmath}
\usepackage{a4wide}
\usepackage{mathrsfs}
\usepackage{mathtools}
\usepackage{epsfig}
\usepackage{palatino}
\usepackage{tikz}
\usepackage{bm}
\usepackage{fp,ifthen}
\usepackage{nicefrac}
\usetikzlibrary{decorations.pathreplacing}
\usepackage{float}
\usepackage{tcolorbox}
\usepackage{enumerate} 
\usepackage{enumitem} 
\usepackage{bbm}

\usepackage[italian,english]{babel}

\usepackage{nccmath}

\mathtoolsset{showonlyrefs}

\let\OLDthebibliography\thebibliography
\renewcommand\thebibliography[1]{
  \OLDthebibliography{#1}
  \setlength{\parskip}{0pt}
  \setlength{\itemsep}{5pt plus 0.3ex}
}

\newcommand{\mres}{\mathbin{\vrule height 1.6ex depth 0pt width
0.13ex\vrule height 0.13ex depth 0pt width 1.3ex}}

\usepackage[a4paper,twoside,top=2.5cm, bottom=2cm, left=2.5cm, right=2.5cm]{geometry} 

\newcommand{\pdfgraphics}{\ifpdf\DeclareGraphicsExtensions{.pdf,.jpg}\else\fi}
\usepackage{graphicx}

\usepackage{color}
\definecolor{hanblue}{rgb}{0.27, 0.42, 0.81}
\definecolor{red}{rgb}{1.0, 0.0, 0.0}
\usepackage[colorlinks, citecolor=blue,linkcolor=blue, urlcolor = blue]{hyperref}

\theoremstyle{plain}

\newtheorem{teo}{Theorem}[section]
\newtheorem{lemma}[teo]{Lemma}
\newtheorem{prop}[teo]{Proposition}
\newtheorem{cor}[teo]{Corollary}

\theoremstyle{definition}
\newtheorem{defn}[teo]{Definition}

\newtheorem{rem}[teo]{Remark}

\theoremstyle{remark}

\numberwithin{equation}{section}

\newcommand \eps{\ensuremath{\varepsilon}} 
\renewcommand{\epsilon}{\varepsilon}
\newcommand{\N}{\ensuremath{\mathbb{N}}}
\newcommand{\Z}{\ensuremath{\mathbb Z}}
\newcommand{\R}{\ensuremath{\mathbb R}}
\newcommand{\C}{\ensuremath{\mathbb C}}

\newcommand{\Cube}{\ensuremath{\mathcal{Q}}}
\newcommand{\link}{\operatorname{link}}

\newcommand{\diam}{\operatorname{diam}}
\newcommand{\osc}{\operatorname{osc}}
\newcommand{\Per}{\operatorname{Per}}

\newcommand{\sign}{\operatorname{sign}}

\usepackage{esint}

\newcommand{\wconv}[1]{\xrightharpoonup{\,#1\,}}

\newcommand{\s}{\hspace{7pt}}
\newcommand{\sms}{\hspace{4.5pt}}
\newcommand{\vsp}{\vspace{3pt}}
\newcommand{\Sp}{\mathbb{S}}
\newcommand{\ep}{\varepsilon}

\newcommand{\M}{\mathbb{M}}

\newcommand{\F}{\mathbf{F}}

\usepackage{tcolorbox}

\DeclarePairedDelimiter{\abs}{\lvert}{\rvert}
\DeclarePairedDelimiter{\norm}{\lVert}{\rVert}

\makeatletter
\renewcommand*\env@matrix[1][*\c@MaxMatrixCols c]{%
  \hskip -\arraycolsep
  \let\@ifnextchar\new@ifnextchar
  \array{#1}}
\makeatother

\usepackage{url}
\usepackage{authblk}

\makeatletter
\newcommand{\addressa}[1]{\gdef\@addressa{#1}}
\newcommand{\emaila}[1]{\gdef\@emaila{\url{#1}}}

\newcommand{\addressb}[1]{\gdef\@addressb{#1}}
\newcommand{\emailb}[1]{\gdef\@emailb{\url{#1}}}

\newcommand{\addressc}[1]{\gdef\@addressc{#1}}
\newcommand{\emailc}[1]{\gdef\@emailc{\url{#1}}}

\newcommand{\@endstuff}{\par\vspace{\baselineskip}\noindent
\begin{tabular}{@{}l}\scshape\@addressa\\\textit{E-mail address:} \@emaila\end{tabular} 

\vspace{12pt} \noindent
\begin{tabular}{@{}l}\scshape\@addressb\\ \textit{E-mail address:} \@emailb\end{tabular}

\vspace{12pt} \noindent
\begin{tabular}{@{}l}\scshape\@addressc\\ \textit{E-mail address:} \@emailc\end{tabular}
}

\AtEndDocument{\@endstuff}
\makeatother

\begin{document}

\pdfgraphics 

\title{A nonlocal approximation of the area in codimension two}


\author{Michele Caselli, Mattia Freguglia and Nicola Picenni}


\addressa{Michele Caselli \\ Scuola Normale Superiore, Piazza dei Cavalieri 7, 56126 Pisa, Italy \\   Stanford University, 450 Jane Stanford Way, Building 380, 94305 Stanford, CA}
\emaila{michele.caselli@sns.it}

\addressb{Mattia Freguglia \\ Scuola Normale Superiore\\ Piazza dei Cavalieri 7, 56126 Pisa, Italy }
\emailb{mattia.freguglia@sns.it}

\addressc{Nicola Picenni \\ University of Pisa, Department of Mathematics \\ Largo Bruno Pontecorvo 5, 56127 Pisa, Italy }
\emailc{nicola.picenni@unipi.it}


\date{}

\maketitle

\vspace{-0.5cm}
\begin{abstract}
\noindent 
For $s\in (0,1)$ we introduce a notion of fractional $s$-mass on $(n-2)$-dimensional closed, oriented surfaces in $\mathbb{R}^n$. Moreover, we prove its $\Gamma$-convergence, with respect to the flat topology, to the $(n-2)$-dimensional area, as $s$ tends to $1$.
\end{abstract}

\tableofcontents

\vspace{5ex}



\vsp


\section{Introduction}

In this paper, for every $s\in (0,1)$, we consider a notion of fractional $s$-mass for surfaces of codimension two in $\R^n$, with $n\geq 3$, and we prove that it suitably converges to the $(n-2)$-dimensional Hausdorff measure, as $s$ tends to $1$.

In order to define this quantity, let $\Sigma\subset \R^n$ be a closed (compact without boundary), oriented, $(n-2)$-dimensional surface of class $C^2$, with locally constant integer multiplicity. More precisely, we consider
\begin{equation}\label{defn:Sigma}
\Sigma:=d_1 \Sigma_1 \cup \dots \cup d_m \Sigma_m,
\end{equation}
where $m\geq 1$ is a positive integer, $d=(d_1,\dots,d_m)\in \N_+ ^m$, and $\Sigma_1,\dots,\Sigma_m$ are pairwise disjoint closed,  connected, oriented, $(n-2)$-dimensional surfaces of class $C^2$. With a little abuse of notation, we use the same symbol $\Sigma$ to denote the submanifold with multiplicities \eqref{defn:Sigma}, the associated integral current (see subsection~\ref{sbs: flat conv}) and the set $\Sigma_1\cup\dots \cup \Sigma_m\subset \R^n$.

Let us consider the following set of curves
\begin{equation}\label{defn:L(Sigma)}
L(\Sigma):=\big\{\gamma\subset \R^n\setminus \Sigma: \gamma \text{ is a bi-Lipschitz image of $\Sp^1$} \big\}.
\end{equation}

For every $s\in (0,1)$ we consider the following class of maps
\begin{multline}\label{Flink}
    \mathfrak{F}_{\hspace{-.7pt} s}(\Sigma)  := \Big\{ u \in C^1(\R^n \setminus \Sigma; \Sp^1) \colon  \mathcal{E}_{\frac{1+s}{2}}(u,\R^n) < +\infty \ \text{and}\\
    \abs{\deg(u, \gamma)} = \abs{d_1\link(\gamma,\Sigma_1)+\dots+d_m \link(\gamma,\Sigma_m)}\ \text{for any} \ \gamma\in L(\Sigma)\Big\},
\end{multline}
where $\abs{\deg(u, \gamma)}$ denotes the absolute value of the degree of the restriction of $u$ to $\gamma$, $\link(\gamma,\Sigma_i)$ denotes the linking number between $\gamma$ and $\Sigma_i$, that is how many times $\gamma$ winds around $\Sigma_i$ (see Subsection \ref{sbs linking} for a precise definition), and
\begin{equation}
    \label{def:E-s}
    \mathcal{E}_{\alpha}(u, \Omega)
    :=
    \iint_{\R^n \times \R^n \setminus \Omega^c \times \Omega^c} \frac{|u(x)-u(y)|^2}{|x-y|^{n+2 \alpha}} \, dxdy
\end{equation}
is defined more in general for any $\alpha \in (0,1)$ and every open set $\Omega \subset \R^n$. 

Moreover, if $u\in \mathfrak{F}_{\hspace{-.7pt} s}(\Sigma)$ we say that \emph{$u$ links with $\Sigma$}. For our purposes, it is useful to extend this definition of linking between a function and a surface to less regular functions. To this aim, we consider the following class of maps 
    \begin{equation}\label{defn:Fweak}
        \mathfrak{F}_{\hspace{-1pt} s}^{\hspace{.9pt} w}  (\Sigma):=\biggl\{ u \colon \R^n \to \Sp^1 \colon \exists \{u_k\}\subset \mathfrak{F}_{\hspace{-.7pt} s}  (\Sigma),\ u_k\to u \sms \mbox{in } H^{\alpha}_{\mathrm{loc}}(\R^n;\Sp^1) \ \forall \alpha \in \biggl( \hspace{-0.05cm} 0, \frac{1+s}{2}\biggr) \biggr\}.
    \end{equation}
In this case, if $u\in \mathfrak{F}_{\hspace{-1pt} s}^{\hspace{.9pt} w}  (\Sigma)$ we say that \emph{$u$ weakly links with $\Sigma$}.

We can now introduce the $s$-fractional mass of $\Sigma$ as the least possible fractional energy of a map that weakly links with $\Sigma$, namely 
\begin{equation}\label{fraccontdef}
   \M_{s}(\Sigma) := \min_{u \in  \mathfrak{F}_{\hspace{-1pt} s}^{\hspace{.9pt} w}  (\Sigma)} 
   \mathcal{E}_{\frac{1+s}{2}}(u, \R^n).
\end{equation}

Here the use of the exponent $(1+s)/2$ in the fractional space is motivated by the fact that this choice yields the expected scaling property of $\M_{s}$ (in analogy with the fractional perimeter), namely
$$\M_{s}(\lambda\cdot\Sigma)=\lambda^{n-1-s}\cdot \M_{s}(\Sigma) \qquad \forall \lambda >0,$$
where $\lambda\cdot \Sigma$ denotes $\Sigma$ rescaled by a factor $\lambda$.

We prove in Section \ref{sec 3} that for every $s\in (0,1)$ and every $\Sigma$ as above, the class $\mathfrak{F}_{\hspace{-.7pt} s}(\Sigma) $ is non-empty, and that the minimum in \eqref{fraccontdef} is strictly positive and is achieved by a minimizer in $\mathfrak{F}_{\hspace{-1pt} s}^{\hspace{.9pt} w} (\Sigma)$ that, in some weak sense, still enjoys the geometric property that defines the class $\mathfrak{F}_{\hspace{-.7pt} s}(\Sigma) $.

\subsection{Main result}

Our main result concerns the asymptotic behavior of $\M_{s}(\Sigma)$ in the limit as $s\to 1^-$. More precisely, we prove $\Gamma$-convergence with respect to the flat norm of boundaries (see Subsection~\ref{sbs: flat conv}) of $\M_{s}(\Sigma)$ to the $(n-2)$-Hausdorff measure with multiplicity, after multiplication by a suitable rescaling factor.

\smallskip

Let $\omega_{n-1}$ be the area of $\Sp^{n-1}$, then the precise statement is the following.

\begin{teo}\label{energy conv teo} 
Let $\Sigma\subset \R^n$ be as in \eqref{defn:Sigma}. Then it holds that
\begin{equation}\label{eq: gamma conv thm}
      \Gamma - \lim_{s \to 1^-} (1-s)^2 \M_{s}(\Sigma) = \frac{2\pi\omega_{n-1}}{n} \sum_{i=1} ^{m} d_i \mathcal{H}^{n-2}(\Sigma_i),
   \end{equation}
where the $\Gamma$-limit is intended with respect to the flat convergence of boundaries in $\R^n$.
\end{teo}

We observe that in codimension one, that is for the classical fractional perimeter, the correct rescaling factor to get a meaningful (nonzero, not identically infinite) $\Gamma-$limit is $(1-s)$ (see \cite{Gammaconv}). The reason why $(1-s)^2$ is the correct factor in the codimension two setting will be clear from our proofs, but the reader can see Lemma \ref{lem: vortex slicing} in the Appendix for a model case that enlightens this feature.

\smallskip

We obtain Theorem~\ref{energy conv teo} as a consequence of the following two estimates, which actually provide a stronger local result involving the fractional seminorm, that is
$$[u]_{H^{\alpha}(\Omega)}^2:=
    \iint_{\Omega \times \Omega} \frac{|u(x)-u(y)|^2}{|x-y|^{n+2 \alpha}} \, dxdy,$$
for any $\alpha \in (0,1)$ and every open set $\Omega \subset \R^n$. We point out that $\mathcal{E}_{\alpha}(u, \Omega) \ge [u]^2 _{H^{\alpha}(\Omega)}$.

\begin{prop}[Lower bound]\label{prop:Gamma-liminf}
Let $\{\Sigma_s\}_{s\in (0,1)}\subset \R^n$ and $\Sigma\subset \R^n$ be as in \eqref{defn:Sigma}, and let us assume that $\Sigma_s \to \Sigma$ as $s\to 1^{-}$ with respect to the flat convergence of boundaries. For every $s\in (0,1)$, let also $u_s\in \mathfrak{F}_{\hspace{-1pt} s}^{\hspace{.9pt} w}  (\Sigma_s)$ be any map that weakly links with $\Sigma_s$.

Then it holds that
\begin{equation}
\label{eq:liminf-ineq}
\liminf_{s\to 1^{-}} \hspace{0.03cm} (1-s)^2  [u_s]^2 _{H^{\frac{1+s}{2}}(\Omega)}\geq \frac{2\pi\omega_{n-1}}{n} \sum_{i=1} ^{m} d_i \mathcal{H}^{n-2}(\Sigma_i\cap\Omega),
\end{equation}
for every open set $\Omega\subseteq \R^n$.
\end{prop}

\begin{prop}[Upper bound]\label{prop:gamma-limsup}
Let $\Sigma\subset \R^n$ be as in \eqref{defn:Sigma}. Then there exists a function $u \colon \R^n\to \Sp^1$, such that $u \in \mathfrak{F}_{\hspace{-.7pt} s}(\Sigma) $ for every $s\in (0,1)$ and it holds that
    \begin{equation}
        \label{eq:limsup-ineq}
        \limsup_{s\to 1^-} \hspace{0.03cm} (1-s)^2 \mathcal{E}_{\frac{1+s}{2}}(u, \Omega) \leq \frac{2\pi\omega_{n-1}}{n} \sum_{i=1} ^{m} d_i ^2 \mathcal{H}^{n-2}(\Sigma_i \cap \overline{\Omega}).
    \end{equation}
for every open set $\Omega\subseteq \R^n$.
\end{prop}

\begin{rem}
We point out that the right-hand sides of \eqref{eq:liminf-ineq} and \eqref{eq:limsup-ineq} coincide when $\Sigma$ has unit multiplicity, namely when $d_1=\dots=d_m=1$. As a consequence, we deduce that in this case
    \[
    \lim_{s\to 1^-} (1-s)^2 \M_{s}(\Sigma) = \Gamma - \lim_{s \to 1^-} (1-s)^2 \M_{s}(\Sigma) = \frac{2\pi\omega_{n-1}}{n} \mathcal{H}^{n-2}(\Sigma).
    \]

In the case of higher multiplicities, the $\Gamma$-limsup estimate for Theorem~\ref{energy conv teo} follows by combining Proposition~\ref{prop:gamma-limsup} with the density result in Lemma~\ref{Lemma:density}. On the other hand, the value of the pointwise limit when some $d_i$ is larger than one is not provided by our analysis. Actually, in the proof of Proposition~\ref{prop:gamma-limsup} we also show that, with our choice of the function $u$, it holds
\begin{equation}
        \label{eq:point-liminf}
        \liminf_{s\to 1^-} \hspace{0.03cm} (1-s)^2 [u]^2 _{H^{\frac{1+s}{2}}(\Omega)} \geq \frac{2\pi\omega_{n-1}}{n} \sum_{i=1} ^{m} d_i ^2 \mathcal{H}^{n-2}(\Sigma_i \cap \Omega).
    \end{equation}

For this reason, we expect that, for a fixed $\Sigma$, estimate \eqref{eq:limsup-ineq} should be optimal, namely the pointwise limit of $(1-s)^2\M_s(\Sigma)$ should be equal to the $(n-2)$-dimensional area with squared multiplicities, and in particular it should be larger than the $\Gamma$-limit. However, a proof of this fact cannot be obtained by refining the proof of Proposition~\ref{prop:Gamma-liminf}, which is optimal in its $\Gamma$-convergence setting, but must rely on the precise behavior of minimizers for \eqref{fraccontdef} around $\Sigma$.
\end{rem}

\begin{rem}\label{rem: uniq storiella}  The combination of Proposition~\ref{prop:Gamma-liminf} and Proposition~\ref{prop:gamma-limsup} shows that our convergence result can be localized, and any reasonable definition of fractional mass inside an open set $\Omega\subset \R^n$, which agrees with~\eqref{fraccontdef} when $\Omega=\R^n$, would converge to the classical area as $s\to 1^-$. 

In the codimension one case, the fractional perimeter inside $\Omega$ of a set $E\subset \R^n$ is usually defined (up to constant factors) by
$$\Per_s(E,\Omega):=\iint_{E\times E^c \setminus \Omega^c \times \Omega^c} \frac{dxdy}{|y-x|^{n+s}}= [\mathbbm{1}_E]_{H^{s/2}(\R^n)}- [\mathbbm{1}_E]_{H^{s/2}(\Omega^c)},$$
where $\mathbbm{1}_E$ denotes the characteristic function of the set $E$ (see \cite{CRS}).

In our case, a few different ways to define the local fractional mass exist. For example, one could just take~\eqref{fraccontdef} and replace $\R^n\times \R^n$ with $\R^n\times \R^n\setminus\Omega^c\times\Omega^c$.

Alternatively, one could take a minimizer for~\eqref{fraccontdef} and compute its fractional energy on $\R^n\times \R^n\setminus\Omega^c\times\Omega^c$. This notion could, in principle, depend on the choice of the minimizer. Nevertheless, it would be the most natural notion if the minimizer for~\eqref{fraccontdef} turned out to be unique (up to rotations in the target $\Sp^1$). Indeed, in this case this definition would be completely analogous to the one of the localized fractional perimeter above, with the unique minimizer playing the role of the characteristic function.
    
\end{rem}

\begin{rem}
It would be possible to define the $s$-mass as the minimum of the fractional Sobolev energy over the class
$$ \mathfrak{F}_s ^J(\Sigma):=\biggl\{u:\R^n\to \Sp^1 : \mathcal{E}_{\frac{1+s}{2}}(u,\R^n)<+\infty, \text{ and } {\star Ju}= \pi \Sigma \biggr\},$$
where $Ju$ is the Jacobian of $u$ (see \cite[Section~8.1]{Bre-Mi}) and $\star Ju$ is the $(n-2)$-current associated to the Jacobian (see \cite[Remark~3.3]{ABO2-singularities}).

This, in principle, might lead to a different notion of fractional mass, since the inclusion $ \mathfrak{F}_{\hspace{-1pt} s}^{\hspace{.9pt} w}  (\Sigma) \subseteq \mathfrak{F}_s ^J(\Sigma)$ is standard, but the opposite inclusion seems to be more delicate. In the present work, we chose to stick to the definition proposed in \cite{SerraSurv}, so we do not address this issue.

However, our Gamma-convergence result applies also to this alternative notion, because every map $u\in\mathfrak{F}^J _s(\Sigma)$ can be approximated by a map $u'\in\mathfrak{F}_s(\Sigma')$, with a possibly different singular set $\Sigma'$, in such a way that both the distance in the fractional Sobolev space between $u$ and $u'$ and the flat distance between $\Sigma$ and $\Sigma'$ are small (in any fixed bounded set). This approximation allows to reduce the liminf inequality to the setting of the present work.
\end{rem}



\subsection{Motivation and related works}

One of the central questions in modern Riemannian geometry is the existence and regularity of minimal surfaces in various codimension. 

\vsp 
Nevertheless, with the exception of the case of geodesics on surfaces that can be treated by parametrization effectively, the existence of such objects turned out to be a challenging problem since regular surfaces (of any codimension) with bounded area is a very non-compact set. To overcome this difficulty, one must either consider weak topologies on larger spaces or approximate the area functional with a better-behaved functional, and recover minimal surfaces after a limit procedure of this relaxation. This approach of relaxing the functional has been successful in many great works such as \cite{Riv-Viscosity,  Gua, GasGua, ChoMant, Stern2018, PigatiStern}. 

\vsp 
In codimension one, many different relaxations of the area functional were successful in the construction of minimal hypersurfaces in $\R^n$ or in ambient Riemannian manifolds. For example, building on the seminal works by Modica-Mortola \cite{MM1, MM2} and Modica \cite{Mod87} on the $\Gamma$-convergence of the Allen-Cahn energy to the classical perimeter, many works succeeded in proving the existence of (two-sided) minimal hypersurfaces in Riemannian manifolds. 

\vsp
The Allen-Cahn energy, provides a PDE-based alternative to the classical min-max theory of existence by Almgren-Pitts \cite{Alm1, Alm2, Pitts81}. Using this approach, together with some profound results on the regularity of stable solutions of the Allen-Cahn equation by Tonegawa-Wickramasekera \cite{TonWick}, Gaspar-Guaraco \cite{GasGua} and Guaraco \cite{Gua} managed to prove the existence of smooth embedded minimal hypersurfaces in ambient Riemannian manifolds. The effectiveness of this approach culminated in the work by Chodosh-Mantoulidis \cite{ChoMant} where the authors prove the multiplicity-one conjecture and, with it, give a new proof of Yau's conjecture for generic metrics, that was recently shown by Irie-Marques-Neves in \cite{IMN18}. 

\vsp
Very recently, a new possible approach to the approximation of minimal surfaces appeared: the one of fractional minimal surfaces. Since their first precise definition and study in the seminal work \cite{CRS} by Caffarelli, Roquejoffre and Savin, much interest has been devoted to their study and the study of the fractional perimeter in general. These objects, at first sight, less ``natural'' than other approximations of minimal surfaces, enjoy a long list of properties utterly analogous to the ones for classical minimal surfaces. We refer to \cite{Serena-comparison, SerraSurv} for a complete discussion of the similarities with the classical world. 

\vsp
Apart from a long list of similarities, the literature has recently pointed out a few striking differences with the world of classical minimal surfaces. 

\vsp
A first instance of these differences appeared in a work by Cinti-Serra-Valdinoci \cite{QuantFlatBV}. In this paper, the authors prove that stable fractional minimal surfaces (say, in a ball of radius one) enjoy an interior uniform bound on the classical perimeter of the surfaces. This is in clear contrast with classical minimal surfaces since many parallel hyperplanes (arbitrarily close) are a stable configuration for the perimeter with, clearly, no uniform bound on the area in any ball.

\vsp
This feature has an analog for the fractional Allen-Cahn equation proved in \cite{CCS}, where the authors show that stable solutions of the fractional Allen-Cahn equation enjoy a uniform $BV$-estimate. 

\vsp
Another remarkable result for the fractional Allen-Cahn equation with no parallel in the classical world is the improvement of flatness theorem proved by Dipierro-Serra-Valdinoci in \cite{ImpFlatNonlocal}. In this work, the authors show that entire solutions to the fractional Allen-Cahn equation with asymptotically flat level sets (say, in Hausdorff distance) are one-dimensional. This is in contrast with the classical Allen-Cahn equation since the solutions constructed in \cite{dPKW-Esol} concentrate on catenoids, which have arbitrary flat blow-downs in Hausdorff distance but are clearly not one-dimensional. 

\vsp
This feature represents a remarkable departure from the classical world and is in accordance with the fact that nonlocal catenoids have conical, nontrivial blow-downs. This is to say the blow-down of the nonlocal catenoids is not a single plane with multiplicity two. These nonlocal catenoids were constructed in \cite{JdPW}, where the authors prove that in $\R^3$ there exists a connected, embedded $s$-minimal surface of revolution whose blow-down is a nontrivial cone. Moreover, in the same work \cite{JdPW} it is also proved that these nonlocal catenoids have infinite Morse index, in any reasonable sense. 

\vsp
These features suggest that sequences of $s$-minimal surfaces with uniformly bounded index should enjoy stronger compactness properties than classical minimal surfaces with bounded index. This turns out to be true. Indeed, in the recent paper \cite{CFS23}, the first author, Florit and Serra, obtain a uniform $ BV$-estimate for finite index solutions of the fractional Allen-Cahn equation, extending the one of \cite{CCS} for stable solutions to the case of finite index. 

\vsp
These surprisingly strong estimates for finite Morse index nonlocal minimal surfaces confer exceptional compactness and regularity properties to these objects. Thanks to this feature and a classical min-max method, the authors in \cite{CFS23} establish far-reaching existence result of infinitely many $s$-minimal surfaces in every $n$-dimensional closed Riemannian manifold. In the same work, they are also proved to be smooth for $n=3,4$ and to have a small singular set for $n\ge 5$. This work suggests that $s$-minimal surfaces are an ideal class of objects on which to apply min-max methods, as they seem to prevent almost every pathology that arises for classical minimal surfaces (as multiplicity and pinching). 

\vsp
From here, a natural question arises: Can one take advantage of these exceptional compactness and regularity properties of $s$-minimal surfaces and send $s\to 1^-$ afterward to recover classical minimal surfaces? Surprisingly, the answer seems to be affirmative. 

\vsp
In the recent work \cite{CDSV}, Chan, Dipierro, Serra, and Valdinoci obtain robust curvature estimates and optimal sheet separation estimates, as $s\to 1^- $, for stable $s$-minimal surfaces in three dimensions. This result allows to send $s\to 1^- $ and to pass to the limit \textit{the supports} of the $s$-minimal surfaces to obtain standard minimal surfaces. Indeed, very recently, building on \cite{CDSV}, Florit-Simon in \cite{Enric24} was able to extend their result to finite index $s$-minimal surfaces and recover the classical Yau's conjecture for generic metrics sending $s\to 1^-$ for the $s$-minimal surfaces constructed in \cite{CFS23}. 

\vsp
In codimension two, the goal of finding a diffuse approximation of the codimension two area that is well-behaved in the limit is an extremely challenging problem. For example, it is known that with the classical complex valued Ginzburg-Landau model, one can produce nontrivial, stationary $(n-2)$-dimensional varifolds in ambient Riemannian manifolds. Nevertheless, it has been shown in \cite{Pigati-Stern-nonquant} that, in general, these varifolds are not integral and that every density $\theta \in \{1\} \cup [2, \infty)$ on an $(n-2)$-plane can be realized as a limit of complex Ginzburg–Landau critical points. 

\vsp
In the groundbreaking work \cite{PigatiStern} by Pigati and Stern, the authors substitute the complex Ginzburg-Landau energy with the Yang-Mills–Higgs energy and can produce stationary, integral $(n-2)$-varifolds in the limit. As a byproduct, in codimension two, they obtain a new proof of Almgren's existence result of nontrivial, stationary integral $(n-2)$-dimensional varifolds in closed Riemannian manifolds. More recently, in \cite{PPS-CPAM-Gammaconv} it is shown that the Yang-Mills–Higgs energy $\Gamma$-converges to the classical area.

\vsp
At this point, it is natural to ask if there is a natural codimension two analog of the fractional perimeter that reflects similar properties for stationary objects as the ones recently discovered for $s$-minimal surfaces. To the authors' knowledge, tentative notions of codimension $\ge 2$ fractional masses appeared just in two very recent preprints: \cite{FM1} for codimension $(n-1)$ and \cite{FM2} in any codimension. In the same works, both notions are proved to converge pointwise, as $s\to 1^-$, to the standard Hausdorff measure of the correct dimension. 

\vsp 
The convergence of the fractional area to the classical one in the limit $s\to 1^-$ is indeed the first basic property that any possible notion need to satisfy if one wants to exploit the fractional approximation to obtain results on classical minimal surface. In the case of the codimension one fractional perimeter, the pointwise convergence is just a particular case of the so-called BBM formula (see \cite{BBM01,Davila-BBM}). It has also been showed in \cite{Gammaconv} (see also \cite{Ponce-calcvar}) that the fractional perimeter converges to the classical perimeter also in the sense of $\Gamma$-convergence with respect to the $L^1$ convergence of sets (that basically corresponds to the flat convergence of their boundaries).

\vsp
Recently, Serra in \cite[Section 5]{SerraSurv} suggested a notion of $s$-fractional mass for codimension $k \in \{1, 2, \dotsc , n-1\}$ smooth, multiplicity one submanifolds of $\R^n$ that are level sets of (regular values of) maps from $\R^n$ to $\R^k$. 

\vsp
In this work, we study the case of codimension two in detail and full generality. We prove that this notion is well-defined for closed, oriented (not necessarily connected) codimension two surfaces of locally constant multiplicity. Moreover, we prove its $\Gamma$-convergence with respect to the
flat topology, as $s\to 1^- $, to the $(n-2)$-dimensional Hausdorff measure with multiplicity. 

\vsp
Our study is robust and suited to be extended naturally to ambient Riemannian manifolds and every codimension. We refer to Section \ref{sec: poss generalizations} for more details on these generalizations. In particular, on Riemannian manifolds, our notion is well-defined for $(n-2)$-dimensional oriented boundaries with integer multiplicity, and this is the natural class to apply min-max methods on ambient Riemannian manifolds. 

\vsp 
Compared to their local counterparts, we expect stationary sets for this notion to exhibit enhanced regularity and compactness properties, much like the codimension one scenario of the $s$-minimal surfaces described above. 

\subsection{Future directions}\label{subs:future}

This work focuses mainly on the definition and study of the fractional $s$-mass for codimension two surfaces in $\R^n$ and its convergence to the standard $\mathcal{H}^{n-2}$ Hausdorff measure. 

Because of this focus, a few basic questions related but not strictly needed to prove this convergence were left unanswered. In particular, we expect to explore the following few (related) questions regarding the minimization in \eqref{fraccontdef} and the characterization of the minimizer, whose structure is left almost entirely open in this work. 

\begin{itemize}

    \item[\textit{(P1)}] Prove that there is no energy gap taking the infimum on $\mathfrak{F}_{\hspace{-.7pt} s}(\Sigma) $ instead of $ \mathfrak{F}_{\hspace{-1pt} s}^{\hspace{.9pt} w}  (\Sigma)$, that is
    \begin{equation*}
        \inf_{u \in  \mathfrak{F}_s  (\Sigma)}\mathcal{E}_{\frac{1+s}{2}}(u,\R^n) = \min_{u \in  \mathfrak{F}_{\hspace{-1pt} s}^{\hspace{.9pt} w}  (\Sigma)} \mathcal{E}_{\frac{1+s}{2}}(u,\R^n).
    \end{equation*}
    
    \item[\textit{(P2)}] Prove that every minimizer $u_\circ $ in \eqref{fraccontdef} is smooth in $\R^n\setminus \Sigma$ and therefore $u_\circ \in \mathfrak{F}_s  (\Sigma)$. 
     
\end{itemize}

Moreover, on an ambient Riemannian manifold the $s$-mass $\M_{s}$ is well-defined on $(n-2)$-dimensional boundaries with integer coefficients (see Subsection \ref{sbs ambient Riem} on the generalization of our notion to Riemannian manifolds), hence it is natural to wonder whether many critical points of $\M_{s}$ exist. This question was already posed by Serra in \cite[Section 5, Open problem 3]{SerraSurv} in general. 

\vsp 
In particular, since $H_{n-2}(\Sp^n)=0$ for all $n\ge 3$, on Riemannian spheres $(\Sp^n,g)$ there holds that every $(n-2)$-dimensional integral cycle is an $(n-2)$-dimensional integral boundary. Moreover, see for example the end of Section 6 in \cite{MNcycles}, the space of $(n-2)$-dimensional integral cycles endowed with the flat topology is weakly homotopic to $\C \mathbb{P}^\infty$, and $\C \mathbb{P}^\infty$ contains a nontrivial class on every even dimensional homology group. Hence, in the space of $(n-2)$-dimensional integral cycles, one could perform min-max for $\M_s$ over these infinitely many nontrivial families. This suggests that the answer to the following open question should be affirmative. 
\begin{itemize}
    \item[\textit{(Q1)}] For $n\ge 3$, do there exist infinitely critical points of $\M_{s}$ on every sphere $(\Sp^n,g)$?
    \end{itemize}

Lastly, other than existence of critical points, it is natural to wonder whether local minimizers of $\M_s$ are regular, outside a singular set of suitable dimension. Precisely
    \begin{itemize}
        \item[\textit{(Q2)}] If $\Sigma$ is a local minimizer (in some suitable class) of $\M_{s}$, is $\Sigma$ smooth/analytic outside a singular set at most $(n-4)$-dimensional? 
    \end{itemize}

\bigskip\noindent
\textbf{Organization.} In Section \ref{sec:Preliminaries}, we deal with all the preliminary tools that we need to define the fractional mass $\M_{s}$. These include the linking number, fractional Sobolev spaces, and the notion of flat convergence.

Section \ref{sec 3} deals with the existence of a competitor in $\mathfrak{F}_{\hspace{-.7pt} s}(\Sigma) $ for the minimum problem \eqref{fraccontdef} (see Theorem~\ref{comp existence}), the existence of a minimizer for this problem, and some properties of this minimizer. In particular, we prove that any minimizer still links with the surface $\Sigma$ in a suitable weak sense, and we provide a lower bound for the energy for fixed $s\in (0,1)$.

Then, Section \ref{sec:main-results} deals with the most technical part of this work and the proof of the $\Gamma$-convergence result of Theorem \ref{energy conv teo}. 

Lastly, in Section \ref{sec: poss generalizations}, we briefly describe possible generalizations (and variants) of our notion to higher codimension or to ambient Riemannian manifolds.

\section{Definitions and notations}\label{sec:Preliminaries}

\subsection{Linking number between submanifolds}\label{sbs linking}

Let $n\geq 3$ and $n_1,n_2\geq 1$ be integers with $n_1+n_2=n-1$. Let $M_1$ and $M_2$ be disjoint, closed, oriented, connected $C^2$ submanifolds of $\R^n$, with dimensions $n_1$ and $n_2$. We denote by $H_k(X)$ and $H^k(X)$ respectively the $k$-th singular homology and cohomology groups of the space $X$, with integer coefficients.

There exist several equivalent definitions of linking number between $M_1$ and $M_2$, and each of them enlightens different properties of this quantity that we need in the sequel of the paper. For this reason, we now consider three of these definitions.

\begin{defn}\label{def:link_G}
Let $M_1$ and $M_2$ be as above. The Gaussian linking number $\link_G(M_1,M_2)$ is the degree of the map $G:M_1\times M_2 \to \Sp^{n-1}$ defined by
$$G(x,y):=\frac{x-y}{|x-y|}.$$
\end{defn}

\begin{defn}\label{def:link_I}
Let $M_1$ and $M_2$ be as above, and let us assume that there exists an oriented $(n_1+1)$-submanifold $N_1\subset \R^n$ with $\partial N_1 =M_1$ intersecting $M_2$ transversally.
Then the intersection linking number is defined as
$$\link_I (M_1,M_2):=\sum_{p\in N_1\cap M_2} s(p),$$
where $s(p)=1$ if the orientation on $T_p N_1 \oplus T_p M_2=\R^n$ induced by the orientations on $N_1$ and $M_2$ agrees with the orientation of $\R^n$, and $s(p)=-1$ otherwise.
\end{defn}

\begin{defn}\label{def:link_H}
Let $M_1$ and $M_2$ be as above. Let $[M_1]\in H_{n_1}(M_1)$ be the fundamental class of $M_1$ and let $[M_1]^* \in H^{n-n_1-1}(\R^n \setminus M_1)=H^{n_2}(\R^n \setminus M_1)$ be its Alexander dual. Let also $[M_2]$ be (the push-forward of) the fundamental class of $M_2$ in $H_{n_2}(\R^n \setminus M_1)$.

Then the homological linking number is defined as
$$\link_H (M_1,M_2):=[M_1]^* ([M_2]).$$
\end{defn}

The following proposition shows that these definitions actually agree up to a sign.

\begin{prop}\label{prop:link}
Let $M_1$ and $M_2$ be as above, and let us assume that there exists $N_1$ as in Definition~\ref{def:link_I}. Then it turns out that
\begin{equation}\label{equiv_link}
\link_G(M_1,M_2)=
(-1)^{n_2} \link_I(M_1,M_2)=
(-1)^{n_2} \link_H(M_1,M_2).
\end{equation}
\end{prop}

\begin{proof}
The first equality is basically proved in \cite[Section~2.11]{ABO2-singularities}. The idea consists in considering a function $F:N_1\times M_2 \to \R^n$ such that $F(x,y)=(x-y)/\abs{x-y}$ when $(x,y)\in M_1\times M_2$, with $F(x,y)=x-y$ when $|x-y|$ is small, and no other zeros outside the diagonal $\{x=y\}$. Then it turns out that
$$\link_G(M_1,M_s)=\deg(F,M_1\times M_2)=\deg(F,N_1\times M_2, \R^n , 0),$$
where the expression in the right-hand side is the degree of $F$ as a map defined on the manifold with boundary $N_1\times M_2$ with values in $\R^n$, computed with respect to the origin. By the construction of $F$, this coincides with $(-1)^{n_2}\link_I(M_1,M_2)$, because $F^{-1}(0)=N_1\cap M_2$ and $\sign (\det \mathrm{d} F(p))=(-1)^{n_2} s(p)$ for every $p\in N_1\cap M_2$.

The second equality can be proved by analyzing the construction of the Alexander duality. Indeed, if $\nu M_1$ is a tubular neighborhood of $M_1$ that does not intersect $M_2$, the Alexander duality can be obtained by composing the following isomorphisms
$$H_{n_1}(M_1)\simeq H_{n_1+1}(\R^n,M_1)\simeq H_{n_1 +1}(\R^n \setminus \nu M_1, \partial (\R^n \setminus \nu M_1))\simeq H^{n_2}(\R^n \setminus M_1).$$

The first isomorphism is given by the inverse of the boundary operator in the long exact sequence of the pair $(\R^n,M_1)$, so it sends $[M_1]$ to $[N_1]$. The second isomorphism is given by the excision theorem, together with some deformation retraction, and is just restricting $[N_1]$ to $\R^n\setminus \nu M_1$. Finally, the last isomorphism is given by the Poincaré-Lefschetz duality (to be precise, one should apply it in $\Sp^n$), again together with a deformation.

This means that $[M_1]^*$ is the Poincaré-Lefschetz dual of $[N_1]$, hence its action on $[M_2]$ is given by the algebraic intersection number (see \cite[Section~VI.11]{Bredon-book}), which is exactly $\link_I(M_1,M_2)$.
\end{proof}

This result allows us to define the linking number $\link(M_1,M_2)$ as any of the three equal quantities appearing in~\eqref{equiv_link}.

The equivalence of the previous definitions shows that the linking number satisfies several useful properties. For example, Definition~\ref{def:link_G} shows that the linking number is commutative, up to a sign, while from Definition~\ref{def:link_H} it follows that $\link(M_1,M_2)=\link(M_1,M_2')$ if $[M_2]=[M_2']$ in $H_{n-2}(\R^n\setminus M_1)$. However, Definition~\ref{def:link_I} provides the simpler way to compute the linking number in concrete situations.

\begin{rem}
Actually, the linking number can be defined for much more general couples of disjoint sets and, since the algebraic intersection number can be defined also for singular cycles, in Proposition~\ref{prop:link} it is not necessary to assume the existence of a regular submanifold $N_1$ with $\partial N_1=M_1$.

In particular, the linking number can be defined also if one of the two submanifolds is a bi-Lipschitz curve in the class $L(\Sigma)$ defined in \eqref{defn:L(Sigma)}. In this paper, however, the only non-smooth objects that we consider are boudaries of two-dimensional squares, in which case extending the three definitions of the linking number and Proposition~\ref{prop:link} is straightforward.
\end{rem}

\begin{rem}
We observe that the choice of the orientations on $M_1$ and $M_2$ affects the linking number at most by a sign, and more precisely $$\link(M_1,M_2)=-\link(-M_1,M_2)=-\link(M_1,-M_2)=\link(-M_1,-M_2),$$
where $-M_1$ and $-M_2$ denote, respectively, the manifolds $M_1$ and $M_2$ with opposite orientation.

As a consequence, if $\Sigma$ is as in \eqref{defn:Sigma} and $\gamma\in L(\Sigma)$, then
    \[
    \abs{d_1 \link(\gamma,\Sigma_1)+\dots + d_m \link(\gamma,\Sigma_m)}
    \]
is independent of the choice of the orientation for $\gamma$.

Similarly, we point out that if $u:\R^n\setminus \Sigma \to \Sp^1$ is continuous, then $\abs{\deg(u,\gamma)}$ is independent of the choice of an orientation (or a parameterization) for $\gamma$, hence the class \eqref{Flink} is well-defined.
\end{rem}

\subsection{Fractional Sobolev spaces}
We briefly recall here the function spaces we will work with. 

Let $\Omega \subset \R^n$ be an open set. For $\alpha\in (0,1)$ we denote by 
 \begin{align*}
        H^\alpha(\Omega; \R^2) = \left \{ u \in L^2(\Omega; \R^2) \, : \, [u]^2_{H^\alpha(\Omega)} = \iint_{\Omega \times \Omega} \frac{|u(x)-u(y)|^2}{|x-y|^{n+2\alpha}} \, dxdy <+\infty \right \} ,
    \end{align*}
    the Sobolev-Slobodeckij space of order $s$. This is a separable Hilbert space endowed with the norm $ \|u\|^2_{H^\alpha(\Omega)} := \| u \|_{L^2(\Omega)}^2+ [u]^2_{H^\alpha(\Omega)}$.
We also set
\begin{equation*}
    H^{\alpha} (\Omega; \Sp^{1}) := \Bigl \{ u \in H^{\alpha} (\Omega; \R^2) \, : \, |u(x)|=1 \,\,\, \text{for almost every } x\in \Omega \Bigr \} ,
\end{equation*}

As usual, we denote by $H^{\alpha}_{\mathrm{loc}}(\R^n;\Sp^{1})$ the space of functions $u\in L^2 _{\mathrm{loc}}(\R^n;\Sp^{1})$ that belong to $H^{\alpha}(B_R,\Sp^{1})$ for every ball $B_R \subset \R^n$ with radius $R>0$ centered in the origin.

\vsp 
For every open set $\Omega\subset \R^n$ with Lipschitz (or empty) boundary, let us consider the following space of functions
    \[
        \mathbb{X}_s(\Omega):=\biggl\{u : \Omega\to \Sp^1 : u\in H^\alpha(\Omega \cap B_R) \text{ for every }R>0 \text{ and every }\alpha\in \biggl(0,\frac{1+s}{2}\biggr) \biggr\},
    \]
endowed with the distance
    \[
        \mathrm{d}_{\mathbb{X}_s(\Omega)}(u_1,u_2):=\sum_{k=2} ^{+\infty} 2^{-k} \min\Bigl\{\|u_1-u_2\|_{H^{\frac{1+s}{2}-\frac{1}{k}}(\Omega\cap B_k)},1\Bigr\}.
    \]

Since the inclusion $H^{\alpha}(\Omega\cap B_R;\Sp^1)\hookrightarrow H^{\alpha'} (\Omega\cap B_R;\Sp^1)$ is continuous (and actually compact, see e.g.~\cite[Theorem~8.5]{Leoni-book}) for every $R>0$ and every $\alpha'<\alpha$, we deduce that a sequence $\{u_k\}\subset \mathbb{X}_s(\Omega)$ converges to some $u\in \mathbb{X}_s(\Omega)$ with respect to this distance if and only if $u_k\to u$ in $H^{\alpha}(\Omega\cap B_R;\Sp^1)$ for every $R>0$ and every $\alpha \in (0,(1+s)/2)$. As a consequence, the completeness of the spaces $H^{\alpha} (\Omega \cap B_R;\Sp^1)$ implies that also $(\mathbb{X}_s(\Omega),\mathrm{d}_{\mathbb{X}_s(\Omega)})$ is a complete metric space.

\begin{rem}\label{rem: Xs closure }
    We observe that the class $\mathfrak{F}_{\hspace{-1pt} s}^{\hspace{.9pt} w} (\Sigma)$ defined in~\eqref{defn:Fweak} is actually the closure in $\mathbb{X}_s(\R^n)$ of the class $\mathfrak{F}_{\hspace{-.7pt} s}(\Sigma)$ defined in~\eqref{Flink}. Moreover, if a sequence $\{u_k\}\subset H^{\frac{1+s}{2}}_{\mathrm{loc}}(\R^n;\Sp^1)$ converges weakly in $H^{\frac{1+s}{2}}_{\mathrm{loc}}(\R^n;\Sp^1)$, then it also converges in the sense of $\mathbb{X}_s(\R^n)$. In particular, any sequence of functions $\{u_k\}\subset \mathbb{X}_s(\R^n)$ for which
$$\sup_{k\in\N} \mathcal{E}_{\frac{1+s}{2}}(u_k,\R^n)<+\infty,$$
is relatively compact in $\mathbb{X}_s(\R^n)$. This follows by a standard weak compactness (see, for example, \cite[Remark 2.2]{Partialreg}) and diagonal argument. However, the weak topology is not metrizable, and this is the reason for which it is convenient to consider the spaces $\mathbb{X}_s$.
\end{rem}

We conclude this subsection by recalling a few results about fractional Sobolev spaces. While all these facts are well-known, we need the explicit dependence on the parameter $s$ of all the constants, so here state them in the form that we need, and we also include short proofs.

The first result is the following version of the fractional Sobolev embedding in the one-dimensional case.

\begin{lemma}[Fractional Sobolev embedding]\label{lemma:frac_embedd} 
For every $s_0\in (0,1)$ there exists a positive constant $C_S=C_S(s_0)>0$ such that for every $s\in [s_0,1)$, every interval $(a,b)\subset \R$, and every function $u:(a,b)\to \R^2$ it holds that
$$ \osc(u,(a,b))^2\leq C_S (b-a)^{s} (1-s) [u]_{H^{\frac{1+s}{2}}((a,b))} ^{2},$$
where $\osc(u,(a,b))$ denotes the (essential) oscillation of $u$ in $(a,b)$, namely
$$\osc(u,(a,b)):=\||u(y)-u(x)|\|_{L^\infty((a,b)\times (a,b))}.$$
In particular, if $[u]^2 _{H^{\frac{1+s}{2}}((a,b))}<+\infty$, then $u$ agrees almost everywhere with a continuous function.
\end{lemma}

\begin{proof}
Let us consider first the case in which $(a,b)=(0,1)$ and $u$ is scalar-valued, and let us assume by contradiction that there exists $s_0\in (0,1)$ such that for every $k\in \N^+$ there exist $s_k\in [s_0,1)$ and a function $u_k:(0,1)\to \R$ such that
\begin{equation}\label{eq:absurd_hp}
\osc(u_k,(0,1))=1 \qquad\text{and}\qquad (1-s_k) [u_k]_{H^{\frac{1+s_k}{2}}((0,1))} ^{2} <\frac{1}{k}.
\end{equation}

Up to a translation, we can also assume that $\{u_k\}$ is bounded in $L^\infty((0,1))$.

From \cite[Remark~5]{BBM01} we deduce that there exists a universal constant $C>0$ such that
$$\lim_{k\to +\infty} [u_k]_{H^{\frac{1+s_0}{2}}((0,1))} ^2 \leq  \lim_{k\to +\infty} C \cdot \frac{1-s_k}{2}\cdot [u_k]_{H^{\frac{1+s_k}{2}}((0,1))} ^2 =0,$$
and hence \cite[Theorem~7.1]{HitGuide} implies that, up to a subsequence, $\{u_k\}$ converges to some constant function in $H^{\frac{1+s_0}{2}}((0,1))$.

As a consequence, from \cite[Theorem~8.2]{HitGuide} we deduce that $\{u_k\}$ converges to a constant also in $C^{0,\alpha}((0,1))$, with $\alpha=s_0/2$, but this contradicts the first condition in~\eqref{eq:absurd_hp}.

At this point, the case of a general interval $(a,b)$ follows just by scaling, and the vector-valued case just by applying the inequality component-wise.
\end{proof}

The second result is a version of the Poincaré inequality. This type of results are well-known, even in a wider generality than what we need for our purposes, see e.g.~\cite[Theorem~1]{BBM02} or~\cite[Theorem~1.1]{Ponce}. However, in our case, the proof is simpler and with the same strategy used in the previous lemma (with the $L^2$-norm instead of the oscillation) we can prove the following result as well.

\begin{lemma}[Poincaré inequality]\label{lemma:Poincaré}
Let $\alpha_0 \in (0,1)$ and let $B_1 \subset \R^n$ be the unit ball. Then, there exists a constant $C_P=C_P(\alpha_0, n) > 0$ such that for every $\alpha \in [\alpha_0,1)$ and for every $u \in H^\alpha(B_1,\R^2)$ it holds that
    \begin{equation}\label{eq:poincaré}
        \norm*{u - \bar{u}}^2_{L^2(B_1)} \le C_P (1-\alpha) [u]_{H^\alpha(B_1)}^2,
    \end{equation}
where $\bar{u} :=  \fint_{B} u$ is the average of $u$ in $B_1$.
\end{lemma}

Finally, we need the following extension result.

\begin{lemma}[Extension theorem]\label{lemma:extension}
Let $\alpha_0\in (0,1)$ and let $B_R \subset \R^n$ be the ball of radius $R$. Then there exists a constant $C_E=C_E(\alpha_0,n)>0$, depending only on $\alpha_0$ and $n$, such that for every $\alpha\in [\alpha_0,1)$ and every $u\in H^{\alpha}(B_R;\R^2)$ with $\fint_{B_R} u=0$ there exists a function $ E(u) \in H^{\alpha}(\R^n;\R^2)$ with $E(u) = u$ in $ B_R$ and
    \begin{equation} \label{est:extension}
        [ E(u) ]_{H^\alpha(\R^n)} ^2\leq C_E [u]_{H^\alpha(B_R)} ^2.
    \end{equation}
\end{lemma}

\begin{proof}
    We start with the case $R=1$. A careful reading of \cite[Section~5]{HitGuide} shows that there exists a constant $C(n) > 0$ such for any $u \in H^\alpha(B_1)$ there exists $E(u) \in H^\alpha(\R^n)$ such that $E(u)=u$ in $B_1$ and
    \begin{equation}\label{est:extension_s}
    [E(u)]^2_{H^\alpha(\R^n)} \le C(n) \Big( \alpha^{-1}[u]^2_{H^\alpha(B_1)} + \alpha^{-1}(1-\alpha)^{-1} \norm*{u}_{L^2(B_1)}^2 \Big).
    \end{equation}

    Indeed, the constant appearing in \cite[Theorem~5.4]{HitGuide} depends on the domain $\Omega$ (via a covering, a partition of unity and the bi-Lipschitz maps locally parameterizing the boundary), and on the constants coming from \cite[Lemma~5.1 and Lemma~5.3]{HitGuide}.

    Now we observe that the constant $C(n,\alpha,p,K,\Omega)$ in \cite[Lemma~5.1]{HitGuide} could actually be written as $C(n,p,K,\Omega)\alpha^{-1}$, while in \cite[Lemma~5.3]{HitGuide} the dependence of the constant on $\alpha$ appears only in the estimate
    $$\int_{\Omega}dx \int_{\Omega\cap |x-y|<1} \frac{|u(x)|^p}{|y-x|^{n+(\alpha-1)p}} + \int_{\Omega}dx \int_{\Omega\cap |x-y|>1} \frac{|u(x)|^p}{|y-x|^{n+\alpha p}}\leq C(n,p,\alpha)\|u\|_{L^p(\Omega)}^p,$$
    and in this case $C(n,p,\alpha)$ could be written as $C(n,p)\alpha^{-1}(1-\alpha)^{-1}$, and it multiplies only the $L^p$ part of the norm. Hence the estimate (5.6) in \cite{HitGuide} in the case $p=2$ could actually be written as
    $$[\psi u]_{H^\alpha(\Omega)}\leq C(n,\Omega,\psi) \Big( [u]^2_{H^\alpha(\Omega)} + \alpha^{-1}(1-\alpha)^{-1} \norm*{u}_{L^2(\Omega)}^2 \Big).$$
    
  In the end, the arguments in \cite[Section~5]{HitGuide}, in the case $p=2$ and $\Omega=B_1$, actually yield~\eqref{est:extension_s}. Then, since $u$ has zero average by hypothesis, the Poincaré inequality \eqref{eq:poincaré} yields \eqref{est:extension} as desired. 
       
    Finally, reasoning component-wise we obtain the same result also when $u$ takes values in $\R^2$, and, by scaling, we also deduce that the same holds for every $R>0$.
\end{proof}

\subsection{Slicing with subspaces}\label{sbs: ONB}

We denote by $\mathcal{G}_k^n$ the Grassmanian of $k$-dimensional linear subspaces in $\R^n$. This comes naturally endowed with a unit measure invariant under the action of the orthogonal group $O(n)$.

In the case $k=1$, it can be easily proved using polar coordinates that
\begin{eqnarray*}
    \iint_{\R^n \times \R^n} f(x,y) \, dxdy&=& \frac{1}{2} \int_{\Sp^{n-1}} d\theta \int_{\theta^\perp} d\mathcal{H}^{n-1}(h) \iint_{\R \times \R} f(h+\theta x,h+\theta y)|x-y|^{n-1} \, dxdy \\
    &=& \frac{\omega_{n-1}}{2} \int_{\mathcal{G}^n_1} dL \int_{L^\perp} d\mathcal{H}^{n-1}(h) \iint_{(L+h) \times (L+h)} f(x,y)|x-y|^{n-1} \, dxdy ,
\end{eqnarray*}
for every measurable function $f$. In particular, choosing $f$ of the form
\begin{equation*}
    f(x,y)=\frac{|u(x)-u(y)|^2}{|x-y|^{n+2\alpha}} \mathbbm{1}_{\Omega \times \Omega}(x,y)
\end{equation*}
gives
\begin{align}
    [u]^2_{H^\alpha(\Omega)} &= \iint_{\Omega \times \Omega} \frac{|u(x)-u(y)|^2}{|x-y|^{n+2\alpha}} dxdy \nonumber \\ & = \frac{\omega_{n-1}}{2} \int_{\mathcal{G}^n_1} dL \int_{L^\perp} d\mathcal{H}^{n-1}(h) \iint_{ \big( (L+h)\cap \Omega \big)^2  }\frac{|u(x)-u(y)|^2}{|x-y|^{1+2\alpha}} \, dxdy \nonumber  \\ & = \frac{\omega_{n-1}}{2} \int_{\mathcal{G}^n_1} dL \int_{L^\perp}   [u]^2_{H^\alpha((L+h)\cap \Omega)} \, d\mathcal{H}^{n-1}(h) . \label{eq: slic for}
\end{align}

In the case in which $n=2$ and $\Omega=\mathrm{D}_r(x)$ is a disk, from this formula and the invariance by rotations of the measure on $\mathcal{G}^n_1\simeq\Sp^1$ we deduce that
\begin{align}
& [u]_{H^{\alpha}(\mathrm{D}_r(x))} ^2 = \frac{1}{2}\int_{\Sp^1}d\theta \int_{-r} ^{r} [u]_{H^{\alpha}(S_{\theta,t}(x))} ^2 \,dt \nonumber\\
& \s  =\frac{1}{4}\int_{\Sp^1}d\theta \int_{0} ^{r}\Bigl( [u]_{H^{\alpha}(S_{\theta,t}(x))} ^2 +[u]_{H^{\alpha}(S_{\theta^\perp,t}(x))} ^2+[u]_{H^{\alpha}(S_{\theta,-t}(x))} ^2+[u]_{H^{\alpha}(S_{\theta^\perp,-t}(x))} ^2\Bigr)\, dt,\label{eq:planar_slicing_squares}
\end{align}
where $\theta^\perp$ denotes the counterclockwise rotation of $\theta$ by an angle equal to $\pi/2$, and the set $S_{\theta,t}(x):=\{x+t\theta^\perp +\xi\theta:\xi\in (-\sqrt{r^2-t^2},\sqrt{r^2-t^2})\}$ is the intersection of the line $t\theta^\perp+\theta \R$ with the disk $\mathrm{D}_r(x)$. We point out that when $t<r/\sqrt{2}$ the four segments in the right-hand side contain the boundary of a square centered at $x$.

\vsp 
The next result of this section is a generalization of (\ref{eq: slic for}) with $k$-dimensional subspaces instead of lines. This is the content of Theorem \ref{fourier slicing} below.
 
Before proving it, we recall a classical fact about expressing Sobolev seminorms in terms of the Fourier transform. We adopt the convention of the normalized Fourier transform, that is
\begin{equation*}
    \mathcal{F}u(\xi) = \frac{1}{(2\pi)^{n/2}} \int_{\R^n} e^{-i x\cdot \xi} u(x) \, dx .
\end{equation*}

\begin{prop}[{\cite[Proposition 3.4]{HitGuide}}] \label{fourier seminorms}
    Let $\alpha\in(0,1)$. Then there exists a positive constant $C_F(n,\alpha)>0$ such that for every $u \in H^\alpha(\R^n)$ it holds that
    \begin{equation*}
        [u]_{H^\alpha(\R^n)}^2 = C_{F}(n,\alpha) \int_{\R^n} |\xi|^{2\alpha} |\mathcal{F}u(\xi)|^2 \, d\xi .
    \end{equation*}
\end{prop}

\begin{teo}\label{fourier slicing}
   Let $\alpha\in(0,1)$ and $k\in \{1,\dots,n-1\}$. Then there exists a constant $C(n,k,\alpha)>0$ such that for every $u \in H^\alpha(\R^n)$ it holds that
   \begin{equation*}
       [u]^2_{H^\alpha(\R^n)} = C(n,k,\alpha) \int_{\mathcal{G}^n_k} dL \int_{L^\perp} [u]^2_{H^\alpha(L+h)} \, d\mathcal{H}^{n-k}(h)  ,
   \end{equation*}
   where 
   \begin{equation*}
        C(n,k,\alpha) = \frac{C_{F}(n,\alpha)}{C_{F}(k,\alpha)} \left( \int_{\mathcal{G}_k^n} |P_L(e_1)|^{2\alpha} dL \right)^{-1} , 
    \end{equation*}
    and $P_L :\R^n \to L$ is the orthogonal projection on the subspace $L$.
\end{teo}
\begin{proof}
    Fix $u \in H^\alpha(\R^n)$ and denote $(x,y)\in \R^n$ with $x\in \R^k$ and $y\in \R^{n-k}$. By Proposition \ref{fourier seminorms} on the $x$-variables of $u$ (that is, with $k$ in place of $n$) we have for every fixed $y\in \R^{n-k}$:
   \begin{equation*}
        [u(\cdot, y)]_{H^\alpha(\R^k)}^2 = C_{F}(k,\alpha) \int_{\R^k} |\xi|^{2\alpha} |\mathcal{F}_x u(\xi,y)|^2 \, d\mathcal{H}^{k}(\xi) .
    \end{equation*}
    Completely analogously, we can do the same for every $L\in \mathcal{G}^n_k$. If we denote by $(x_L , y_L) \in L \times L^\perp = \R^n$ the coordinates on $\R^n$, we get 
\begin{equation*}
        [u(\cdot, y_L)]_{H^\alpha(L)}^2 = C_{F}(k,\alpha) \int_{L} |\xi_L|^{2\alpha} |\mathcal{F}_L u(\xi_L,y_L)|^2 \, d\mathcal{H}^{k}(\xi_L) .
    \end{equation*}
    Integrating this formula over $y_L \in L^\perp$ and using Plancherel's identity on the $y_L$-variables gives
    \begin{align*}
        \int_{L^\perp} [u(\cdot, y_L)]_{H^\alpha(L)}^2 &  \, d\mathcal{H}^{n-k}(y_L) \\ & = C_{F}(k,\alpha)\int_{L^\perp} \int_{L} |\xi_L|^{2\alpha} |\mathcal{F}_L u(\xi_L,y_L)|^2 \, d\mathcal{H}^{k}(\xi_L) d\mathcal{H}^{n-k}(y_L) \\ &= C_{F}(k,\alpha) \int_{L} |\xi_L|^{2\alpha} \left( \int_{L^\perp}  |\mathcal{F}_L u(\xi_L,y_L)|^2 \,  d\mathcal{H}^{n-k}(y_L) \right) d\mathcal{H}^{k}(\xi_L) \\ &= C_{F}(k,\alpha) \int_{L} \int_{L^\perp} |\xi_L|^{2\alpha} \big| (\mathcal{F}_{L^\perp} \circ \mathcal{F}_L ) u(\xi_L, \xi_{L^\perp}) \big|^2 \,  d\mathcal{H}^{n-k}(\xi_{L^\perp}) d\mathcal{H}^{k}(\xi_L) \\ &=  C_{F}(k,\alpha) \int_{\R^n} |P_L(\xi)|^{2\alpha} |\mathcal{F} u |^2 d\mathcal{H}^{n}(\xi) ,
    \end{align*}
    since $P_L(\xi)=\xi_L$.

    Since this formula holds for every $L\in \mathcal{G}_k^n$, we can integrate it over the entire Grassmanian $\mathcal{G}_k^n$ to get
\begin{align*}
    \int_{\mathcal{G}_k^n} dL \int_{L^\perp} [u(\cdot, y_L)]_{H^\alpha(L)}^2 \, d\mathcal{H}^{n-k}(y_L) = C_{F}(k,\alpha) \int_{\R^n} \left(  \int_{\mathcal{G}_k^n} |P_L(\xi)|^{2\alpha}  dL \right) |\mathcal{F} u |^2 d\mathcal{H}^{n}(\xi) .
\end{align*}
Moreover, just by scaling and rotational invariance, we see that
    \begin{equation*}
        \int_{ \mathcal{G}_k^n} |P_L(\xi)|^{2\alpha} \, dL = |\xi|^{2\alpha} \int_{ \mathcal{G}_k^n} |P_L(\xi/|\xi|)|^{2\alpha} \, dL = |\xi|^{2\alpha} \int_{ \mathcal{G}_k^n} |P_L(e_1)|^{2\alpha} \, dL ,
    \end{equation*}
    and the last integrals depends only on $n$, $k$ and $\alpha$. Putting everything together 
    \begin{align*}
         \int_{\mathcal{G}_k^n} dL \int_{L^\perp} [u(\cdot, y_L)]_{H^\alpha(L)}^2 \, d\mathcal{H}^{n-k}(y_L) & =  C_{F}(k,\alpha) \left( \int_{ \mathcal{G}_k^n} |P_L(e_1)|^{2\alpha} \, dL \right) \int_{\R^n} |\xi|^{2\alpha}|\mathcal{F} u |^2 d\mathcal{H}^{n}(\xi )  \\ &= \frac{C_{F}(k,\alpha)}{C_{F}(n,\alpha)} \left( \int_{ \mathcal{G}_k^n} |P_L(e_1)|^{2\alpha} \, dL \right) [u]_{H^\alpha(\R^n)}^2 ,
    \end{align*}
    where, in the last line, we have used Proposition \ref{fourier seminorms} again. 
    
    Lastly, since
    \begin{equation*}
        \int_{\mathcal{G}_k^n} dL \int_{L^\perp} [u(\cdot, y_L)]_{H^\alpha(L)}^2 \, d\mathcal{H}^{n-k}(y_L)  = \int_{ \mathcal{G}_k^n} dL \int_{L^\perp} [u]_{H^\alpha(L+y_L)}^2 \, d\mathcal{H}^{n-k}(y_L)   , 
    \end{equation*}
    this concludes the proof with 
   \begin{equation*}
        C(n,k,\alpha) = \frac{C_{F}(n,\alpha)}{C_{F}(k,\alpha)} \left( \int_{ \mathcal{G}_k^n} |P_L(e_1)|^{2\alpha} dL \right)^{-1} >0 \, .
    \end{equation*}
\end{proof}

\begin{rem}\label{rem: const in s dep}  Since by \cite[Corollary 4.2]{HitGuide} we have
\begin{equation*}
    \lim_{\alpha\to  1^-} (1-\alpha)C_F(n,\alpha) = \frac{\omega_{n-1}}{2n} ,
\end{equation*}
one can easily see that when $\alpha\geq \alpha_0\in (0,1)$ the constant $C(n,k,\alpha)$ in Theorem \ref{fourier slicing} is bounded, above and below, by constants depending only on $n$ and $\alpha_0$.
    
\end{rem}

\begin{cor}\label{cor:slicing_local}
Let $\alpha_0 \in (0,1)$ and let $\alpha \in [\alpha_0,1)$ and $R>0$. Then there exists a constant $c(\alpha_0,n)$ depending only on $\alpha_0$ and the space dimension $n$, such that for every $u\in H^{\alpha}(B_R)$ and every $k\in \{1,\dots,n-1\}$ it holds that
$$[u]^2_{H^\alpha(B_R)} \geq c(\alpha_0,n) \int_{\mathcal{G}^n_k} dL \int_{L^\perp} [u]^2_{H^\alpha((L+h) \cap B_R)} \, d\mathcal{H}^{n-k}(h).$$
\end{cor}

\begin{proof}
Let $\bar{u} :=  \fint_{B_R} u$. By Lemma~\ref{lemma:extension} there exists a constant $c(\alpha_0,n)=C_E(\alpha_0,n)^{-1}$ and an extension $E(u-\bar{u})$ such that
        \begin{equation*}
            [u]_{H^{\alpha}(B_R)}=[u-\bar{u}]_{H^{\alpha}(B_R)}\geq c(\alpha_0,n) [E(u-\bar{u})]_{H^{\alpha}(\R^n)} . 
        \end{equation*}
From Theorem~\ref{fourier slicing} we deduce that
\begin{eqnarray*}
[E(u-\bar{u})]_{H^{\alpha}(\R^n)}&=& C(n,k,\alpha) \int_{\mathcal{G}^n_k} dL \int_{L^\perp} [E(u-\bar{u})]^2_{H^\alpha((L+h))} \, d\mathcal{H}^{n-k}(h) \nonumber \\
&\geq& C(n,k,\alpha)\int_{\mathcal{G}^n_k} dL \int_{L^\perp} [E(u-\bar{u})]^2_{H^\alpha((L+h) \cap B_R)} \, d\mathcal{H}^{n-k}(h) \\ &=& C(n,k,\alpha)\int_{\mathcal{G}^n_k} dL \int_{L^\perp} [u]^2_{H^\alpha((L+h) \cap B_R)} \, d\mathcal{H}^{n-k}(h) ,
\end{eqnarray*}
where in the last line we have used that $E(u-\bar{u})=u-\bar{u}$ in $B_R$ and that $[u-\bar{u}]^2_{H^\alpha((L+h) \cap B_R)} = [u]^2_{H^\alpha((L+h) \cap B_R)} $ by the translation invariance of the seminorm. By Remark~\ref{rem: const in s dep}, this concludes the proof.
\end{proof}

Finally, we recall one of the Crofton formulas for rectifiable sets in $\R^n$. 

\begin{teo}[{Crofton formula, \cite[Theorem 5.4.3]{StocIntGeo}}]\label{teo:crofton}
Let $n \ge 2$, $ k\in \{1, \dotsc, n-1\}$ be integers, and let $M \subset \R^n$ be a $q$-dimensional rectifiable set with $k+q\ge n$. Then 
\begin{equation*}
    \int_{ \mathcal{G}^n_k} dL \int_{L^\perp}  \mathcal{H}^{k+q-n}(M \cap (L+h) ) \, d\mathcal{H}^{n-k}(h) = \beta_{k+q-n,n}^{k,q} \mathcal{H}^{q}(M) ,
\end{equation*}
where 
\begin{equation*}
    \beta_{i,j}^{k,\ell} := \frac{\Gamma\left(\frac{k+1}{2}\right)\Gamma\left(\frac{\ell+1}{2}\right)}{\Gamma\left(\frac{i+1}{2}\right)\Gamma\left(\frac{j+1}{2}\right)} .
\end{equation*}
    
\end{teo}

\subsection{Flat convergence of boundaries}\label{sbs: flat conv}

In this paper, we use integral currents just as a tool to define a distance, and hence a notion of convergence, for submanifolds of $\R^n$.

For this reason, here we recall the few basic notions we need to define the flat norm, and we refer to \cite{ABO2-singularities,2005-Indiana-ABO,1960-FF} for more details.

For $k\in \{1,\dots,n\}$, the space of $k$-currents in $\R^n$ is defined to be the dual of the space of smooth $k$-forms with compact support in $\R^n$.

For every $k\in\{1,\dots,n\}$, we have a boundary operator that associates to a $k$-current $T$ the $(k-1)$-current $\partial T$ defined as $ \langle \partial T, \omega \rangle \,:=\, \langle T,\mathrm{d}\omega \rangle $ for every $(k-1)$-form $\omega$. 

A $k$-current $T$ is called \emph{rectifiable} if it can be represented by integration on a rectifiable set, up to an integer multiplicity, namely if
$$\langle T,\omega \rangle =\int_{M} \theta(x) \, \omega_x(\tau(x)) \,d\mathcal{H}^k(x) ,$$
where $M$ is a $k$-rectifiable set, $\theta:M\to \Z\setminus\{0\}$ is an integer-valued measurable function, and $\tau(x)$ is a measurable choice of a unitary $k$-vector spanning $T_x M$ at $\mathcal{H}^k$-almost every $x\in M$. In this case we write $T=[M,\theta,\tau]$.

The mass of a rectifiable current $T=[M,\theta,\tau]$ is then defined as
$$\M(T):=\int_{M} |\theta(x)|\,d\mathcal{H}^k(x).$$

A $k$-current is called \emph{integral} if both $T$ and $\partial T$ are rectifiable.

To any submanifold $\Sigma\subset \R^n$ as in \eqref{defn:Sigma} we can associate the integral $(n-2)$-current
$$[\Sigma_1,d_1,\tau_{\Sigma_1}]+\dots+[\Sigma_m,d_m,\tau_{\Sigma_m}],$$
where $\tau_{\Sigma_i}(x)$ is the unitary $(n-2)$-vector given by the orientation of $T_x \Sigma_i$. Stokes' theorem ensures that the boundary operator agrees with the classical notion of boundary on currents coming from submanifolds.

Another special class of intergral $k$-currents that we need in the sequel is the class of \emph{integral Lipschitz chains}, namely  Lipschitz images of linear combinations (with integer coefficients) of currents of the form $[P,1,\tau_P]$, where $P$ is a $k$-dimensional convex polyhedron. We point out that, through a triangulation, the current associated to any oriented compact manifold can be written as an integral Lipschitz chain.

In \cite[Theorem~5.11]{1960-FF}, it was proved that on reasonable subsets of $\R^n$, and in particular in the whole Euclidean space, the boundary operator on integral currents generates a homology that is equivalent to the singular homology with integer coefficient. In particular, any integral $k$-current $T$ with $\partial T=0$ can be obtained as the boundary $\partial S$ of an integral $(k+1)$-current $S$.

Therefore, for an integral $k$-current without boundary $T$, we can define its \emph{flat norm of boundaries} (see \cite[section~2.5]{2005-Indiana-ABO}) as
$$\F(T):=\inf\{\M(S): S \text{ is an integral $(k+1)$-current with }\partial S =T \}.$$

Consequently, we say that a sequence $\{\Sigma_i\}$ of integral $k$-currents with null boundary converge to an integral $k$-current $\Sigma$ with respect to the flat distance of boundaries if $\F(\Sigma-\Sigma_i)\to 0$ as $i\to +\infty$.

\begin{rem}\label{rem:lip-chain}
If $\partial S=T$ is an integral Lipschitz chain, by \cite[Lemma~5.7]{1960-FF} we can approximate $S$ with an integral Lipschitz chain $S'$ with the same boundary $\partial S'=T$. Therefore, in this case, and in particular when $T$ is the difference between two closed oriented submanifolds with multiplicities, we have
$$\F(T):=\inf\{\M(S): S \text{ is an integral Lipschitz chain with }\partial S =T \}.$$
\end{rem}

\section{Prescribing singularities and weakly linking maps}\label{sec 3}

\subsection{Existence of a competitor}

The following result shows that the family $\mathfrak{F}_{\hspace{-.7pt} s}(\Sigma) $ is always not empty. This might not be the case for ambient Riemannian manifolds, as there could be topological obstructions to the existence of such a map. See Subsection \ref{sbs ambient Riem} for a short discussion on this feature.

\begin{prop}\label{comp existence} Let $\Sigma\subset \R^n$ be as in \eqref{defn:Sigma}. Then there exists a function $u: \R^n\to \Sp^1$, such that $u\in \mathfrak{F}_{\hspace{-.7pt} s}(\Sigma) $ for every $s\in (0,1)$. Moreover, such map is constant outside a compact set containing $\Sigma$. 
\end{prop}

There are many results in the literature, similar to our Proposition \ref{comp existence}, regarding functions with prescribed singularities.

In \cite{ABO2-singularities}, it has been shown (see in particular \cite[Theorem 4.4]{ABO2-singularities}) a similar existence result of a map $u:\R^n \to \Sp^1$ with prescribed Jacobian equal to (a multiple of) a codimension two integral current with finite mass and no boundary in $\R^n$. Nevertheless, on the one hand, the result in \cite[Theorem 4.4]{ABO2-singularities} is more general than ours since it deals with general integral currents. On the other hand it only shows $u\in W^{1,p}_{\rm loc}(\R^n; \Sp^1)$ for every $1<p<n/(n-1)$, and this range of $p$ is not enough even to conclude that $u\in H^\alpha_{\rm loc}(\R^n; \Sp^1)$ for $\alpha$ close to $1$.  

More recently, in \cite[Theorem 3.3]{Bre-Mi} the authors proved a similar existence result in the smooth setting and for $n=3$ in $ W^{1,p}_{\rm loc}(\R^3; \Sp^1) \cap C^\infty(\R^3 \setminus \Sigma; \Sp^1)$, for every $1<p<2$.

However, for our purposes, it is essential that the map $u$ can be taken constant outside a compact set since this implies that $[u]_{H^\alpha(\R^n)}^2 < +\infty $, and also that it looks like a vortex in each fiber of a tubular neighborhood of $\Sigma$, in order to establish Proposition~\ref{prop:gamma-limsup}.

For this reason, we present here a different proof, which relies on the following result about codimension two submanifolds (see~\cite[Chapter~VIII]{Kirby}).

\begin{teo}\label{teo:kirby}
Let $N$ be an oriented $n$-dimensional manifold and let $\Sigma \subset N$ be a closed, oriented, connected, $(n-2)$-dimensional submanifold such that $[\Sigma]=[0]\in H_{n-2}(N)$. Then the normal bundle of $\Sigma$ in $N$ is trivial, and there exists an oriented hypersurface $M\subset N$ such that $\Sigma=\partial M$.
\end{teo}

In our setting $N=\R^n$, so the assumption of Theorem~\ref{teo:kirby} is always trivially satisfied. Hence, in the following proof, for any $i\in\{1,\dots,m\}$, we let $M_i \subset\R^n$ be a so-called ``Seifert hypersurface'' for $\Sigma_i$, namely an oriented hypersurface such that $\partial M_i = \Sigma_i$.

\begin{proof}[Proof of Proposition \ref{comp existence}]

We consider first the case of a connected surface $\Sigma=\partial M$ with multiplicity one, that is $m=d_1=1$ in \eqref{defn:Sigma}. We parameterize the normal bundles of $M$ and $\Sigma$ in the following compatible way.

Let $\delta>0$ be a small positive number. First, we take $\Phi_M \colon M\times [-\delta,\delta]\to \R^n$ defined by
$$\Phi_M(y,t):= y + t n_M(y),$$
where $n_M\in C^1(M;\Sp^{n-1})$ is normal to $M$ and such that $(n_M(y), v_1(y), \dots, v_{n-1}(y))$ is a positive basis of $\R^n$ whenever $(v_1(y), \dots, v_{n-1}(y))$ is a positive basis of $T_y M$.

Now we construct a parameterization $\Phi_\Sigma:\Sigma \times [-\delta,\delta]^2\to \R^n$ of a tubular neighborhood for $\Sigma$ in the following way. We start by parameterizing a collar neighborhood of $\Sigma$ inside $M$ with the map $\Phi_\Sigma:\Sigma \times [0,\delta]\times \{0\} \to M$ defined as
$$\Phi_\Sigma(\sigma,p_1,0):=\mathrm{Exp}_{\sigma} ^{M} (p_1 n_\Sigma(\sigma)),$$
where $n_\Sigma(\sigma)\in T_\sigma M$ is the unit normal vector to $\Sigma$ at $\sigma$ pointing inward.

We observe that this definition implies that
\begin{equation}\label{def:collar}
\Phi_\Sigma(\sigma,0,0)=\sigma, \qquad \text{and}\qquad\frac{\partial \Phi_\Sigma}{\partial p_1} (\sigma,0,0)=n_\Sigma(\sigma), \qquad \forall \sigma\in\Sigma.
\end{equation}

Then we extend it to $\Phi_\Sigma:\Sigma \times [-\delta,\delta]^2 \to \R^n$ by setting 
\begin{equation} \label{defn:Phi_Sigma}
\Phi_\Sigma(\sigma,p_1,p_2):=
\begin{cases}
\Phi_M(\Phi_\Sigma(\sigma,p_1,0),p_2) & \text{if }p_1\geq 0,\\[0.5ex]
\sigma+ p_1 n_\Sigma (\sigma) + p_2 n_M(\sigma) & \text{if } p_1\leq 0.
\end{cases}
\end{equation}

By computing the partial derivatives, it is easy to check that $\Phi_\Sigma\in C^1(\Sigma \times [-\delta,\delta]^2 ;\R^n)$. We also observe that $\mathrm{d}\Phi_\Sigma(\sigma,0,0)$ is an isometry because of~\eqref{def:collar} and
    \begin{equation} \label{eq:isom-p2}
        \frac{\partial \Phi_\Sigma}{\partial p_2} (\sigma,0,0)=n_M(\sigma).
    \end{equation}
Therefore there exists a small positive real number $\delta_0>0$ such that both the maps
    \begin{equation} \label{eq:delta-0-uniform}
        \Phi_\Sigma:\Sigma \times [-\delta_0,\delta_0]^2\to \Phi_\Sigma(\Sigma \times [-\delta_0,\delta_0]^2)
        \quad \text{and} \quad
        \Phi_M:M\times [-\delta_0,\delta_0]\to \Phi_M(M\times [-\delta_0,\delta_0])
    \end{equation}
are diffeomorphisms of class $C^1$, namely they parameterize some tubular neighborhoods of $\Sigma$ and $M$ in $\R^n$. We also observe that $\Phi_\Sigma ^{-1}(M)=\Sigma \times [0,\delta_0]\times \{0\}$.

Now let us fix a smooth function $\vartheta:[-1,1]\to [0,2\pi]$ such that $\vartheta(t)=0$ for $t\in [-1,-1/2]$ and $\vartheta(t)=2\pi$ for $t\in [1/2,1]$.

Then we fix a smooth function $u_* \colon [-1,1]^2\to \Sp^1$ such that
    \[
    u_*(p) =
    \begin{cases}
        p/|p| &\text{if $p\in [-1/2,1/2]^2$},\\[0.5ex]
        (1,0) &\text{if $p\in [-1,3/4]\times [-1,1] \setminus [-3/4,3/4]^2$},\\[0.5ex]
        (\cos(\vartheta(p_2)),\sin(\vartheta(p_2)) ) &\text{if $p\in [3/4,1]\times [-1,1]$}.
    \end{cases}
    \]

We observe that such a function exists because
$$\deg(u_*,\partial([-1/2,1/2]^2))=\deg(u_*,\partial ([-3/4,3/4]^2)),$$
and any two functions with the same degree on the circle are homotopic (and in this case the homotopy can be made smooth).

Now we define a function $u_{\Sigma} \colon \R^n\setminus \Sigma \to \Sp^1$ as follows.
    \begin{equation} \label{eq:u-Sigma}
        u_{\Sigma}(x):=
        \begin{cases}
            u_*(p/\delta_0) &\text{if $x=\Phi_\Sigma(\sigma,p)$ with $(\sigma,p)\in \Sigma\times [-\delta_0,\delta_0]^2$},\\[0.5ex]
            u_*(1,t/\delta_0) 
            &\text{if $x=\Phi_M(y,t)$ with $(y,t)\in M\times [-\delta_0,\delta_0] \setminus \Phi_M^{-1}( \Phi_\Sigma(\Sigma \times [-\delta_0,\delta_0]^2))$}.\\[0.5ex]
            (1,0) &\text{if }x\in \R^n\setminus ( \Phi_\Sigma(\Sigma\times [-\delta_0,\delta_0]^2) \cup \Phi_M(M\times [-\delta_0,\delta_0]) )
        \end{cases}
    \end{equation}

We point out that the definitions of the parameterizations $\Phi_M$, $\Phi_\Sigma$, and of the function $u_*$ ensure that the junction in $\Phi_\Sigma(\Sigma\times\{\delta_0\}\times [-\delta_0,\delta_0])$ is of class $C^1$.

We also observe that the function $u_{\Sigma}$ defined in this way is identically equal to $(1,0)$ near the boundary of $\Phi_\Sigma(\Sigma\times [-\delta_0,\delta_0]^2) \cup \Phi_M(M\times [-\delta_0,\delta_0])$, hence also the junction at this boundary is smooth.

Hence we have defined a function $u_{\Sigma} \in C^1(\R^n\setminus \Sigma ;\Sp^1)$ that coincides with the planar vortex on each fiber of a small tubular neighborhood of $\Sigma$, and is constant outside a neighborhood of the compact submanifold $M$.

We claim that the function $u_{\Sigma}$ that we have just defined actually links with $\Sigma$.

Let $\gamma \in L(\Sigma)$. Up to a homotopy, we can assume that $\gamma\subset \R^n\setminus \Phi_\Sigma(\Sigma\times [-\delta_0,\delta_0]^2)$, and that $\gamma$ intersects $M$ transversally, and more precisely that $\gamma\cap \Phi_M(M\times [-\delta_0,\delta_0])$ consists of finitely many segments of type $\Phi_M(\{y\}\times [-\delta_0,\delta_0])$, for some $y\in M$. Since the restriction of $u_{\Sigma}$ to $\gamma$ is constant outside these segments, if we fix any parameterization (or any orientation) of $\gamma$ we obtain that
    \begin{equation}
    \label{eq:linking-condition}
    \deg(u_{\Sigma},\gamma)
    =
    \sum_{y\in M\cap \gamma} s(y)
    =
    (-1)^{n-2} \link(\gamma,\Sigma)
    \end{equation}
where $s(y)\in \{\pm 1\}$ depends on whether the segment $\Phi_M(\{y\}\times [-\delta_0,\delta_0])$ is crossed from $\Phi_M(y,-\delta_0)$ to $\Phi_M(y,\delta_0)$, or vice versa.

It remains to prove that $\mathcal{E}_{\alpha}(u_{\Sigma}, \R^n) < +\infty$ for any $\alpha \in (0,1)$. This is a consequence of the fact that the function $p/\abs{p} \in H^{\alpha}_{\mathrm{loc}}(\R^2; \Sp^1)$ for any $\alpha \in (0,1)$. However, we will perform this computation in detail in the proof of Proposition~\ref{prop:gamma-limsup}.

\smallskip

Now, we treat the general case. Let $\Sigma$ be as in~\eqref{defn:Sigma}, then there exists $\delta_0 > 0$ such that for any $i\in\{1,\dots,m\}$ the maps $\Phi_{\Sigma_i}$ and $\Phi_{M_i}$ in~\eqref{eq:delta-0-uniform} are diffeomorphism of class $C^1$. In addition, we can assume that the sets $\Phi_{\Sigma_i}(\Sigma_i \times [-\delta_0, \delta_0]^2)$ are pairwise disjoint. For any $i\in\{1,\dots,m\}$, we consider the function $u_{\Sigma_i}$ defined by~\eqref{eq:u-Sigma}, then we claim that the function
    \begin{equation} \label{eq:u-Sigma-2}
        u_{\Sigma}(x):= u_{\Sigma_1}^{d_1}(x) \cdot\ldots\cdot u_{\Sigma_m}^{d_m}(x), \quad \forall x \in \R^n \setminus \Sigma,
    \end{equation}
has the desired properties, where the product has to be understood in the sense of complex numbers identifying $\Sp^1$ as the unit circle in $\C$. 

From the preceding discussion we know that for any $i\in\{1,\dots,m\}$ we have $u_{\Sigma_i} \in C^1(\R^n \setminus \Sigma_i; \Sp^1)$, therefore $u_{\Sigma} \in C^1(\R^n \setminus \Sigma; \Sp^1)$. We prove that $u_{\Sigma}$ links with $\Sigma$. It is well-known that the degree is additive with respect to complex multiplication, hence combining this fact with~\eqref{eq:linking-condition} and~\eqref{eq:u-Sigma-2}, we deduce that for any $\gamma \in L(\Sigma)$ we have
    \begin{equation}
    \label{eq:linking-condition-2}
    \deg(u_{\Sigma},\gamma)
    =
    d_1 \deg(u_{\Sigma_1},\gamma)
    +\dots+
    d_m \deg(u_{\Sigma_m},\gamma)
    =
    (-1)^{n-2} \big( \link(\gamma,\Sigma_1)
    +\dots+
    \link(\gamma,\Sigma_m) \big),
    \end{equation}
which in particular implies that $u_{\Sigma}$ links with $\Sigma$.

At this point, combining the elementary inequalities $\abs{ab-cd} \le \abs{a-c} + \abs{b-d}$, which is valid for any $a,b,c,d \in \Sp^1 \subset \mathbb{C}$, and $(a_1+\dots+a_m)^2\leq m(a_1^2+\dots+a_m^2)$, we derive that for any $\alpha \in (0,1)$ it holds that
    \begin{equation}
        \label{eq:subadd}
        \mathcal{E}_{\alpha}(u_{\Sigma}, \R^n)
        \le
        m\big( d_1 ^2 \mathcal{E}_{\alpha}(u_{\Sigma_1}, \R^n)
        +\dots+
        d_m ^2
        \mathcal{E}_{\alpha}(u_{\Sigma_m}, \R^n) \big),
    \end{equation}
    and since each addendum on the right-hand side is finite this completes the proof.   
\end{proof}

\subsection{Existence of a minimizer}

We can now show that there always exists a minimizer in the problem defining the fractional mass \eqref{fraccontdef}.

\begin{lemma}\label{lem: minimizer existence}
     Let $\Sigma\subset \R^n$ be as in \eqref{defn:Sigma}. Then the minimum in \eqref{fraccontdef} is achieved by a map $ u_\circ \in \mathfrak{F}_{\hspace{-1pt} s}^{\hspace{.9pt} w}  (\Sigma)$.
\end{lemma}

\begin{proof}
   
Let $m_0 \ge 0$ be the infimum in \eqref{fraccontdef} and $ \{u_k\} \subset \mathfrak{F}_{\hspace{-1pt} s}^{\hspace{.9pt} w}  (\Sigma) $ be a minimizing sequence, that is $\mathcal{E}_{\frac{1+s}{2}}(u_k, \R^n) \downarrow m_0$. The fact that $m_0<+\infty$ is clear by considering the explicit competitor given in Proposition \ref{comp existence}. 

\vsp 
In particular $\mathcal{E}_{\frac{1+s}{2}}(u_k, \R^n) \le C $ for some $C>0$ independent of $k$. Hence by Remark \ref{rem: Xs closure } there exists a function $u_\circ \in H^{\frac{1+s}{2}}_{\mathrm{loc}}(\R^n;\Sp^1) $ such that $u_k \to u_\circ$ in $\mathbb{X}_s(\R^n)$. Moreover, there holds
\begin{equation*}
    \mathcal{E}_{\frac{1+s}{2}}(u_\circ, B_R) \le \liminf_{k\to \infty } \mathcal{E}_{\frac{1+s}{2}}(u_k, B_R) ,
\end{equation*}
for every $R \ge 1$. Therefore 
\begin{equation*}
\mathcal{E}_{\frac{1+s}{2}}(u_\circ, \R^n ) = \sup_{R\geq 1} \mathcal{E}_{\frac{1+s}{2}}(u_\circ, B_R) \le \sup_{R\geq 1} \liminf_{k\to \infty } \mathcal{E}_{\frac{1+s}{2}}(u_k, B_R) \le  \liminf_{k\to \infty } \mathcal{E}_{\frac{1+s}{2}}(u_k, \R^n) = m_0 .
\end{equation*}

To show that $u_\circ$ is a minimizer we are left to prove that $u_\circ \in \mathfrak{F}_{\hspace{-1pt} s}^{\hspace{.9pt} w}  (\Sigma)$. Since $\{u_k\} \subset \mathfrak{F}_{\hspace{-1pt} s}^{\hspace{.9pt} w}  (\Sigma)$ there is $\{u_{k,i}\} \subset \mathfrak{F}_{\hspace{-1pt} s}^{\hspace{.9pt} }  (\Sigma)$ such that $u_{k,i} \to u_k$ in $\mathbb{X}_s(\R^n)$ as $i\to \infty $ (see Remark \ref{rem: Xs closure }). Moreover, by the first part of this proof, $u_k \to u_\circ$ in $\mathbb{X}_s(\R^n)$. Hence, as $\mathbb{X}_s(\R^n)$ is a complete metric space there exists a diagonal subsequence $\{i(k)\} \subset \N^+$ such that $u_{k,i(k)} \to u_\circ $ in $\mathbb{X}_s(\R^n)$ as $k\to \infty$. Since $\mathfrak{F}_{\hspace{-1pt} s}^{\hspace{.9pt} w}  (\Sigma)$ is the closure of $\mathfrak{F}_{\hspace{-1pt} s}^{\hspace{.9pt} }  (\Sigma)$ with respect to the distance $\mathrm{d}_{\mathbb{X}_s(\R^n)}$ of $\mathbb{X}_s(\R^n)$, we have that $u_\circ \in \mathfrak{F}_{\hspace{-1pt} s}^{\hspace{.9pt} w}  (\Sigma)$.
\end{proof}

 \subsection{Linking with squares}\label{sec: link sq}

In this section we show that all weakly linking maps, and in particular the minimizers for \eqref{fraccontdef} whose existence is proved in Lemma \ref{lem: minimizer existence}, still satisfy in a weak sense the geometric property that defines the class \eqref{Flink}.

    To this end, for $n\ge 2$, we need to parameterize the space $\mathcal{S}_n$ of planar squares in $\R^n$, and more precisely of their boundaries. Indeed, we remark that we will call ``square'' the union of the four segments forming the boundary of a planar square.

\vsp 
Let us first consider the case $n=2$. In this case we can describe the set of squares in $\R^2$ as
$$\mathcal{S}_2:=\Big\{ \, Q_{\theta,t}(x) : \theta \in \Sp^1,\ x \in \R^2 ,\ t>0   \Big\} ,$$
where $Q_{\theta,t}(x)$ is (the boundary of) the square centered at $x$, with sides of length $2t$, two of which are parallel to $\theta$.

We observe that actually each square corresponds to four different choices of $\theta$. Anyway, we can exploit this description to endow $\mathcal{S}_2$ with a measure $Q_2$ defined by
$$\int_{\mathcal{S}_2} \varphi\, dQ_2 := \int_{\Sp^{1}}d\theta \int_{\R^2} dx \int_{0} ^{+\infty} \varphi(Q_{\theta,t}(x))\,dt$$
for every $\varphi \in C_c(\mathcal{S}_2)$.

In the higher dimensional case $n\geq 3$, we also need to fix the plane in which the square lies and, since we do not have a canonical set of directions in each plane, the parameter $\theta$ has to be chosen a bit more carefully. So we write
\begin{equation*}
   {\mathcal{S}}_n := \Big\{ \, Q^{L+h}_{\theta,t}(x)  : L\in \mathcal{G}^n_2 , h\in L^\perp, x \in L+h , \theta \in \partial \mathrm{D}^{L+h}_1(x) , t>0   \Big\} ,
\end{equation*}
where $\mathrm{D}^{L+h}_1(x)$ is the unit disk with center $x$ in the affine $2$-plane $L+h$ and $Q^{L+h}_{\theta,t}(x)$ is (the boundary of) the square contained $L+h$, centered at $x$, with sides of length $2t$, two of which are parallel to $\theta-x$. We endow this space with the measure $Q_n$ defined by
   \begin{equation}\label{eq: measure def Qn}
       \int_{\mathcal{S}_n} \varphi \,  dQ_n:= \int_{\mathcal{G}^n_2 } dL \int_{L^\perp} d\mathcal{H}^{n-2}(h) \int_{L+h} d\mathcal{H}^2(x) \int_{\partial\mathrm{D}^{L+h}_1(x)} d\theta \int_0^\infty \varphi(Q^{L+h}_{\theta,t}(x)) \, dt
    \end{equation}
for every $\varphi \in C_c(\mathcal{S}_n)$.

In what follows, we say that some property holds for \textit{almost every square $Q\in S$}, for some subset $S\subset \mathcal{S}_n$, meaning that the set of squares $Q\in S$ for which that property fails is negligible with respect to the measure $Q_n$.

\begin{lemma}\label{lem: cont trace a.e.} Let $\Sigma \subset \R^n$ be as in \eqref{defn:Sigma}, let $ s \in (0,1)$ and $u \in \mathfrak{F}_{\hspace{-1pt} s}^{\hspace{.9pt} w}  (\Sigma)$. Then $u$ has continuous trace on almost every square $Q\in \mathcal{S}_n$ and
\begin{equation}\label{th:link_squares}
\abs{\deg(u,Q)}=\abs{d_1\link(Q,\Sigma_1)+\dots+d_m\link(Q,\Sigma_m)},
\end{equation}
for almost every square $Q\in \mathcal{S}_n\cap L(\Sigma)$.
\end{lemma}
\begin{proof}

Let us fix $s_0\in(0,s)$. In what follows, let $c$ denote a constant that depends only on $n$ and $s_0$, and whose value may vary in each line.

Let us fix $R>0$. By Corollary~\ref{cor:slicing_local} with $k=2$, for every $v \colon B_R\to \R^2$ we have 
 \begin{equation}
       [v]^2_{H^\frac{1+s_0}{2}(B_R)}   \geq c \int_{\mathcal{G}^n_2} dL \int_{L^\perp} [v]^2_{H^\frac{1+s_0}{2}((L+h) \cap B_{R} )} \, d\mathcal{H}^{n-2}(h)  .
       \label{eq: wlink 1}
   \end{equation}

We observe that $(L+h) \cap B_{R}$ is a two-dimensional disk of radius $\sqrt{R^2-|h|^2 }$ and center $h$ in the plane $L+h$, if $|h|\le R$, and is empty otherwise.

Hence, for every $L\in \mathcal{G}^n_2$ and for every $h \in L^\perp \cap \{|h| \le R/4 \}$, from \eqref{eq:planar_slicing_squares} we deduce that
\begin{align*}
    [v]^2_{H^\frac{1+s_0}{2}((L+h)\cap B_{R})}
    &\geq \fint_{(L+h)\cap B_{R/4}(h) } [u]_{H^{\frac{1+s_0}{2}}(\mathrm{D}^{L+h} _{R/4}(x))} ^2 \, d\mathcal{H}^{2}(x)
    \\
    &\geq \fint_{(L+h)\cap B_{R/4}(h) } \, d\mathcal{H}^{2}(x)  \int_{\partial\mathrm{D}^{L+h}_1(x)}d\theta \\
    \int_0 ^{R/4}\Bigl( [u]_{H^{\frac{1+s_0}{2}}(S^{L+h}_{\theta,t}(x))} ^2&+[u]_{H^{\frac{1+s_0}{2}}(S^{L+h}_{\theta^\perp,t}(x))} ^2+[u]_{H^{\frac{1+s_0}{2}}(S^{L+h}_{\theta,-t}(x))} ^2+[u]_{H^{\frac{1+s_0}{2}}(S^{L+h}_{\theta^\perp,-t}(x))} ^2\Bigr)\, dt ,
\end{align*}
where $S^{L+h}_{\theta,t}(x):=\{x+t(\theta-x)^\perp +\xi (\theta-x):\xi \in (-\sqrt{R/16-t^2},\sqrt{R/16-t^2}) \}$.

Applying Lemma~\ref{lemma:frac_embedd} on the four sides of $Q^{L+h}_{\theta,t}(x)$, since the sum of the oscillation of $u$ on these sides is larger than or equal to the oscillation of $u$ on $Q^{L+h}_{\theta,t}(x)$, we obtain that
$$[v]^2_{H^\frac{1+s_0}{2}((L+h)\cap B_{R})}
    \geq c\fint_{(L+h)\cap B_{R/4}(h) } \, d\mathcal{H}^{2}(x)  \int_{\partial\mathrm{D}^{L+h}_1(x)} d\theta \int_0 ^{R/8} t^{-s_0} \osc(v, Q^{L+h}_{\theta, t}(x)   )^2\, dt.
$$

Hence, from \eqref{eq: wlink 1} we get 
 \begin{multline}
    [v]^2_{H^\frac{1+s_0}{2}(B_R)} \ge c \int_{\mathcal{G}^n_2} dL \int_{L^\perp \cap \{|h| \le R/4\}} d\mathcal{H}^{n-2}(h)\\
    \fint_{(L+h)\cap B_{R/4}(h)} d\mathcal{H}^{2}(x) \int_{\partial\mathrm{D}^{L+h}_1(x)} d\theta \int_0^{R/8} t^{-s_0} \osc(v, Q^{L+h}_{\theta, t}(x) )^2 dt . \label{eq: wlink 2}
\end{multline}

Since $u \in \mathfrak{F}_{\hspace{-1pt} s}^{\hspace{.9pt} w}  (\Sigma)$, there exists a sequence $\{u_k\} \subset \mathfrak{F}_{\hspace{-.7pt} s}  (\Sigma)$ such that $u_k \to u$ in $\mathbb{X}_s(\R^n)$.
Then, applying \eqref{eq: wlink 2} with $v=u-u_k$, we get
\begin{multline*}
     0=\lim_{k\to +\infty} [u-u_k]^2_{H^{\frac{1+s_0}{2}}(B_{R})} \ge \lim_{k\to +\infty} c \int_{\mathcal{G}^n_2} dL \int_{L^\perp \cap \{|h| \le R/4\}} d\mathcal{H}^{n-2}(h)\\
     \fint_{(L+h)\cap B_{R/4}(h)} d\mathcal{H}^{2}(x) \int_{\partial\mathrm{D}^{L+h}_1(x)} d\theta \int_0^{R/8} t^{-s_0} \osc(u-u_k, Q^{L+h}_{\theta, t}(x))^2 \, dt . 
   \end{multline*} 
   
Hence, since the oscillation is bounded by $2$, the dominated convergence theorem implies that $u_k\to u$ uniformly on almost every square $Q^{L+h}_{\theta, t}(x)\in \mathcal{S}_n$ with $|h|<R/4$, $|x-h|<R/4$, and $t<R/8$.

Since all squares in $\mathcal{S}_n$ satisfy these conditions for $R$ large enough, we deduce that $u_k\to u$ in $L^\infty(Q)$ for almost every square $Q\in \mathcal{S}_n$, and in particular $u$ has continuous trace on almost every $Q\in \mathcal{S}_n$.

Moreover, since $u_k \in \mathfrak{F}_{\hspace{-.7pt} s}  (\Sigma)$, we know that
$$\abs{\deg(u_k, Q)}=\abs{d_1\link(Q,\Sigma_1)+\dots+d_m\link(Q,\Sigma_m)}$$
for all $Q \in \mathcal{S}_n \cap L(\Sigma)$. Since the degree is continuous under uniform convergence, this implies that \eqref{th:link_squares} holds for almost every $ Q\in \mathcal{S}_n \cap L(\Sigma)$. 
\end{proof}

\subsection{Positivity of the fractional mass}\label{sbs non sharp lower}

Building on the previous results about weakly linking functions, in this subsection we prove that the fractional mass is always positive. Actually, we prove a stronger result, namely that the mass is locally positive around any point in $\Sigma$.

Before stating our result, we prove a few preliminary estimates in the planar case.

\begin{lemma}\label{lemma:deg_square}
Let $Q=Q_{\theta,t}(x)\in \mathcal{S}_2$ be a square in $\R^2$ and let $L_1,L_2,L_3,L_4$ denote its four sides. Let $u:Q \to \Sp^1$ be continuous and $d=\abs{\deg(u,Q)}$.
Then, for every $s_0\in (0,1)$ there exists a constant $c=c(s_0)>0$ depending only on $s_0$ such that for every $s\in (s_0,1)$ we have
$$[u]_{H^{\frac{1+s}{2}}(L_1)} ^2 +[u]_{H^{\frac{1+s}{2}}(L_2)}^2+[u]_{H^{\frac{1+s}{2}}(L_3)}^2+[u]_{H^{\frac{1+s}{2}}(L_4)}^2\geq c \,  \frac{d}{1-s}  t^{-s}$$
\end{lemma}

\begin{proof}
Let us assume that $d>0$, otherwise the thesis is trivial.

Then we can write $Q$ ad the union of $d$ connected curves $\gamma_1,\dots,\gamma_d$, intersecting at most at their endpoints, such that $u\colon\gamma_i\to \Sp^1$ is surjective for every $i$.

Then we can write
$$\sum_{j=1}^{4} [u]_{H^{\frac{1+s}{2}}(L_j)}^2 \geq \sum_{i=1}^{d}\sum_{j=1}^{4} [u]_{H^{\frac{1+s}{2}}(L_j\cap \gamma_i)}^2,$$
hence, by Lemma~\ref{lemma:frac_embedd}, we obtain that
\begin{equation*}
\sum_{j=1}^{4} [u]_{H^{\frac{1+s}{2}}(L_j)}^2 
\geq \frac{C_S^{-1}}{1-s} \sum_{i=1}^{d}\sum_{j=1} ^{4} \mathcal{H}^1(L_j\cap \gamma_i)^{-s} \osc(u,L_j\cap \gamma_i)^2
\geq \frac{C_S^{-1}(2t)^{-s}}{1-s} \sum_{i=1}^{d}\sum_{j=1} ^{4} \osc(u,L_j\cap \gamma_i)^2,
\end{equation*}
and we conclude by observing that
$$\sum_{i=1}^{d}\sum_{j=1} ^{4} \osc(u,L_j\cap \gamma_i)^2\geq \sum_{i=1}^{d}\frac{1}{4} \Biggl(\sum_{j=1} ^{4} \osc(u,L_j\cap \gamma_i)\Biggr)^2\geq \sum_{i=1}^{d}\frac{1}{4} \osc(u,\gamma_i)^2= d,$$
because the surjectivity of $u$ on $\gamma_i$ implies that $\osc(u,\gamma_i)=2$ for every $i\in \{1,\dots,d\}$.
\end{proof}

\begin{lemma}\label{lem: 2d lb squares} In the case $n=2$, let $\mathrm{D}_r(x_0)$ denote the disk of center $x_0\in\R^2$ and radius $r>0$. Let $d\in \N_+$, let $s_0 \in (0,1)$ and $s\in [s_0,1)$, and let $ u \in  H^{\frac{1+s}{2}}(\mathrm{D}_r(x_0) ; \Sp^1) $ be a map with $\abs{\deg(u , Q )}=d$ for almost every square $Q \in \mathcal{S}_2 $ that is contained in $\mathrm{D}_r(x_0)$ for which $x_0$ lies inside $Q$, namely in the bounded connected component of $\R^2\setminus Q$. Then there exists a constant $c=c(s_0)>0$ such that 
\begin{equation*} 
    [u]_{H^\frac{1+s}{2}(\mathrm{D}_r(x_0))}^2 \ge c \, \frac{d}{(1-s)^2}  r^{1-s} .
\end{equation*}
    
\end{lemma}
\begin{proof}
By scaling and translating, we can assume that $r=1$ and $x_0=0$. Let us fix $x\in \mathrm{D}_{1/4}(0)$, and let us set $S_{\theta, t} (x)= \{x+t\theta^\perp + \xi \theta : \xi \in (-\sqrt{1/4-t^2},\sqrt{1/4-t^2}) \}$, where $\theta^\perp$ is the counterclockwise rotation of $\theta$ by an angle equal to $\pi/2$. 
    
By the slicing formula \eqref{eq:planar_slicing_squares}, we obtain that
 \begin{multline*}
        [u]^2_{H^{\frac{1+s}{2}}(\mathrm{D}_{1/2}(x))}  =\\
         \frac{1}{4} \int_{ \Sp^1} d\theta  \int_{0}^{1/2} \Big( [u]^2_{H^\frac{1+s}{2}(S_{\theta, t}(x))} + [u]^2_{H^\frac{1+s}{2}(S_{\theta, -t}(x))} +  [u]^2_{H^\frac{1+s}{2}(S_{\theta^\perp, t}(x))} + [u]^2_{H^\frac{1+s}{2}(S_{\theta^\perp, -t}(x))} \Big) \, dt.
    \end{multline*}
   
We observe that when $t<\sqrt{2}/4$ the four segments $S_{\theta, t}(x)$, $S_{\theta, -t}(x)$, $S_{\theta^\perp, t}(x)$ and $S_{\theta^\perp, -t}(x)$ appearing in the right-hand side contain the four sides of the square $Q_{\theta,t}(x)$.

By assumption we know that for almost every $x \in \mathrm{D}_{1/4}(0)$, for almost every $\theta \in \Sp^1$ and for almost every $t\in (|x|,\sqrt{2}/4)$, we have that $\abs{\deg(u,Q_{\theta,t}(x))}=d$. Therefore, from Lemma~\ref{lemma:deg_square} we deduce that for almost every $x\in \mathrm{D}_{1/4}(0)$ it holds that
$$[u]^2_{H^{\frac{1+s}{2}}(\mathrm{D}_{1/2}(x))}\geq \frac{1}{4} \int_{\Sp^1}d\theta \int_{|x|} ^{\sqrt{2}/4} \frac{c d}{1-s}  t^{-s}\,dt = \frac{cd}{(1-s)^2}\bigl(2^{-\frac{3}{2}(1-s)}- |x|^{1-s}\bigr),$$
for some constants $c_1,c_2>0$ depending only on $s_0$.

Now let us fix $\lambda \in (0,1/4)$. Then we have that
\begin{multline*}
    [u]^2 _{H^{\frac{1+s}{2}}(\mathrm{D}_1(0))}  = \fint_{\mathrm{D}_\lambda(0)} [u]^2 _{H^{\frac{1+s}{2}}(\mathrm{D}_1(0))} \, dx \geq  \fint_{\mathrm{D}_\lambda(0) } [u]^2 _{H^{\frac{1+s}{2}}(\mathrm{D}_{1/2}(x))} \,dx  \\
\geq \frac{cd}{(1-s)^2} \fint_{\mathrm{D}_\lambda} \bigl(2^{-\frac{3}{2}(1-s)}- |x|^{1-s}\bigr)\, dx \geq \frac{cd}{(1-s)^2} \bigl(2^{-\frac{3}{2}(1-s)}- \lambda^{1-s}\bigr).
\end{multline*}

Letting $\lambda \to 0^+$, we conclude the proof.
\end{proof}

Exploiting the slicing formulas proved in subsection~\ref{sbs: ONB}, we can pass from the planar case to our case of general dimension. To this end, let us introduce some notation and prove another preliminary geometric lemma.

We endow $\mathcal{G}_2 ^n$ with the following metric structure. For $L, L' \in \mathcal{G}_2 ^n$ we set
$$\mathrm{d}_\mathcal{G}(L,L'):= \| P_L - P_{L'}\| = \sup_{x \in \Sp^{n-1}} |(P_L - P_{L'})(x)|,$$
where $P_L, P_{L'}$ are the orthogonal projections on $L,L'$ respectively. This metric is invariant under the action of $O(n)$. Moreover, for every $L\in \mathcal{G}_2^n$ and $x\in \R^n$, let us denote by $x_{L^\perp}$ the orthogonal projection of $x$ onto $L^\perp$, namely $x_{L^\perp}:=x-P_L(x)$.

\begin{lemma}\label{lemma:r0}
Let $\Sigma$ be as in \eqref{defn:Sigma}. Then there exists a real number $r_0=r_0(\Sigma)>0$, depending only on $\Sigma$, with the following property.

For every $(r,x)\in (0,r_0)\times \Sigma$, every $L\in \mathcal{G}_2^n$ with $\mathrm{d}_{\mathcal{G}}(L, (T_x \Sigma)^\perp) \le 1/4$ and every $h\in L^\perp$ with $|h-x_{L^\perp}|\le r/8$, there exists a unique point $y\in B_{r/4}(x)$ such that
    \begin{equation*}
        (L+h) \cap \Sigma \cap B_r(x) = \{y\} ,
    \end{equation*}
and the intersection between $L+h$ and $\Sigma$ in $B_r(x)$ is transversal.
\end{lemma}

\begin{proof}
Let us assume by contradiction that there exists $\Sigma$ as in \eqref{defn:Sigma} such that there is no $r_0>0$ with the required property.

Then we can find a sequence of numbers $r_k\to 0^+$, a sequence of points $\{x_k\}\subset \Sigma$, a sequence of $2$-planes $\{L_k\}\subset \mathcal{G}_2 ^n$ and a sequence $\{h_k\}\subset \R^n$ with $h_k\in L_k ^\perp$ for every $k\in \N$ such that $|h_k-(x_k)_{L_k ^\perp}|\leq r_k/8$ and $\mathrm{d}_{\mathcal{G}} (L_k, (T_{x_k} \Sigma)^\perp)\leq 1/4$, and one of the following three options holds
\begin{itemize}
    \item $(L_k+h_k)\cap \Sigma \cap B_{r_k/4}(x_k)$ is empty,
    \item $(L_k+h_k)\cap \Sigma \cap B_{r_k}(x_k)$ contains (at least) two distinct points,
    \item $L_k+h_k$ is not transversal to $\Sigma$ in $B_r(x)$.
\end{itemize}

Up to subsequences (not relabelled), we can also assume that all the points $x_k$ belong to the same connected component $\Sigma_i$ and that, as $k\to +\infty$, it holds that
$$\frac{h_k-(x_k)_{L_k ^\perp}}{r_k} \to h_*, $$
for some $h_* \in \R^n$ with $|h_*|\leq 1/8$, that $x_k \to x_*$ for some $ x_*\in \Sigma_i$, and that $L_k\to L_*$ for some $L_*\in \mathcal{G}_2 ^n$ with $\mathrm{d}_{\mathcal{G}}(L_*,(T_{x_*}\Sigma_i)^\perp)\leq 1/4$.

We observe that
$$\frac{L_k +h_k-x_k}{r_k} \cap B_1(0)=\Biggl(L_k +\frac{h_k-(x_k)_{L_k ^\perp}}{r_k} \Biggr)\cap B_1(0),$$
is a sequence of disks converging to $(L_*+h_*) \cap B_1(0)$, while, taking into account the $C^2$ regularity of $\Sigma_i$, we deduce that
$$\frac{\Sigma-x_k}{r_k} \cap B_1(0)$$
is a sequence of $C^2$ graphs over $T_{x_k} \Sigma_i$, and hence eventually also over $T_{x_*}\Sigma_i$, converging in $C^2$ to the disk $T_{x_*} \Sigma_i \cap B_1(0)$ (that is the graph of the null function).

Now we observe that for every $(n-2)$-plane $T\in\mathcal{G}^n _{n-2}$, every $2$-plane $L\in \mathcal{G}^n _2$ with $\mathrm{d}_{\mathcal{G}}(L,T^\perp)\leq 1/4$ and every $h\in \R^n$ with $|h|\leq 1/8$, it holds that $L+h$ intersects $T$ transversally in a single point $z$, and $|z|<1/4$. Indeed, since $z\in T\cap (L+h)$, we have that
$$|z|=|z-P_{T^{\perp}}(z)|\leq |z-P_{L}(z)|+|P_{L}(z)-P_{T^{\perp}}(z)|\leq 1/8 + |z|/4,$$
from which we deduce that $|z|\leq 1/6<1/4$.

Therefore, we deduce that eventually
$$\frac{L_k +h_k-x_k}{r_k}\cap \frac{\Sigma-x_k}{r_k} \cap B_1(0)$$
contains exactly one point $z_k$, at which $(L_k+h_k-x_k)/r_k$ and $(\Sigma-x_k)/r_k$ intersect transversally, and, moreover, $|z_k|<1/4$.

As a consequence, when $k$ is large enough, $(L_k+h_k)\cap \Sigma \cap B_{r_k}(x_k)$ contains exactly one point $y_k=x_k +r_k z_k$, which also belongs to $B_{r_k/4}(x_k)$ and at which $L_k+h_k$ and $\Sigma$ are transversal, so none of the previous three options can be true. This concludes the proof.
\end{proof}

We are finally ready to prove that the fractional mass is locally positive around $\Sigma$. The precise statement is the following.

\begin{prop}\label{maindegineq1} Let $\Sigma\subset \R^n$ be as in \eqref{defn:Sigma}, let $s_0\in (0,1)$, let $s\in [s_0,1)$ and let $u \in \mathfrak{F}_{\hspace{-1pt} s}^{\hspace{.9pt} w}  (\Sigma)$. Then there exists a radius $r_0 = r_0(\Sigma) >0 $ and a constant $c=c(n,s_0)>0$ such that, for every $i\in \{1,\dots,m\}$, every $x \in \Sigma_i$ and $r\le r_0$ there holds
\begin{equation*}
    [u]_{H^\frac{1+s}{2}(B_r(x))}^2 \ge \frac{c}{(1-s)^2}d_i r^{n-1-s} .
\end{equation*}
\end{prop}

\begin{proof}
    By Corollary~\ref{cor:slicing_local} with $k=2$ and $\alpha_0=1/2$ there exists a constant $c(n)>0$, depending only on $n$ such that for every $r>0$ it holds that
    \begin{equation}
        [u]^2_{H^\frac{1+s}{2}(B_r(x))} \ge c(n) \int_{\mathcal{G}_2^n} dL \int_{L^\perp} [u]^2_{H^\frac{1+s}{2}((L+h)\cap B_r(x))} \, d\mathcal{H}^{n-2}(h) . \label{eq: non sharp lem 1}
    \end{equation}

Let $r_0$ be as in Lemma~\ref{lemma:r0}. Then, for every $(r,x,L,h)$ as in the Lemma~\ref{lemma:r0}, with $x\in \Sigma_i$, we know that $(L+h)\cap B_r(x)$ is a two-dimensional disk of radius $\sqrt{r^2-|h-x_{L^\perp}|^2} \ge r \sqrt{1-1/64}  > r/2$ intersecting $\Sigma$ transversally in a single point $y=y(r,x,L,h) \in B_{r/4}(x)\cap \Sigma_i$.
    
In particular, a two-dimensional disk $\mathrm{D}^{L+h} _{r/4} (y)$ in the plane $L+h$, with radius equal to $r/4$ and center at $y$, is contained in $(L+h)\cap B_r(x)$. Using Definition~\ref{def:link_I}, we see that all the squares $Q^{L+h}_{\theta,t}(y') \in \mathcal{S}_n $ that are contained in $\mathrm{D}^{L+h} _{r/4} (y)$ and for which $y$ lies inside $Q^{L+h}_{\theta,t}(y')$ in the plane $L+h$, satisfy $\abs{\link(Q^{L+h}_{\theta,t}(y'),\Sigma_i)}=1$ and $\link(Q^{L+h}_{\theta,t}(y'),\Sigma_j)=0$ for $j\neq i$. 
    
Therefore, since $u \in \mathfrak{F}_{\hspace{-1pt} s}^{\hspace{.9pt} w}  (\Sigma)$, from Lemma \ref{lem: cont trace a.e.} we deduce that $\abs{\deg(u,Q^{L+h}_{\theta,t}(y'))}=d_i $ for almost every square $Q^{L+h}_{\theta,t}(y')\in \mathcal{S}_n$ as above. Hence, for every $(r,x)\in (0,r_0)\times \Sigma$, for almost every $L\in \mathcal{G}_2^n$ and almost every $h\in L^\perp$ as in Lemma~\ref{lemma:r0}, the disk $\mathrm{D}^{L+h}_{r/4}(y)$ fits into the framework of Lemma~\ref{lem: 2d lb squares}, which implies that 
    \begin{equation*}
        [u]^2_{H^{\frac{1+s}{2}}((L+h)\cap B_r(x))}\geq [u]^2_{H^{\frac{1+s}{2}}(\mathrm{D}_{r/4} ^{L+h}(y))} \ge \frac{c(s_0)}{(1-s)^2} d_i r^{1-s} 
    \end{equation*} 
for every $s\in (0,s_0)$.
    
Then, we can continue from \eqref{eq: non sharp lem 1} to get
   \begin{align*}
        [u]^2_{H^{\frac{1+s}{2}}(B_r(x))} & \ge c(n) \int_{\mathcal{G}_2^n \cap \{ \mathrm{d}_{\mathcal{G}} (L,(T_x \Sigma_i)^\perp) \le 1/4 \} } dL \int_{L^\perp \cap \{|h-x_{L^\perp}| \le r/8\}} \frac{c(s_0)}{(1-s)^2}d_i r^{1-s} \, d\mathcal{H}^{n-2}(h) \\
        & \ge \frac{c(n)c(s_0)}{(1-s)^2}d_i r^{1-s} \cdot {\rm Vol}_{\mathcal{G}^n_2} \bigl(\{ \mathrm{d}_{\mathcal{G}} (L, (T_x\Sigma_i) ^\perp) \le 1/4  \} \bigr) \cdot \mathcal{H}^{n-2}\bigl(B_{r/8}(x_{L^\perp})\cap L^\perp\bigr)
        \\
        &= \frac{c(n,s_0)}{(1-s)^2}d_i r^{n-1-s} 
    \end{align*}
    where $c(n,s_0)$ is a positive constant depending only on $n$ and $s_0$.
\end{proof}

\section{Proof of the main result}
\label{sec:main-results}

Let $\Phi_i:=\Phi_{\Sigma_i} \colon \Sigma_i \times [-\delta_0, \delta_0]^2 \to \R^n$ be the map defined by~\eqref{defn:Phi_Sigma}, with $\Sigma=\Sigma_i$. For any $\delta \le \delta_0$, we define the open set $\mathcal{U}_{i,\delta}:=\Phi_i(\Sigma_i \times (-\delta,\delta)^2)$. Clearly, we can assume that $\delta_0$ is sufficiently small to ensure that both~\eqref{eq:delta-0-uniform} holds and $\mathcal{U}_{i,\delta}\cap \mathcal{U}_{j,\delta}=\varnothing$ for every $i\neq j$ and every $\delta \leq \delta_0$. Moreover, we will use the symbol $\Psi_i = (\Psi_{i,1},\Psi_{i,2})$ to denote the inverse map of $\Phi_i$, which is defined on $\overline{\mathcal{U}}_{i,\delta_0}$.

As we already observed, $\mathrm{d}\Phi_i(\sigma,0)$ is an isometry for any $\sigma \in \Sigma_i$, see~\eqref{def:collar} and~\eqref{eq:isom-p2}, therefore for any $\eta > 0$, there exists $\delta(\eta) \in (0,\delta_0)$ such that for any $W \subset \Sigma_i$ with $\mathrm{diam}(W) \le \delta(\eta)$ it holds that
    \begin{equation} \label{eq:alm-isom-below}
        (1-\eta)^2 \big( \mathrm{d}_{\Sigma_i}(\sigma,\tau)^2 + \abs{p-q}^2 \big)
        \le
        \abs{\Phi_i(\sigma,p)-\Phi_i(\tau,q)}^2 \le (1+\eta)^2 \big( \mathrm{d}_{\Sigma_i}(\sigma,\tau)^2 + \abs{p-q}^2 \big),
    \end{equation}
for any $\sigma, \tau \in W$, for any $p, q \in (-\delta(\eta),\delta(\eta))^2$. From the previous estimates we deduce that
    \begin{equation}
        \label{eq:jac-est-1}
         (1 - \eta)^n
         \le
         \abs{\det \mathrm{d} \Phi_i(\sigma, p)}
         \le
         (1+\eta)^n,
         \qquad \forall (\sigma, p) \in \Sigma_i \times (-\delta(\eta),\delta(\eta))^2.
    \end{equation}

Now, for any $\eta > 0$ and for any $\delta > 0$, there exist $N=N(\eta,\delta)$ and a finite disjoint collection $W_{i,1},\dots,W_{i,N}$ of Lipschitz open subsets of $\Sigma_i$ with the following properties:
    \begin{itemize}
        \item[a)] for any $i\in \{1,\dots, N\}$, it holds that $\mathrm{diam}(W_i) \le \delta$, and $\bigcup_{j=1}^N \overline{W}_j = \Sigma_i$,
        \item[b)] for any $j\in \{1,\dots, N\}$, there exists a Lipschitz (open) domain $A_{i,j} \subset \R^{n-2}$ and a bi-Lipschitz diffeomorphism $\varphi_{i,j} \colon A_{i,j} \to W_{i,j}$ such that
        \begin{equation} \label{eq:bi-Lip}
            (1-\eta) \abs{z-w} \le \mathrm{d}_{\Sigma_i}(\varphi_{i,j}(z),\varphi_{i,j}(w)) \le (1+\eta) \abs{z-w}, \qquad \forall z,w \in A_{i,j},
        \end{equation}
        and, as a consequence,
        \begin{equation}
            \label{eq:jacobians}
            (1-\eta)^{n-2} \le \abs{\mathrm{det} \, \mathrm{d} \varphi_{i,j}(z)} \le (1+\eta)^{n-2}, \qquad \forall z \in A_{i,j}.
        \end{equation}
    \end{itemize}
For the existence of such a collection, see for example~\cite[Subsection~4.2]{Lio-Mar-Nev}.

\subsection{Lower bound (proof of Proposition~\ref{prop:Gamma-liminf})}

This section is devoted to the proof of Proposition~\ref{prop:Gamma-liminf}, which provides the liminf inequality for Theorem~\ref{energy conv teo}. The main ingredient for the proof is the following flat version of the estimate, for maps weakly linking with a $(n-2)$-plane in a bounded region.

\begin{prop}\label{prop:local-liminf}
Let $n\geq 3$ be an integer. Let $E\subset \R^{n-2}$ be a bounded open set with Lipschitz boundary, and let $\delta \in (0,1)$. Let us set
    
    \[
        F:=E\times (-\delta,\delta)^2
        \qquad\mbox{and}\qquad
        F ^{*} :=F \setminus (E \times \{(0,0)\} ).
    \]

We also consider for every $\ell \in (0,\delta)$ and every $\sigma \in E$ the closed curve
    \[
        \gamma_{\sigma,\ell}:= \{\sigma\}\times \partial ([-\ell,\ell]^2) \subset F ^{*}.
    \]
    
Moreover, for every $s\in (0,1)$ let $M_{s} \subset F$ be a relatively closed, $(n-1)$-rectifiable set, and let $w_s\in H^{\frac{1+s}{2}}(F; \Sp^1)$ be a function.

Let also $d\in \N^+$ be a positive integer, and let us assume that for every $s\in (0,1)$ there exists a sequence of functions $\{w_{s,k}\} \subset C^1(F^{*}\setminus M_s; \Sp^1)\cap H^{\frac{1+s}{2}}(F; \Sp^1)$ such that $w_{s,k} \to w_s $ in $\mathbb{X}_s(F)$ and $\abs{\deg(w_{s,k},\gamma)}=d$ for every $k$ and for every curve $\gamma \subset F^{*}\setminus M_{s}$ that is homotopic to some $\gamma_{\sigma,\ell}$ in $F^{*}$. Finally, let us assume that $\mathcal{H}^{n-1}(M_{s})\to 0$ as $s\to 1^{-}$. Then, it turns out that
\begin{equation}\label{th:local-liminf}
\liminf_{s\to 1^{-}} \hspace{0.03cm} (1-s)^2 [w_s]_{H^{\frac{1+s}{2}}(F)} ^2 \geq \frac{2\pi\omega_{n-1}}{n} d \cdot \mathcal{H}^{n-2}(E).
\end{equation}



\end{prop}

Let us first show that Proposition~\ref{prop:Gamma-liminf} follows from Proposition~\ref{prop:local-liminf}.

\begin{proof}[Proof of Proposition~\ref{prop:Gamma-liminf}]
For every $i\in \{1,\dots,m\}$ let us set $\Omega_i:=\Omega\cap \mathcal{U}_{i,\delta_0}$. Then we have that
    \[
        [u]^2_{H^{\frac{1+s}{2}}(\Omega)}\geq\sum_{i=1} ^{m} [u]^2_{H^{\frac{1+s}{2}}(\Omega_i)},
    \]
so~\eqref{eq:liminf-ineq} follows if we prove that
    \begin{equation}\label{eq:liminf_connected}
        \liminf_{s\to 1^{-}} \hspace{0.03cm} (1-s)^2 [u_s]^2_{H^{\frac{1+s}{2}}(\Omega_i)} \geq \frac{2\pi\omega_{n-1}}{n} d_i \mathcal{H}^{n-2}(\Sigma_i\cap \Omega),
    \end{equation}
for every $i\in \{1,\dots,m\}$.

For every $s\in(0,1)$ let $\{u_{s,k}\}\subset \mathfrak{F}_s(\Sigma_s)$ be a sequence converging to $u_s$ in $\mathbb{X}_s(\R^n)$ as $k\to +\infty$.

We fix $i\in\{1,\dots,m\}$ and $\eta>0$ and we let $\delta(\eta) \in (0,1)$ as in~\eqref{eq:alm-isom-below}. In addition, let $W_{i,1},\dots,W_{i,N}$ be the collection of open sets associated to $\eta$ and $\delta(\eta)$ as in the beginning of Section~\ref{sec:main-results}. We fix $\delta \in (0,\delta(\eta))$, and for every $j\in\{1,\dots,N\}$ we consider the set
    \[
        W_{i,j} ^{\delta}:=\Big\{\sigma \in W_{i,j} : \Phi_i\big(\{\sigma\}\times [-\delta,\delta]^2 \big) \subset \Omega_i \Big\}.
    \]

It is clear that
\begin{equation}\label{est:low2}
[u_s]_{H^{\frac{1+s}{2}}(\Omega_i)} ^{2}\geq  \sum_{j=1} ^{N} [u_s]_{H^{\frac{1+s}{2}}(\Phi_i(W_{i,j} ^{\delta} \times (-\delta,\delta)^2))} ^{2}.
\end{equation}
Combining the change of variables $(x,y)=(\Phi_i(\sigma,p), \Phi_i(\tau,q))$ with~\eqref{eq:alm-isom-below} and~\eqref{eq:jac-est-1}, and then the change of variables $(\sigma,\tau)=(\varphi_{i,j}(z), \varphi_{i,j}(w))$ with~\eqref{eq:bi-Lip} and~\eqref{eq:jacobians}, we deduce that
    \begin{equation}
         \label{est:low3}
        [u_s]_{H^{\frac{1+s}{2}}(\Phi_i(W_{i,j} ^{\delta} \times (-\delta,\delta)^2))} ^{2}
        \geq
        \frac{(1-\eta)^{2n+2(n-2)}}{(1+\eta)^{2(n+1+s)}}[\hat{u}_{s,i,j}]_{H^{\frac{1+s}{2}}(A_{i,j} ^{\delta} \times (-\delta,\delta)^2)} ^{2},
\end{equation}
where $\hat{u}_{s,i,j}:=u_s\circ \Phi_i\circ (\varphi_{i,j}\times \mathrm{Id})$ and $A_{i,j} ^{\delta}:=\varphi_{i,j}^{-1}(W_{i,j}^{\delta})$.

We also observe that \eqref{eq:bi-Lip} yields
\begin{equation}\label{eq:low4}
\mathcal{H}^{n-2}(A_{i,j} ^{\delta})\geq \frac{1}{(1+\eta)^{n-2}} \mathcal{H}^{n-2}(W_{i,j} ^{\delta}) \qquad \forall j\in \{1,\dots,N\}.
\end{equation}
Since $W_{i,j}^{\delta} \to W_{i,j} \cap \Omega_i$, as $\delta \to 0^+$, we have that
    \begin{equation}\label{eq:low5}
        \lim_{\delta\to 0^{+}} \sum_{j=1}^{N} \mathcal{H}^{n-2}(W_{i,j}^{\delta})
        =
        \sum_{j=1}^{N}\mathcal{H}^{n-2}(W_{i,j} \cap \Omega_i)=\mathcal{H}^{n-2}(\Sigma_i\cap\Omega).
    \end{equation}

Now, by Remark~\ref{rem:lip-chain}, for every $s\in (0,1)$ we can find an integral Lipschitz chain $T_s$ such that $\partial T_s = \Sigma_s -\Sigma$ in the sense of integral currents and
$$\M(T_s) \leq 2\F(\Sigma_s -\Sigma).$$

Let $N_s$ be the support of $T_s$, which is a closed $(n-1)$-rectifiable set, because it is the Lipschitz image of a finite union of 
$(n-1)$-dimensional polyhedra. Moreover, we have that
$$\mathcal{H}^{n-1}(N_s)\leq \M(T_s)\leq 2\F(\Sigma_s -\Sigma),$$
and hence $\mathcal{H}^{n-1}(N_s)\to 0$ as $s\to 1^-$.

Let us set $N_s':=N_s\cup \Sigma \cup \Sigma_s$. Since $\Sigma$ and $\Sigma_s$ are closed and $(n-2)$-dimensional, we have that $N_s'$ is still a closed $(n-1)$-rectifiable set and that $\mathcal{H}^{n-1}(N_s ')\to 0$ as $s\to 1^-$.

We claim that for every closed curve $\gamma\subset \R^n\setminus N_s '$ it turns out that
    \begin{equation} \label{eq:DEG1}
        d_1\link(\gamma,\Sigma_1)+\dots+d_m\link(\gamma,\Sigma_m)= d_{s,1}\link(\gamma,\Sigma_{s,1})+\dots+d_{s,m_s}\link(\gamma,\Sigma_{s,m_s}).
    \end{equation}

Indeed, we know that $\Sigma_s=\Sigma+\partial T_s$ as currents and the support of $T_s$ is contained in $\R^n\setminus \gamma$, so $[\Sigma_s]=[\Sigma]$ in $H_{n-2}(\R^n\setminus\gamma)$, because the homology of integral currents in $\R^n\setminus \gamma$ is isomorphic to the singular homology (see~\cite[Theorem~5.11]{1960-FF}). Therefore, our claim follows using Definition~\ref{def:link_H} to compute the linking numbers, because
$$[\gamma]^*(d_1[\Sigma_1]+\dots +d_m[\Sigma_m])
=[\gamma]^*([\Sigma])=[\gamma]^*([\Sigma_s])=
[\gamma]^*(d_{s,1}[\Sigma_{s,1}]+\dots +d_{s,m}[\Sigma_{s,m}]).$$

As a consequence, for every $i \in \{1,\dots,m\}$ and every $j\in\{1,\dots,N\}$ we are in the framework of Proposition~\ref{prop:local-liminf} with
    \begin{gather*}
        E:=A_{i,j}^{\delta},
        \quad
        M_s:=(\Phi_i\circ (\varphi_{i,j}\times \mathrm{Id}))^{-1}(N_s'), \\[1ex]
        w_{s,k}:=u_{s,k}\circ\Phi_i\circ (\varphi_{i,j}\times \mathrm{Id}),
        \quad
         w_s:=u_{s}\circ\Phi_i\circ (\varphi_{i,j}\times \mathrm{Id}), 
        \quad d:=d_i. 
    \end{gather*}

Indeed, let us check that the assumptions are satisfied. We observe that $\mathcal{H}^{n-1}(M_s)\to 0$ because $\mathcal{H}^{n-1}(N_s')\to 0$, and that $w_{s,k}\to w_s$ in $\mathbb{X}_s(A_{i,j}^{\delta} \times (-\delta,\delta)^2)$, because $u_{s,k}\to u_s$ in $\mathbb{X}_s(\R^n)$ as $k\to +\infty$. We notice that $w_{s,k}$ is $C^1$ outside $(\Phi_i\circ (\varphi_{i,j}\times \mathrm{Id}))^{-1}(\Sigma_s)$, so in particular it is $C^1$ outside $M_s$.

Moreover, if $\gamma\subset (A_{i,j}^{\delta} \times (-\delta,\delta)^2 ) \setminus M_s$ is a curve that is homotopic to some $\gamma_{\sigma,\ell}$ in $A_{i,j} ^{\delta}\times((-\delta,\delta)^2\setminus \{(0,0)\})$, then the curve $\hat{\gamma}:= (\Phi_i \circ (\varphi_{i,j}\times \mathrm{Id}))(\gamma) \subset \R^n\setminus N_s '$ is homotopic to the curve $\hat{\gamma}_{\sigma,\ell}:= (\Phi_i \circ (\varphi_{i,j}\times \mathrm{Id}))(\gamma_{\sigma,\ell})$ in $\mathcal{U}_{i,\delta} \setminus \Sigma_i$, thus, in particular in $\R^n \setminus \Sigma$. We obtain that
    \begin{align}
        \notag
        \abs{d_{s,1}\link(\hat{\gamma},\Sigma_{s,1})+\dots+d_{s,m_s}\link(\hat{\gamma},\Sigma_{s,m_s})}
        & =\abs{d_1\link(\hat{\gamma},\Sigma_1)+\dots+d_m\link(\hat{\gamma},\Sigma_m)} \\[1ex]
        \notag
        & = \abs{d_1\link(\hat{\gamma}_{\sigma,\ell},\Sigma_1)+\dots+d_m\link(\hat{\gamma}_{\sigma,\ell},\Sigma_m)} \\[1ex]
        \label{eq:DEG2}
        & = d_i.
    \end{align}
Here, the first identity is a consequence of~\eqref{eq:DEG1}, the second one follows from the homotopy invariance of the linking number, and the third one follows from Definition~\ref{def:link_I}, because we know that $(\Phi_i\circ (\varphi_{i,j} \times \mathrm{Id}))(\{\sigma\}\times [-\ell,\ell]^2)$ intersects $\Sigma_i$ transversally in the point $(\Phi_i \circ (\varphi_{i,j} \times \mathrm{Id}))(\sigma,0)$, and does not intersect $\Sigma_{i'}$ for any $i'\neq i$.

Therefore, we deduce that for every such curve $\gamma$ it holds that
    \[
        \abs{\deg(w_{s,k},\gamma)}
        =
        \abs{\deg(u_{s,k}, \hat{\gamma})}
        =
        \abs{d_{s,1}\link(\hat{\gamma},\Sigma_{s,1})+\dots+d_{s,m_s}\link(\hat{\gamma},\Sigma_{s,m_s})}
        =
        d_i,
    \]
where the first identity is simply due to the fact that $w_{s,k} \circ \gamma = u_{s,k} \circ \hat{\gamma}$, the second one holds because $u_{s,k}\in \mathfrak{F}_s(\Sigma_s)$, and the third one is exactly~\eqref{eq:DEG2}. In particular, all the assumptions of Proposition~\ref{prop:local-liminf} hold true.

Recalling that $w_s:=\hat{u}_{s,i,j}$, then \eqref{th:local-liminf} yields
    \[
        \liminf_{s\to 1^{-}} \hspace{0.03cm} (1-s)^2 [\hat{u}_{s,i,j}]_{H^{\frac{1+s}{2}}(A_{i,j} ^{\delta} \times (-\delta,\delta)^2)} ^{2}\geq \frac{2\pi\omega_{n-1}}{n}d_i \mathcal{H}^{n-2}(A_{i,j}^\delta).
    \]

Combining this estimate with \eqref{est:low2},~\eqref{est:low3} and~\eqref{eq:low4} we obtain that
    \[
        \liminf_{s\to 1^{-}} \hspace{0.03cm} (1-s)^2 [u_s]_{H^{\frac{1+s}{2}}(\Omega_i)} ^{2}\geq \frac{(1-\eta)^{4(n-1)}}{(1+\eta)^{3n+2}} \sum_{j=1} ^{N} \frac{2\pi\omega_{n-1}}{n} d_i \mathcal{H}^{n-2}(W_{i,j}^\delta).
    \]

Letting $\delta \to 0^{+}$ and $\eta\to 0^{+}$, and taking into account~\eqref{eq:low5}, we obtain \eqref{eq:liminf_connected}.
\end{proof}

We now focus on the proof of Proposition~\ref{prop:local-liminf}, which is quite involved and requires several preliminary results. The first step consists in writing our fractional energy in terms of the discrete functionals studied in~\cite{2009-ARMA-AC}, as it was done in~\cite{2023-Solci} in a similar context.

To this end, let us introduce some notation. Following~\cite[Section~4]{2023-Solci} we consider the set $V_n$ of all orthonormal bases of $\R^n$, namely the set of $\nu \in (\Sp^{n-1})^n$ such that $\nu_i \cdot \nu_j =0$ for every $i\neq j$. This set can be identified with $O(n)$ (which is a Lie group of dimension $n(n-1)/2$) by associating to every basis $\nu\in V_n$ the rotation $R_\nu \in O(n)$ bringing the canonical basis into $\nu$.

Therefore it is natural to endow $V_n$ with the restriction of the $(n(n-1)/2)$-dimensional Hausdorff measure. 


Moreover, for every non-negative measurable function $f \colon \Sp^{n-1}\to [0,+\infty)$ we have the following formula
    \begin{equation}\label{eq:Sphere-Vn}
        \int_{\Sp^{n-1}} f(\theta)\,d\theta
        =
        \frac{\omega_{n-1}}{n}\fint_{V_n} \sum_{i=1} ^{n} f(\nu_i)\,d\nu,
    \end{equation}
which is a direct consequence of the fact that
the push-forward of the measure on $V_n$ through the map $\nu \mapsto \nu_i$ is a constant multiple of the standard measure on the sphere, with the constant given by $\mathcal{H}^{n(n-1)/2}(V_n)/\omega_{n-1}$.

Now let us denote by $\Cube:=(0,1)^n$ the unit cube in $\R^n$. For an open set $\Omega\subset \R^n$ and a triple $(\eps,\nu,z)\in (0,1)\times V_n \times \Cube$, we consider the lattice
$$\Z_{\eps,\nu,z} ^n (\Omega):=R_\nu(\eps\Z^n +\eps z) \cap \Omega,$$
the set of cubes into which this lattice divides the domain
    \[
        \mathscr{C}_{\eps,\nu,z}(\Omega):=\Big\{j+R_\nu(\eps\Cube) : j\in \Z_{\eps,\nu,z}^{n} (\Omega) \ \text{and} \ j+R_\nu(\eps\Cube)\Subset \Omega \Big\},
    \]
and the set of all their one-dimensional (closed) edges $\mathscr{E}_{\eps,\nu,z}(\Omega)$.


 For every function $w \colon \Z_{\eps,\nu,z} ^n(\Omega)\to \Sp^1$ we consider the following discrete energy
    \[
        E_{\eps,\nu,z}(w,\Omega):=\frac{\eps^{n-2}}{2\abs{\log \eps}} \sum_{j \in \Z_{\eps,\nu,z}^n(\Omega)}
        \sum_{i=1}^{n} \abs{w(j +\eps \nu_i) -w(j)}^2.
    \]

For every $r\in (0,1)$ and every $s\in (0,1)$, we set $r_s:=r^\frac{2}{1-s}$, and for every $\eta \in (0,\delta/4)$ we fix $s_0(\eta)\in (0,1)$ in such a way that
\begin{equation}\label{eq:defn_s0}
(3\sqrt{n}+1)r_s\leq \eta \qquad\qquad \forall (r,s) \in [0,1-\eta]\times [s_0(\eta),1).
\end{equation}

Finally, for every small real number $\eps>0$, we set
    \[
        E_\eps:=\{x\in E : \mathrm{dist} (x,\R^{n-2}\setminus E)>\eps\},
        \qquad \text{and}\qquad
        F_\eps:=\{x\in F : \mathrm{dist} (x,\R^n\setminus F)>\eps\}.
    \]

The next lemma shows that we can estimate the fractional energy of a function $w:F\to \Sp^1$ in terms of an average of the discrete functional computed on the restriction of $w$ to some lattices, as the step, the orientation and the position of the lattice vary.

\begin{lemma}
For every $\eta \in (0,1)$, every $s\in (s_0(\eta),1)$ and every $w\in H^{\frac{1+s}{2}}(F; \Sp^1)$ it turns out that
    \begin{equation}
        \label{eq:discretization_frac_norm}
        (1-s)^2 [w]_{H^{\frac{1+s}{2}}(F)} ^{2} \geq \frac{\omega_{n-1}}{n} \int_{0}^{1-\eta} 8r\abs{\log r} \, dr \fint_{V_n}d\nu \int_{\Cube}E_{r_s,\nu,z}(w^{r_s,\nu,z},F_{\eta- \sqrt{n}r_s})\,dz,
    \end{equation}
where for every $(\eps,\nu,z)\in (0,1)\times V_n\times\Cube$ the function $w^{\eps,\nu,z}:\Z_{\eps,\nu,z} ^{n} (F)\to\Sp^1$ is the restriction of $w$ to $\Z_{\eps,\nu,z} ^{n} (F)$.
\end{lemma}

\begin{proof}
With the change of variable $y=x+r_s \theta$, we obtain that
$$(1-s)^2 [w]_{H^{\frac{1+s}{2}}(F)} ^2 \geq
\int_{0} ^{1-\eta} 2(1-s) r \,dr \int_{\Sp^{n-1}} d\theta \int_{F_{\eta-3\sqrt{n}r_s}} \frac{\abs{w(x+r_s\theta) -w(x)}^2}{r_s ^{2}} \,dx,$$
where the inequality follows just from the fact that we are reducing the domain of integration, because of~\eqref{eq:defn_s0}. Observing that $1-s=2\abs{\log r}/\abs{\log r_s}$ and exploiting~\eqref{eq:Sphere-Vn}, we can rewrite the right-hand side as
$$\frac{\omega_{n-1}}{n}\int_{0} ^{1-\eta} 8r\abs{\log r} \,dr \fint_{V_n} d\nu \frac{1}{2\abs{\log r_s}} \int_{F_{\eta-3\sqrt{n}r_s}}  \sum_{i=1}^{n} \frac{\abs{w(x+r_s\nu_i) -w(x)}^2}{r_s ^{2}} \,dx.$$

Now, we note that
    \begin{eqnarray*}
        \lefteqn{\hspace{-3em}\frac{1}{2\abs{\log r_s}}\int_{F_{\eta-3\sqrt{n}r_s}}  \sum_{i=1}^{n} \frac{\abs{w(x+r_s\nu_i) -w(x)}^2}{r_s ^{2}} \,dx}\\[0.5ex]
        \qquad &\geq& \frac{1}{2\abs{\log r_s}} \sum_{j\in \Z_{r_s,\nu,0} ^n(F_{\eta-2\sqrt{n}r_s})}\int_{j+r_sR_\nu(\Cube)} \sum_{i=1}^{n}\frac{\abs{w(x+r_s\nu_i) -w(x)}^2}{r_s ^{2}}\,dx\\[0.5ex]
        &=&\frac{r_s ^{n-2}}{2\abs{\log r_s}} \sum_{j\in \Z_{r_s,\nu,0} ^n(F_{\eta-2\sqrt{n}r_s})} \int_{\Cube} \sum_{i=1}^{n} \abs{w(j+r_s R_\nu(z)+r_s\nu_i) -w(j+r_s R_\nu(z))}^2\,dz\\[0.5ex]
        &\geq& \frac{r_s ^{n-2}}{2\abs{\log r_s}} \int_{\Cube} \sum_{j\in \Z_{r_s,\nu,z} ^n(F_{\eta-\sqrt{n}r_s})}\sum_{i=1}^{n} \abs{w(j+r_s\nu_i) -w(j)}^2\,dz\\[0.5ex]
        &=&\int_{\Cube}E_{r_s,\nu,z}(w^{r_s,\nu,z},F_{\eta-\sqrt{n}r_s})\,dz,
    \end{eqnarray*}
where for the first inequality we used that for every $j \in \Z_{r_s,\nu,0}^{n} (F_{\eta-2\sqrt{n}r_s})$ the cube $j + r_s R_\nu(\Cube)$ is contained in $F_{\eta-3\sqrt{n}r_s}$, while for the second inequality we used that $\Z_{r_s,\nu,z}^{n} (F_{\eta-\sqrt{n}r_s})$ is contained in $ \Z_{r_s,\nu,0}^{n} (F_{\eta-2\sqrt{n}r_s}) + r_s R_\nu(z)$. Putting everything together we obtain~\eqref{eq:discretization_frac_norm}.
\end{proof}

At this point, following \cite{2023-Solci}, one might try to exploit the compactness statement and the $\Gamma$-liminf estimate of \cite[Theorem~3]{2009-ARMA-AC} to deduce that for almost every $(r,\nu,z)$, up to subsequences, it holds that $w_s ^{r_s,\nu,z}\to \Sigma^{r,\nu,z}$ (in the sense used in \cite{2009-ARMA-AC}) for some $(n-2)$-dimensional integral current $\Sigma^{r,\nu,z}$ and
$$\liminf_{s\to 1^{-}} \hspace{0.03cm} E_{r_s,\nu,z} (w_s ^{r_s,\nu,z}, F_{\eta-\sqrt{n}r_s}) \geq \pi \M(\Sigma^{r,\nu,z}).$$

At this point, one might be tempted to infer the assumption on the degree implies that the rectifiable set supporting the current $\Sigma^{r,\nu,z}$ must contain the set $E_{\eta}\times \{(0,0)\}$, and hence $\M(\Sigma^{r,\nu,z})\geq \mathcal{H}^{n-2}(E_{\eta})$.

Unfortunately, this is not simple, and actually it is not even true that in general $\M(\Sigma^{r,\nu,z})$ is comparable with $\mathcal{H}^{n-2}(E)$ for almost every choice of the parameters. Indeed, in our case, the argument that we have just outlined does not work for two reasons.
\begin{itemize}
\item The first one is that the notion of convergence used in \cite{2009-ARMA-AC,2023-Solci} involves the Jacobian of some piecewise affine interpolations of the discrete functions, and in general it is not true that these interpolations still satisfy some condition on the degree. Actually, the interpolations take values in the unit disk, so they could also vanish, in which case the degree might not even be defined, or it could also happen that $w_s$ is constant on some of the lattices $\Z_{r_s,\nu,z} ^n(F)$, in which case the interpolation would be a constant function, and its degree on any curve would be equal to zero.

\item The second issue is that, even is one knows that a sequence of maps satisfies our assumption on the degree, it is not immediate to deduce that any limit point (in the sense of currents) of their Jacobians has positive mass in $F$, because the convergence allows cancellations, and some mass could also concentrate on the boundary.
\end{itemize}

We overcome the first issue by exploiting the one-dimensional fractional Sobolev embedding of Lemma \ref{lemma:frac_embedd} on the one-dimensional edges in $\mathscr{E}_{r_s,\nu,z}(F)$, in order to find sufficiently many of these edges on which we can control the oscillation of $w_s$, and hence the difference between $w_s$ and its interpolations.

As for the second one, we first pass from the discrete energy to the Ginzburg-Landau energy of the interpolations, as in \cite{2009-ARMA-AC}, then we exploit an estimate of \cite{1999-SIAM-Jerrard,2005-Indiana-ABO} that relates the Ginzburg-Landau energy of a function defined in a two-dimensional square directly to its degree on the boundary, instead of the mass of the weak limits, up to a boundary remainder term.

In the end, for sufficiently many values of the parameters $(r,\nu,z)$ we manage to find sufficiently many points $\sigma\in E$ such that there exists $\ell_\sigma\in (0,\delta)$ for which we can apply the estimate of \cite{2005-Indiana-ABO} on $\{\sigma\}\times[-\ell_\sigma,\ell_\sigma]^2$ and also control the remainder term.

\smallskip

In order to make this reasoning precise, the first step consists in slicing the fractional norm along the lattices $\Z^n_{\ep,\nu,z}(\Omega)$, and applying Lemma~\ref{lemma:frac_embedd} in each edge $I\in \mathscr{E}_{\ep,\nu,z}(\Omega)$.

\begin{lemma}\label{lemma:holder_edges}
For every $s_0\in (0,1)$ there exists a positive constant $c_{n,s_0}>0$ such that for every open set $\Omega\subset \R^n$, every $s\in (s_0,1)$, every function $w\in H^{\frac{1+s}{2}}(\Omega)$ and every $\ep>0$ it holds that
$$\fint_{V_n} d\nu \int_{\Cube} dz \sum_{I\in \mathscr{E}_{\ep,\nu,z}(\Omega)} \ep^{n-1-s} \osc(w,I) ^2 \leq c_{n,s_0} (1-s) [w]_{H^{\frac{1+s}{2}}(\Omega)} ^2 $$
\end{lemma}

\begin{proof}
From the usual slicing formula \eqref{eq: slic for} we obtain that
$$[w]_{H^{\frac{1+s}{2}}(\Omega)} ^2 = \frac{1}{2}\int_{\Sp^{n-1}}d\theta \int_{\theta^\perp}  [w]_{H^{\frac{1+s}{2}}(\Omega_{\theta,z'})} ^2 \,dz',$$
where $\Omega_{\theta,z'}$ denotes the intersection of the line $z'+ \R\theta$ with $\Omega$. Then from~\eqref{eq:Sphere-Vn} we deduce that
$$[w]_{H^{\frac{1+s}{2}}(\Omega)} ^2
=\frac{\omega_{n-1}}{2n} \fint_{V_n} d\nu \sum_{i=1} ^{n} \int_{\nu_{i} ^\perp}  [w]_{H^{\frac{1+s}{2}}(\Omega_{\nu_{i},z'})} ^2 \,dz'.$$

Now let $\{e_1,\dots,e_n\}$ be the canonical basis of $\R^n$, so $\nu_i=R_\nu(e_i)$ for every $\nu\in V_n$ and every $i\in \{1,\dots,n\}$. Then we can rewrite the previous equality as
\begin{eqnarray}
[w]_{H^{\frac{1+s}{2}}(\Omega)} ^2
&=&\frac{\omega_{n-1}}{2n} \fint_{V_n} d\nu \sum_{i=1} ^{n} \int_{R_\nu(e_{i} ^\perp)}  [w]_{H^{\frac{1+s}{2}}(\Omega_{\nu_{i},z'})} ^2 \,dz'\nonumber\\
&=&\frac{\omega_{n-1}}{2n} \fint_{V_n} d\nu \sum_{i=1} ^{n} \int_{e_{i} ^\perp}  [w]_{H^{\frac{1+s}{2}}(\Omega_{\nu_{i},R_\nu(z')})} ^2 \,dz'\nonumber\\
&=&\frac{\omega_{n-1}}{2n} \fint_{V_n} d\nu \sum_{i=1} ^{n} \frac{1}{\ep} \int_{e_{i} ^\perp \times [0,\ep]e_i}  [w]_{H^{\frac{1+s}{2}}(\Omega_{\nu_{i},R_\nu(z)})} ^2 \,dz\nonumber\\
&=&\frac{\omega_{n-1}}{2n} \fint_{V_n} d\nu \sum_{i=1} ^{n} \sum_{j\in \ep\Z^{n}\cap e_{i} ^{\perp}} \frac{1}{\ep}\int_{\ep\Cube}  [w]_{H^{\frac{1+s}{2}}(\Omega_{\nu_{i},R_\nu(j+z)})} ^2 \,dz\nonumber\\
&=&\frac{\omega_{n-1}}{2n} \fint_{V_n} d\nu  \int_{\Cube}dz \sum_{i=1} ^{n} \sum_{j\in \ep\Z^{n}\cap e_{i} ^{\perp}} \ep^{n-1}  [w]_{H^{\frac{1+s}{2}}(\Omega_{\nu_{i},R_\nu(j+\ep z)})} ^2.\label{eq:slicing_edges}
\end{eqnarray}

Since in the last expression for every $(\nu,z)\in V_n \times \Cube$ we are summing the fractional seminorms on a collection of lines that contains all the edges of the cubes in $\mathscr{C}_{\ep,\nu,z}(\Omega)$, and the fractional norm of a function on an interval is larger than the sum of the fractional norms on any family of disjoint subintervals, we finally obtain that
\begin{equation*}
[w]_{H^{\frac{1+s}{2}}(\Omega)} ^2\geq \frac{\omega_{n-1}}{2n} \fint_{V_n} d\nu  \int_{\Cube}dz \sum_{I\in \mathscr{E}_{\ep,\nu,z}(\Omega)} \ep^{n-1}  [w]_{H^{\frac{1+s}{2}}(I)} ^2.
\end{equation*}

Since the length of each edge $I\in \mathscr{E}_{\ep,\nu,z}(\Omega)$ is equal to $\ep$, the conclusion follows from Lemma~\ref{lemma:frac_embedd}.
\end{proof}


In the case where $w$ is the limit in $\mathbb{X}_s(\Omega)$ of a sequence of functions that are continuous on an open subset of $\Omega$ we can use the previous lemma to deduce that $w$ has continuous ``traces'' on certain one-dimensional subsets of $\Omega$. More precisely, the following holds.

\begin{lemma}\label{lemma:N_ep}
Let $s\in (0,1)$, let $\Omega\subset \R^n$ be a bounded open set with Lipschitz boundary, let $M \subset \Omega$ be a (relatively) closed set, let $w \in H^\frac{1+s}{2}(\Omega; \Sp^1)$ be a function, and let $ \{w_k\}\subset C^{0}(\Omega \setminus M; \Sp^1)\cap H^\frac{1+s}{2}(\Omega; \Sp^1) $ be a sequence such that $w_k \to w$ in $\mathbb{X}_s(\Omega)$.

Then, for every $\ep>0$, there exists a negligible set $N_\ep \subset V_n\times \Cube$ such that for every $(\nu,z)\notin N_\ep$ it turns out that $w$ is continuous (namely it agrees $\mathcal{H}^1$ almost everywhere with a continuous function) on every $I\in \mathscr{E}_{\ep,\nu,z}(\Omega)$ that does not intersect $M$, and
\begin{equation}\label{eq:lim_deg}
\liminf_{k\to + \infty} \hspace{0.05cm} \abs{\deg(w_k,\gamma)} \leq \abs{\deg(w,\gamma)} \leq \limsup_{k\to +\infty} \hspace{0.05cm} \abs{\deg(w_k,\gamma)},
\end{equation}
for every closed curve $\gamma$ that is contained in the union of all $I\in \mathscr{E}_{\ep,\nu,z}(\Omega)$ and does not intersect $M$.
\end{lemma}

\begin{proof}
Let us fix $s_0\in (0,s)$. By assumption, we know that $w_k\to w$ strongly in $H^{\frac{1+s_0}{2}}(\Omega;\Sp^1)$. Hence, from Lemma~\ref{lemma:holder_edges} and Fatou's lemma we deduce that
$$\fint_{V_n} d\nu \int_{\Cube} dz \liminf_{k\to +\infty} \sum_{I\in \mathscr{E}_{\ep,\nu,z}(\Omega)} \ep^{n-1-s_0} \osc(w-w_k,I) ^2 \leq \liminf_{k\to +\infty} c_n (1-s_0) [w-w_k]_{H^{\frac{1+s_0}{2}}(\Omega)} ^2 =0.$$

Therefore, there exists a negligible set $N_\ep \subset V_n\times \Cube$ such that for every $(\nu,z)\notin N_\ep$ there is a subsequence $\{w_{k_h}\}$, which depends on $(\nu,z)$, that converges to $w$ in $L^\infty(I)$ for every edge $I\in \mathscr{E}_{\ep,\nu,z}(\Omega)$.

As a consequence, $w$ is continuous on each of these edges that do not intersect $M$ and
$$\abs{\deg(w,\gamma)}=\lim_{h\to +\infty} \abs{\deg(w_{k_h},\gamma)}$$
for every closed curve $\gamma$ contained in the union of all $I\in \mathscr{E}_{\ep,\nu,z}(\Omega)$ not intersecting $M$. This implies~\eqref{eq:lim_deg}.
\end{proof}

With the same slicing technique that we used in Lemma~\ref{lemma:holder_edges}, we also obtain the following estimate, that allows us to control the number of cubes with an edge intersecting $M_{s}$.

\begin{lemma}\label{lemma:M_edges}
There exists a positive constant $c_n>0$ such that for every open set $\Omega\subset \R^n$, every $(n-1)$-rectifiable set $M\subset\Omega$ with $\mathcal{H}^{n-1}(M)<+\infty$, and every $\ep>0$ it turns out that
$$\fint_{V_n} d\nu \int_{\Cube}dz \sum_{I\in \mathscr{E}_{\ep,\nu,z}(\Omega)}\ep^{n-1} \mathcal{H}^0(I\cap M)\leq c_n \mathcal{H}^{n-1}(M).$$
\end{lemma}

\begin{proof}
From Theorem~\ref{teo:crofton} we obtain that
$$\mathcal{H}^{n-1}(M)=\frac{\beta^{0,n}_{1,n-1}}{\omega_{n-1}} \int_{\Sp^{n-1}}d\theta \int_{\theta^\perp} \mathcal{H}^0(\Omega_{\theta,z'} \cap M)\,dz',$$
where $\Omega_{\theta,z'}$ denotes the intersection of the line $z'+ \R\theta$ with $\Omega$.

We conclude arguing as in~\eqref{eq:slicing_edges}.
\end{proof}

As we anticipated above, the next step consists in relating the discrete energies
$$E_{r_s,\nu,z}(w_{s} ^{r_s,\nu,z}, F_{\eta-\sqrt{n}r_s})$$
to the Ginzburg-Landau energy of a piecewise affine interpolation of $w_{s} ^{r_s,\nu,z}$. This was done in \cite{2009-ARMA-AC}, so we just have to introduce some notation to recall the results that we need.

For every $(r,\nu,z)\in (0,1-\eta)\times V_n \times \Cube$ and every $s\in (s_0(\eta),1)$ we denote by $v_{s} ^{r_s,\nu,z}$ the piecewise affine interpolation of $w_{s}$ on the simplices of the Kuhn decomposition of each cube in $\mathscr{C}_{r_s,\nu,z} (F)$.

We refer to \cite{2009-ARMA-AC} for the detailed construction of this interpolation, and we limit to describe it in words. Basically, the idea is to divide the cube $\Cube$ into the $n!$ simplices $T_{\pi}:=\{x_{\pi(1)}<\dots<x_{\pi(n)}\}$, where $\pi$ is a permutation of $\{1,\dots,n\}$. We observe that the vertices of all these simplices are also vertices of $\Cube$.

Thus, we can divide each of the cubes in $\mathscr{C}_{r_s,\nu,z} (F)$ into $n!$ simplices whose vertices lie in the lattice $\Z_{r_s,\nu,z} ^{n} (F)$. In each of these simplices we take the affine function coinciding with $w_{s}$ (and hence with $w_{s} ^{r_s,\nu,z}$) on its vertices. What we obtain is a Lipschitz continuous piecewise affine function $v_{s} ^{r_s,\nu,z}$ with values in the unit ball of $\R^2$, that coincides with $w_{s} ^{r_s,\nu,z}$ on $\Z_{r_s,\nu,z} ^{n} (F)$. We point out that $v_{s} ^{r_s,\nu,z}$ is not defined on the whole $F$, but only on a union of cubes contained in $F$, which anyway contains $F_{\sqrt{n}r_s}$, and hence also $F_{\eta-\sqrt{n}r_s}$, because of~\eqref{eq:defn_s0}.

From \cite[equation~(4.21)]{2009-ARMA-AC}, we deduce that
$$E_{r_s,\nu,z}(w_{s} ^{r_s,\nu,z}, F_{\eta-\sqrt{n}r_s}) \geq
\frac{1}{2\abs{\log r_s}} \int_{F_{\eta}} |\nabla v_{s} ^{r_s,\nu,z}(x)|^2 \,dx ,$$
while from \cite[Lemma~2]{2009-ARMA-AC} we deduce that there exists a constant $c_0>0$ such that
$$E_{r_s,\nu,z}(w_{s} ^{r_s,\nu,z},F_{\eta-\sqrt{n}r_s}) \geq
\frac{c_0}{r_s ^2 \abs{\log r_s}}\int_{F_{\eta}} (1-|v_{s} ^{r_s,\nu,z}(x)|^2)^2\,dx.$$

As a consequence, if for every $(\ep,\beta) \in (0,1)^2$ we consider the Ginzburg-Landau density
$$e_{\ep,\beta}(v):=\frac{1-\beta}{2\abs{\log \ep}} |\nabla v|^2+ \frac{\beta c_0}{\ep^2 \abs{\log \ep}} (1-|v|^2)^2,$$
and, for every open set $\Omega \subset \R^n$, the corresponding energy
$$GL_{\ep,\beta}(v,\Omega):=\int_{\Omega} e_{\ep,\beta}(v)\,dx,$$
then we obtain that
\begin{equation}\label{eq:E>GL}
E_{r_s,\nu,z}(w_{s} ^{r_s,\nu,z},F_{\eta-\sqrt{n}r_s}) \geq GL_{r_s,\beta}(v_{s} ^{r_s,\nu,z}, F_{\eta}),
\end{equation}
for every $\eta \in (0,1)$, every $(r,\nu,z)\in (0,1-\eta)\times V_n \times \Cube$, every $s\in (s_0(\eta),1)$, and every $\beta \in (0,1)$.

\medskip


Now, we recall the following result from~\cite{2005-Indiana-ABO}, which is a refinement of a previous estimate proved in~\cite{1999-SIAM-Jerrard}. We point out that our definition of the Ginzburg-Landau density differs from that in~\cite{2005-Indiana-ABO} by a factor of $(1-\beta)/\abs{\log \eps}$. Consequently, the same modification applies to the estimate below, where $e_{\eps,\beta}$ denotes the Ginzburg-Landau density according to our convention.

\begin{lemma}[Lemma~3.10 in \cite{2005-Indiana-ABO}]\label{lemma:ABO}
There exists a universal constant $K>0$ such that for every $\beta\in (0,1)$ and every $\rho>0$ there exist two constants $D_0,D_1>0$, depending only on $\beta$ and $\rho$, such that for every $d\in \N_+$ and every smooth function $v:[-\ell,\ell]^2 \to \R^2$ with $|v|>1/2$ on $\partial [-\ell,\ell]^2$ and $\abs{\deg(v/|v|,\partial [-\ell,\ell]^2)}=d$ it holds that
\begin{equation}\label{estimate:ABO}
GL_{\ep,\beta} (v,[-\ell,\ell]^2)+ K\rho \ell \int_{\partial [-\ell,\ell]^2} e_{\ep,\beta}(v)\, d\mathcal{H}^1 \geq (1-\beta)\pi d \biggl(1-\frac{\log (2\ell)+D_1(1+\log d)}{\abs{\log \ep}}\biggr),
\end{equation}
for every $\ep\in (0,2\ell)$ such that $(\ep/2\ell)\abs{\log(\ep/2\ell)}<D_0/d$.
\end{lemma}

\begin{rem}\label{rem:ABO-transversality}
We will need to apply Lemma~\ref{lemma:ABO} to (the restriction to some planar domain of) the functions $v_{s} ^{r_s,\nu,z}$, which are not smooth, but Lipschitz continuous and piecewise affine. However, if we approximate $v_{s} ^{r_s,\nu,z}$ with smooth functions, every term in \eqref{estimate:ABO} passes to the limit, provided the boundary of the domain intersects the set where $v_{s} ^{r_s,\nu,z}$ is not smooth (which is a finite union of $(n-1)$-dimensional simplices) transversally.
\end{rem}

Now we are ready to prove Proposition~\ref{prop:local-liminf}.

\begin{proof}[Proof of Proposition~\ref{prop:local-liminf}]
First of all, we observe that without loss of generality we can assume that the liminf in \eqref{th:local-liminf} is finite, and up to extracting a subsequence (that we do not relabel), we can also assume that it is a limit, so that in particular
\begin{equation}\label{eq:limsup_finite}
\limsup_{s\to 1^{-}} \hspace{0.03cm} (1-s)^2 [w_s]_{H^{\frac{1+s}{2}}(F)} ^{2} <+\infty.
\end{equation}

Let us fix $\eta\in (0,\delta/4)$ and $s_0=s_0(\eta)$ such that~\eqref{eq:defn_s0} holds.

Let us fix also $r\in (0,1-\eta)$ and $s\in (s_0,1)$, and let $\mathcal{N}_{s,r}\subset V_n\times \Cube$ be the negligible set provided by Lemma~\ref{lemma:N_ep} applied with
$$\Omega=F,\qquad M=M_s, \qquad w=w_s,\qquad w_k=w_{s,k},\qquad \ep=r_s.$$

Now we need to select the ``good'' cubes in $\mathscr{C}_{r_s,\nu,z}(F_{\eta -\sqrt{n}r_s})$, namely those cubes for which we can find a suitable estimate on the oscillation of $w_s$ on each of their one-dimensional edge, and such that $M_{s}$ does not intersect any of these edges.

To this end, for every $(\nu,z)\in (V_n\times \Cube) \setminus \mathcal{N}_{s,r}$, we consider the set
$$\mathscr{O}_{s,r,\nu,z}:=\left\{Q\in \mathscr{C}_{r_s,\nu,z}(F_{\eta -\sqrt{n}r_s}) : \osc(w_s,I)> \frac{\sqrt{2}}{n} \text{ for some edge $I$ of $Q$} \right\},$$
of cubes for which we do not have a good oscillation estimate on all their edges, and the set
$$\mathscr{M}_{s,r,\nu,z}:=\left\{Q\in \mathscr{C}_{r_s,\nu,z}(F_{\eta -\sqrt{n}r_s}) : I \cap M_{s} \neq \varnothing \text{ for some edge $I$ of $Q$}\right\},$$
of cubes with at least one edge intersecting $M_{s}$.

Now we fix a small positive real number $\lambda >0$ and we consider the set
$$\mathcal{B}_{s,r}:=\mathcal{N}_{s,r}\cup \left\{(\nu,z)\in (V_n \times \Cube)\setminus\mathcal{N}_{s,r} : \mathcal{H}^0(\mathscr{O}_{s,r,\nu,z}\cup \mathscr{M}_{s,r,\nu,z}) > \frac{2\lambda^2 \delta}{r_{s} ^{n-1}}\right\}.$$

Since each edge belongs to at most $2^{n-1}$ cubes, from Lemma~\ref{lemma:holder_edges} we deduce that
\begin{multline*}
    \frac{\mathcal{H}^{n(n+1)/2}\left(\left\{(\nu,z)\in (V_n \times \Cube) \setminus\mathcal{N}_{s,r} : \mathcal{H}^0(\mathscr{O}_{s,r,\nu,z}) > \frac{\lambda^2 \delta}{r_{s} ^{n-1}}\right\}\right)}{\mathcal{H}^{n(n-1)/2}(V_n)} \\[1ex]
    \leq \frac{r_s ^{n-1}}{\lambda^2\delta}\cdot 2^{n-1} \cdot\frac{n^2}{2}\cdot \frac{c_{n,s_0}}{ r_s^{n-1-s}} (1-s) [w_s]_{H^{\frac{1+s}{2}}(F)} ^{2}
    =
    \frac{c_{n,s_0} 2^{n-2} n^2}{\lambda^2\delta } r_{s} ^{s} (1-s) [w_s]_{H^{\frac{1+s}{2}}(F)} ^{2},
\end{multline*}
while from Lemma~\ref{lemma:M_edges} we deduce that
    \[
        \frac{\mathcal{H}^{n(n+1)/2}\left(\left\{(\nu,z)\in (V_n \times \Cube) \setminus\mathcal{N}_{s,r} : \mathcal{H}^0(\mathscr{M}_{s,r,\nu,z}) > \frac{\lambda^2 \delta}{r_{s} ^{n-1}}\right\}\right)}{\mathcal{H}^{n(n-1)/2}(V_n)}
        \leq \frac{c_n 2^{n-1}}{\lambda^2 \delta} \mathcal{H}^{n-1}(M_s).
    \]

Since $\mathcal{N}_{s,r}$ is negligible, combining these two estimates we obtain that
    \[
        \frac{\mathcal{H}^{n(n+1)/2}(\mathcal{B}_{s,r})}{\mathcal{H}^{n(n-1)/2}(V_n)} \leq \frac{c_{n,s_0} 2^{n-2} n^2}{\lambda^2\delta } \frac{r_{s} ^{s}}{(1-s)} (1-s)^2 [w_s]_{H^{\frac{1+s}{2}}(F)} ^{2} + \frac{c_n 2^{n-1}}{\lambda^2 \delta} \mathcal{H}^{n-1}(M_{s}),
    \]
which in turn, recalling~\eqref{eq:limsup_finite}, yields
    \begin{equation}
        \label{est:meas_B}
        \lim_{s\to 1^{-}}\mathcal{H}^{n(n+1)/2}(\mathcal{B}_{s,r})=0 \qquad \forall r\in (0,1-\eta).
    \end{equation}

Let $\pi_{1}: \R^{2} \times \R^{n-2} \to \R^{2}$ and $\pi_{2}: \R^2 \times \R^{n-2} \to \R^{n-2}$ be the orthogonal projections onto the two factors, and let consider the set
$$\mathfrak{S}_{s,r,\nu,z}:=\left\{\sigma\in E_{\eta} : \mathcal{H}^{0}(\{Q\in \mathscr{O}_{s,r,\nu,z}\cup \mathscr{M}_{s,r,\nu,z} : \sigma\in \pi_{2}(Q)\}) > \frac{\lambda\delta}{r_{s}} \right\},$$
that consists of the points $\sigma\in E_\eta$ for which the number of ``bad'' cubes intersecting $\{\sigma\}\times (-\delta,\delta)^2$ is not small enough.

We observe that, if $\chi_A$ denotes the indicator function of the set $A$, then
\begin{eqnarray*}
\frac{\lambda\delta}{r_{s}} \mathcal{H}^{n-2}(\mathfrak{S}_{s,r,\nu,z})&\leq&
\int_{E_\eta} \sum_{Q\in \mathscr{O}_{s,r,\nu,z}\cup \mathscr{M}_{s,r,\nu,z}} \chi_{\pi_{2}(Q)}(\sigma) \,d\sigma\\
&=& \sum_{Q\in \mathscr{O}_{s,r,\nu,z}\cup \mathscr{M}_{s,r,\nu,z}} \mathcal{H}^{n-2}(\pi_{2}(Q)) \\
&\leq& p_{n}r_{s}^{n-2} \mathcal{H}^0(\mathscr{O}_{s,r,\nu,z}\cup \mathscr{M}_{s,r,\nu,z}),
\end{eqnarray*}
where $p_n:=\sup\{\mathcal{H}^{n-2}(\pi_{2}(R_{\nu}(\Cube)):\nu\in V_n\}$ is a constant depending only on the dimension $n$.

As a consequence, if $(\nu,z)\in (V_n \times \Cube) \setminus \mathcal{B}_{s,r}$, then we have that
\begin{equation}\label{est:meas_Y}
\mathcal{H}^{n-2}(\mathfrak{S}_{s,r,\nu,z})\leq 2\lambda p_n.
\end{equation}

Finally, for every $\sigma\in E_{\eta}$ let us consider the set
$$\mathcal{L}_{s,r,\nu,z}(\sigma):=\{\ell\in (\delta/2,\delta-\eta): \exists Q \in \mathscr{O}_{s,r,\nu,z}\cup \mathscr{M}_{s,r,\nu,z} \text{ such that } \gamma_{\sigma,\ell}\cap Q \neq \varnothing \},$$
namely the set of $\ell \in (\delta/2,\delta-\eta)$ for which $\gamma_{\sigma,\ell}$ intersects some of the ``bad'' cubes. We recall that we have chosen $\eta<\delta/4$, so $(\delta/2,\delta-\eta)$ is not empty.

We observe that for every $\sigma\in E_{\eta} \setminus \mathfrak{S}_{s,r,\nu,z}$ it holds that
\begin{equation}\label{est:meas_L}
\mathcal{H}^{1}(\mathcal{L}_{s,r,\nu,z}(\sigma)) \leq \sum_{Q\in\mathscr{O}_{s,r,\nu,z}\cup \mathscr{M}_{s,r,\nu,z}} \diam (Q\cap (\{\sigma\}\times(-\delta,\delta)^2))\leq \frac{\lambda\delta}{r_{s}} \sqrt{n} r_{s}=\lambda \sqrt{n} \delta.    
\end{equation}

As a consequence, if we take $\lambda$ small enough, and more precisely such that $\lambda\sqrt{n}\delta<\delta/2-\eta$, then for every $(\nu,z)\in (V_n \times \Cube)\setminus \mathcal{B}_{s,r}$ and every $\sigma\in E_{\eta}\setminus \mathfrak{S}_{s,r,\nu,z}$, we can find many curves $\gamma_{\sigma,\ell}$ that do not intersect any ``bad'' cube.

For every such curve, namely for every $\gamma_{\sigma,\ell}$ with $\ell \notin \mathcal{L}_{s,r,\nu,z}(\sigma)$, let us denote by $U_{\sigma,\ell}$ the union of the cubes in $\mathscr{C}_{r_s,\nu,z}(F_{\eta-\sqrt{n}r_s})$ intersected by $\gamma_{\sigma,\ell}$ and with $S_{\sigma,\ell}$ the union of their edges.

We claim that for every $\ell\notin \mathcal{L}_{s,r,\nu,z}(\sigma)$, with $(\nu,z)\notin \mathcal{N}_{s,r}$,  it holds that
\begin{equation}\label{claim_v}
|v_{s} ^{r_s,\nu,z}(x)|>1/2 \quad \forall x\in U_{\sigma,\ell}
\qquad \text{and} \qquad
\abs{\deg(v_{s} ^{r_s,\nu,z}/|v_{s} ^{r_s,\nu,z}|, \gamma_{\sigma,\ell})} = d.
\end{equation}

Indeed, if $Q\notin \mathscr{O}_{s,r,\nu,z}$, then from the estimate on the oscillation of $w_s$ we deduce that $|w_s(x)-w_s(y)|\leq \sqrt{2}/n$ for every couple $(x,y)$ of vertices of $Q$ that are connected by an edge. Since every couple of vertices of a $n$-dimensional cube can be connected by a path made of at most $n$ edges, by the triangle inequality we deduce that $|w_s(x)-w_s(y)|\leq \sqrt{2}$ for every couple $(x,y)$ of vertices of $Q$, hence $w_s$ sends the vertices of $Q$ in a quarter of the circle. Since by construction $v_{s} ^{r_s,\nu,z}(Q)$ is contained in the convex hull of the image of the vertices of $Q$, and on these vertices $v_{s} ^{r_s,\nu,z}$ coincides with $w_s$, we deduce that actually $|v_{s} ^{r_s,\nu,z}|\geq\sqrt{2}/2>1/2$ on each cube $Q\notin \mathscr{O}_{s,r,\nu,z}$, and hence on $U_{\sigma,\ell}$.

Moreover, $\gamma_{\sigma,\ell}$ can be continuously deformed inside $U_{\sigma,\ell}$ (and in particular inside $F^*$) onto a curve $\gamma$ contained in $S_{\sigma,\ell}$, and hence not intersecting $M_{s}$, because the cubes in $U_{\sigma,\ell}$ do not belong to $\mathscr{M}_{s,r,\nu,z}$.

Therefore our hypotesis implies that $\abs{\deg(w_{k,s},\gamma)}=d$ for every $k\in\N$, so \eqref{eq:lim_deg} yields $\abs{\deg(w_s,\gamma)}=d$.

Finally, we observe that the estimate on the oscillation implies that $v_{s} ^{r_s,\nu,z}/|v_{s} ^{r_s,\nu,z}|$ is homotopic to $w_s$ on each of the edges in $S_{\sigma,\ell}$, because in these edges $w_s$ take values in an arc of length smaller than $\pi/2$ and $v_{s} ^{r_s,\nu,z}$ in its convex hull, which does not contain the origin. Hence $v_{s} ^{r_s,\nu,z}/|v_{s} ^{r_s,\nu,z}|$ is homotopic to $w_s$ on $\gamma$, so we deduce that
    \[
        \abs{\deg(v_{s} ^{r_s,\nu,z}/|v_{s} ^{r_s,\nu,z}|,\gamma_{\sigma,\ell})} =\abs{\deg(v_{s} ^{r_s,\nu,z}/|v_{s} ^{r_s,\nu,z}|,\gamma)}
        =
        \abs{\deg(w_s,\gamma)}=d,
    \]
and this concludes the proof of \eqref{claim_v}.

The last thing we need to do before being able to exploit Lemma~\ref{lemma:ABO} is to make sure that we can control the boundary term. To this end it is enough to notice that for every fixed $\beta\in (0,1)$ it holds that
$$\int_{0}^{\delta-\eta} d\ell \int_{\gamma_{\sigma,\ell}} e_{r_s,\beta}(v_{s} ^{r_s,\nu,z})\,d\mathcal{H}^1 = \int_{\{\sigma\}\times [-\delta+\eta,\delta-\eta]^2} e_{r_s,\beta}(v_{s} ^{r_s,\nu,z})\,d\mathcal{H}^2,$$
and hence, thanks to \eqref{est:meas_L}, for almost every $\sigma\in E_\eta\setminus\mathfrak{S}_{s,r,\nu,z}$ there exists $\ell_\sigma\in (\delta/2,\delta-\eta)\setminus \mathcal{L}_{s,r,\nu,z}(\sigma)$ such that
$$\int_{\gamma_{\sigma,\ell_\sigma}} e_{r_s,\beta}(v_{s} ^{r_s,\nu,z})\,d\mathcal{H}^1 \leq \frac{1}{\delta/2-\eta-\lambda\sqrt{n}\delta} \int_{\{\sigma\}\times [-\delta+\eta,\delta-\eta]^2} e_{r_s,\beta}(v_{s} ^{r_s,\nu,z})\,d\mathcal{H}^2,$$
and such that both $\{\sigma\}\times [-\delta,\delta]^2$ and $\gamma_{\sigma,\ell_\sigma}$ intersect the singular set of $v_{s} ^{r_s,\nu,z}$ transversally (see Remark~\ref{rem:ABO-transversality}). We point out that for almost every $\sigma\in E_{\eta}$, this transversality condition may fail only for a negligible set of $\ell\in (0,\delta)$. We also remark that actually $\ell_\sigma$ depends also an all the other parameters, namely on $\eta,\beta,s,r,\nu,z$.

Now we fix $\rho>0$ and $s_1=s_1(\eta,\delta,\beta,\rho,d)\in [s_0,1)$ such that
$$r_s<2\ell \qquad\text{and}\qquad \frac{r_s}{2\ell}\left|\log \frac{r_s}{2\ell}\right|<\frac{D_0}{d}$$
for every $s\in (s_1, 1)$ and every $\ell \in (\delta/2,\delta)$, where $D_0$ is the constant in Lemma~\ref{lemma:ABO}. Then from \eqref{estimate:ABO} we deduce that
\begin{eqnarray*}
\lefteqn{\hspace{-8em}\left(1+\frac{K\rho\delta}{\delta/2-\eta-\lambda\sqrt{n}\delta}\right)\int_{\{\sigma\}\times [-\delta+\eta,\delta-\eta]^2} e_{r_s,\beta}(v_{s} ^{r_s,\nu,z})\,d\mathcal{H}^2}\\
\hspace{7em} &\geq& \int_{\{\sigma\}\times [-\ell_\sigma,\ell_\sigma]^2} e_{r_s,\beta}(v_{s} ^{r_s,\nu,z})\,d\mathcal{H}^2+K\rho\ell_\sigma \int_{\gamma_{\sigma,\ell_\sigma}} e_{r_s,\beta}(v_{s} ^{r_s,\nu,z})\,d\mathcal{H}^1 \\
&\geq& \int_{\{\sigma\}\times [-\ell_\sigma,\ell_\sigma]^2} \widetilde{e}_{r_s,\beta}(v_{s} ^{r_s,\nu,z})\,d\mathcal{H}^2+K\rho\ell_\sigma \int_{\gamma_{\sigma,\ell_\sigma}} \widetilde{e}_{r_s,\beta}(v_{s} ^{r_s,\nu,z})\,d\mathcal{H}^1 \\
&\geq& (1-\beta)\pi d \left(1-\frac{\log(2\ell_\sigma) +D_1(1+\log d)}{\abs{\log r_s}}\right),
\end{eqnarray*}
for every $r_s$ such that $s\in (s_1,1)$, where $ \widetilde{e}_{r_s,\beta}(v_{s} ^{r_s,\nu,z})$ denote the Ginzburg-Landau energy density of $v_{s} ^{r_s,\nu,z}$ as a function defined on the two dimensional set $[-\ell_\sigma,\ell_\sigma]^2 \times \{y\}$ or, equivalently, the energy density obtained by replacing the full gradient of $v_{s} ^{r_s,\nu,z}$ with its ``horizontal'' component.

We are now ready to put all the estimates togheter and conclude the proof. First, we observe that for every $(r,\nu,z)\in (0,1-\eta)\times V_n \times \Cube$ and every $s\in (s_1,1)$ it holds that
\begin{multline*}
GL_{r_s,\beta}(v_{s}^{r_s,\nu,z},F_{\eta})
\geq \int_{E_{\eta}\setminus \mathfrak{S}_{s,r,\nu,z}} d\sigma \int_{\{\sigma\}\times [-\delta+\eta, \delta-\eta]^2} e_{r_s,\beta}(v_{s} ^{r_s,\nu,z})\,d\mathcal{H}^2\\
\geq \mathcal{H}^{n-2}(E_{\eta} \setminus \mathfrak{S}_{s,r,\nu,z}) (1-\beta)\pi d \left(1-\frac{\log(2\delta)+D_1(1+\log d)}{\abs{\log r_s}}\right)\left(1+\frac{K\rho\delta}{\delta/2-\eta-\lambda\sqrt{n}\delta}\right)^{-1},
\end{multline*}
and hence, if $(\nu,z)\notin \mathcal{B}_{s,r}$, from \eqref{est:meas_Y} we deduce that
\begin{multline*}
GL_{r_s,\beta}(v_{s}^{r_s,\nu,z},F_{\eta})\\
\geq (\mathcal{H}^{n-2}(E_{\eta}) - 2\lambda p_n)(1-\beta)\pi d \left(1-\frac{\log(2\delta)+D_1(1+\log d)}{\abs{\log r_s}}\right)\left(1+\frac{K\rho\delta}{\delta/2-\eta-\lambda\sqrt{n}\delta}\right)^{-1}.
\end{multline*}

Combining this estimate with \eqref{eq:discretization_frac_norm},~\eqref{eq:E>GL}, Fatou's lemma and~\eqref{est:meas_B} we obtain that
\begin{multline*}
\liminf_{s\to 1^{-}} \hspace{0.03cm} (1-s)^2 [w_s]_{H^{\frac{1+s}{2}}(F)} ^{2}\\
\geq \frac{\omega_{n-1}}{n}\int_{0} ^{1-\eta} 8r\abs{\log r}\,dr (\mathcal{H}^{n-2}(E_{\eta}) - 2\lambda p_n)(1-\beta)\pi d \left(1+\frac{K\rho\delta}{\delta/2-\eta-\lambda\sqrt{n}\delta}\right)^{-1}
\end{multline*}

Letting $\rho\to 0^{+}$, $\beta\to 0^{+}$, $\lambda\to 0^{+}$, and $\eta\to 0^{+}$, and computing the integral
$$\int_{0}^{1}8r\abs{\log r}\,dr=2,$$
we obtain exactly~\eqref{th:local-liminf}.
\end{proof}

\subsection{Upper bound (proof of Proposition~\ref{prop:gamma-limsup})} \label{s:upper-bound}

This section is devoted to the proof of Proposition~\ref{prop:gamma-limsup}, which provides the limsup inequality for Theorem~\ref{energy conv teo}. 

As for the lower bound, we first need a flat version of the estimate, namely the computation of the energy of the vortex of degree $d$ around a $(n-2)$-plane. More precisely, we consider the vortex-type map $u_{\star} \colon \R^{n-2} \times (\R^2 \setminus \{ 0 \}) \to \Sp^1$ defined by
\begin{align*}
         u_{\star} (\sigma,p) := \frac{p}{\abs{p}} \qquad\forall (\sigma,p)\in \R^{n-2} \times (\R^2 \setminus \{ 0 \}),
\end{align*}
and the map $\phi_d :\Sp^1\to \Sp^1$ defined by $\phi_d(\theta):=\theta^d$ for every $\theta\in \Sp^1\subset \C$.

In the following lemma we compute the limit of the rescaled fractional seminorms of $\phi_d\circ u_{\star}$ on bounded domains. This computation will be crucial to treat the general case.

\begin{lemma}\label{l:est-std-vortex} Let $n \ge 2$ be an integer. Let $E \subset \R^{n-2}$ be an open set, and let $\delta \in (0,1)$. Let us set $F:=E \times (-\delta, \delta)^2$. Then,
    \begin{equation} \label{eq:vor-inside}
          \lim_{s \to 1^-} \hspace{0.03cm} (1-s)^2 [\phi_d\circ u_{\star}]^2_{H^{\frac{1+s}{2}}(F)} = \frac{ 2 \pi \omega_{n-1}}{n} d^2 \cdot \mathcal{H}^{n-2}(E),
    \end{equation}
and
\begin{equation} \label{eq:vor-outside}
    \lim_{s \to 1^-} \hspace{0.03cm} (1-s)^2
    \! \iint_{F \times F^c} 
    \frac{\abs{\phi_d(u_{\star}(x)) - \phi_d(u_{\star}(y))}^2}{\abs{x-y}^{n+1+s}} \, dx dy = 0.
\end{equation}
\end{lemma}

\begin{proof}
    For any $x=(\sigma, p) \in \R^{n-2} \times \R^2$ and for any $\mu > 0$ we define the quantities
    \begin{align} \label{eq:I-1}
        \mathrm{I}_{s, \mu}(x) & := \int_{B_{\mu}} \frac{\abs{\phi_d(u_{\star}(x+z)) - \phi_d(u_{\star}(x))}^2}{\abs{z}^{n+1+s}} \, dz, \\ 
        \notag
        \mathrm{II}_{s, \mu}(x) & := \int_{B_{\mu}^c} \frac{\abs{\phi_d(u_{\star}(x+z)) - \phi_d(u_{\star}(x))}^2}{\abs{z}^{n+1+s}} \, dz,
    \end{align}
   and we claim that for every $\ep\in (0,1)$ there exists $\kappa\in (0,1)$, depending only on $\ep$ and $d$, such that, if $p \neq 0$ and $\mu \leq \kappa \abs{p}$, we have
    \begin{gather} \label{eq:I-2-1}
        \mathrm{I}_{s, \mu}(x) \le \frac{\omega_{n-1}d^2}{n(1-s)} \cdot \frac{\mu^{1-s}}{(\abs{p}-\mu)^2}, \\[1ex]
        \label{eq:I-2-2}
        \mathrm{I}_{s, \mu}(x) \ge (1-\ep) \frac{\omega_{n-1}d^2}{n(1-s)} \cdot \frac{\mu^{1-s}}{(\abs{p}+\mu)^2},
    \end{gather}
    and
    \begin{equation} \label{eq:II-1}
        \mathrm{II}_{s, \mu}(x) \le 4 \omega_{n-1} \mu^{-(1+s)}.
    \end{equation}
    
    The estimate~\eqref{eq:II-1} is an easy consequence of the fact that $\phi_d \circ u_{\star}$ takes value in the circle and hence it is bounded. We turn our attention to the bounds~\eqref{eq:I-2-1} and~\eqref{eq:I-2-2}. We preliminarily observe that $\mathrm{I}_{s,\mu}(x)$ depends only on $\abs{p}$, where $x=(\sigma,p)$. Indeed, let $\lambda=\abs{p}$ and let $A \in O(2)$ such that $p = A(\lambda e_1)$, where $e_1 = (1,0) \in \R^2$, we define $B \in O(n)$ by the formula $B(\tau,q) = (\tau, A(q))$.

    Combining the identities $\abs{a-b}^2 = 2(1 - a \cdot b)$ and $\phi_d(A(a))\cdot \phi_d(A(b))=\phi_d(a)\cdot\phi_d(b) $, which are valid for any $a, b \in \Sp^1$, with the change of variable $z = B(\tau,q)$, we obtain
    \begin{align} \notag
         \mathrm{I}_{s, \mu}(x)
         & =
         2 \int_{B_{\mu}} \frac{1 - \phi_d(u_{\star}(B(\sigma + \tau, \lambda e_1 + q))) \cdot \phi_d(u_{\star}(B(\sigma,\lambda e_1))) }{(\abs{\tau}^2+\abs{q}^2)^{\frac{n+1+s}{2}}} \, d\tau dq \\[1.5ex] \notag
         & =
         2 \int_{B_{\mu}} 
         \bigg[1 - \phi_d \bigg(\frac{A(\lambda e_1 + q)}{\abs{A(\lambda e_1 + q)}}\bigg) \cdot \phi_d \bigg(\frac{A(\lambda e_1)}{\abs{A(\lambda e_1)}}\bigg) \bigg]
         \frac{d\tau dq}{(\abs{\tau}^2+\abs{q}^2)^{\frac{n+1+s}{2}}} \\[1.5ex] \notag
         & =
         2 \int_{B_{\mu}} 
         \bigg[1 - \phi_d\bigg(\frac{\lambda e_1 + q }{\abs{\lambda e_1 + q}}\bigg) \cdot e_1 \bigg]
         \frac{d\tau dq}{(\abs{\tau}^2+\abs{q}^2)^{\frac{n+1+s}{2}}} \\[1.5ex]
         \label{eq:I-3}
         & =
         2 \int_{B_{\mu}} 
         \bigg[1 - \cos \bigg(d \arctan \bigg(\frac{q_2}{\lambda+q_1}\bigg)\bigg) \bigg]
         \frac{d\tau dq}{(\abs{\tau}^2+\abs{q}^2)^{\frac{n+1+s}{2}}},
    \end{align}
    and the last expression depends only on $\lambda=\abs{p}$. 
    
    Now we choose $\kappa\in (0,1)$ such that 
    $$(1-\ep) \frac{d^2 t^2}{2}\leq 1-\cos(d \arctan(t))\leq \frac{d^2 t^2}{2} \qquad \forall t \in \bigg(-\frac{\kappa}{1-\kappa},\frac{\kappa}{1-\kappa} \bigg),$$
    and we observe that this choice depends only on $\ep$ and $d$.

    Applying this estimate with $t=q_2/(\lambda + q_1)$, which belongs to $(-\kappa/(1-\kappa),\kappa/(1-\kappa))$ whenever $|q|<\mu\leq \kappa\lambda$, from \eqref{eq:I-3} we get
    \begin{equation} \label{eq:I-4}
        \mathrm{I}_{s, \mu}(x)
        \le
        \int_{B_{\mu}}  
        \frac{d^2 q_2^2}{(\lambda + q_1)^2(\abs{\tau}^2+\abs{q}^2)^{\frac{n+1+s}{2}}} \, d\tau dq,
    \end{equation}
    and
    \begin{equation} \label{eq:I-5}
        \mathrm{I}_{s, \mu}(x)
        \ge
        \int_{B_{\mu}}  
        \frac{(1-\ep) d^2 q_2^2}{(\lambda + q_1)^2(\abs{\tau}^2+\abs{q}^2)^{\frac{n+1+s}{2}}} \, d\tau dq
    \end{equation}
    Finally, the estimates~\eqref{eq:I-2-1} and~\eqref{eq:I-2-2} follow easily from~\eqref{eq:I-4} and~\eqref{eq:I-5} after noticing that $\lambda - \mu \le \lambda + q_1 \le \lambda +\mu$ for any $(\tau, q) \in B_{\mu}$ and that
     \[   \int_{B_{\mu}}  
        \frac{q_2^2}{(\abs{\tau}^2+\abs{q}^2)^{\frac{n+1+s}{2}}} \, d\tau dq
        =
        \frac{1}{n}
        \int_{B_{\mu}}  
        \frac{d\tau dq}{(\abs{\tau}^2+\abs{q}^2)^{\frac{n-1+s}{2}}} 
        =
        \frac{\omega_{n-1}}{n} \frac{\mu^{1-s}}{1-s},\]
    which can be proved using polar coordinates.

    We are ready to prove~\eqref{eq:vor-inside}, we start by proving that:
    \begin{equation} \label{eq:limsup-in-1}
        \limsup_{s \to 1^-} \hspace{0.03cm} (1-s)^2 [\phi_d\circ u_{\star}]^2_{H^{\frac{1+s}{2}}(F)}
        \le
        \frac{ 2 \pi \omega_{n-1}}{n} d^2 \cdot \mathcal{H}^{n-2}(E).
    \end{equation}


    As before, we write $x=(\sigma, p) \in \R^{n-2} \times \R^2$. Combining Fubini's theorem and the fact that both $\mathrm{I}_{s,\mu}$ and $\mathrm{II}_{s,\mu}$ depend only on $p$, we obtain
    \begin{equation} \label{eq:limsup-in-2}
        \iint_{F \times \R^n} 
        \frac{\abs{\phi_d(u_{\star}(x)) - \phi_d(u_{\star}(y))}^2}{\abs{x-y}^{n+1+s}} \, dxdy
        = 
        \mathcal{H}^{n-2}(E) \int_{(-\delta, \delta)^2} \mathrm{I}_{s,\mu}(p) + \mathrm{II}_{s,\mu}(p) \, dp.
    \end{equation}

    Now, we fix $\theta \in (0,\kappa]$ and we consider $\mu=\theta \abs{p}$. Then, from~\eqref{eq:I-2-1} and~\eqref{eq:II-1}, we have
    \begin{equation}
        \label{eq:dominazione-2}
        \mathrm{I}_{s, \mu}(p) \le \frac{\omega_{n-1}d^2}{n(1-s)} \cdot \frac{\theta^{1-s}}{(1-\theta)^2} \abs{p}^{-(1+s)}, \quad \text{and} \quad
        \mathrm{II}_{s, \mu}(p) \le 4 \omega_{n-1} d^2 \theta^{-(1+s)} \abs{p}^{-(1+s)}.
    \end{equation}
    For any $R > 0$, we can integrate the previous inequalities in $\mathrm{D}_R \subset \R^2$ to get
    \begin{equation*}
        \int_{\mathrm{D}_R} \mathrm{I}_{s,\mu}(p) + \mathrm{II}_{s,\mu}(p) \, dp
        \le
        \frac{ 2 \pi \omega_{n-1}d^2}{n(1-s)^2} \cdot \frac{\theta^{1-s}}{(1-\theta)^2} R^{1-s}
        +
        \frac{8 \pi \omega_{n-1}d^2}{1-s} \theta^{-(1+s)} R^{1-s}.
    \end{equation*}
    Since $(-\delta,\delta)^2 \subset \mathrm{D}_{\sqrt{2} \delta}$, then the last formula implies that
    \begin{equation} \label{eq:UBN-1}
        \limsup_{s \to 1^{-}} \hspace{0.03cm} (1-s)^2 \! \int_{(-\delta, \delta)^2} \mathrm{I}_{s,\mu}(p) + \mathrm{II}_{s,\mu}(p) \, dp
        \le
        \frac{ 2 \pi \omega_{n-1}d^2}{n} \cdot \frac{1}{(1-\theta)^2}.
    \end{equation}
    Letting $\theta \to 0^+$ in the previous estimate, from~\eqref{eq:limsup-in-2} we obtain
        \begin{equation} \label{eq:UBN-2}
            \limsup_{s \to 1^-} \hspace{0.03cm} (1-s)^2 \iint_{F \times \R^n} 
        \frac{\abs{\phi_d(u_{\star}(x)) - \phi_d(u_{\star}(y))}^2}{\abs{x-y}^{n+1+s}} \, dxdy
            \le
            \frac{ 2 \pi \omega_{n-1}}{n} d^2 \cdot \mathcal{H}^{n-2}(E).
        \end{equation}
    In particular, this estimate implies~\eqref{eq:limsup-in-1}. Indeed, from the definition of the fractional seminorm it holds that
        \[
            [\phi_d\circ u_{\star}]^2_{H^{\frac{1+s}{2}}(F)}
        \le 
        \iint_{F \times \R^n} 
        \frac{\abs{\phi_d(u_{\star}(x)) - \phi_d(u_{\star}(y))}^2}{\abs{x-y}^{n+1+s}} \, dxdy.
        \]

    \smallskip

    Now, we prove that
    \begin{equation} \label{eq:inf1-in}
        \liminf_{s \to 1^-} \hspace{0.03cm} (1-s)^2 [\phi_d \circ u_{\star}]^2_{H^{\frac{1+s}{2}}(F)}
        \ge
        \frac{ 2 \pi \omega_{n-1} }{n} d^2 \cdot \mathcal{H}^{n-2}(E).
    \end{equation}

    We recall that $F=E \times (-\delta,\delta)^2$. For any $\eta \in (0,\delta/2]$, we set
    \[
        E_\eta := \{ x\in E : \mathrm{dist} (x,\R^{n-2}\setminus E) > \eta \},
        \qquad \text{and} \qquad
        F_\eta:=\{x \in F : \mathrm{dist} (x,\R^n\setminus F)>\eta\}.
    \]


    We observe that $\{ (x,y) : x \in F_\eta, \ y \in B_\eta(x) \} \subset F \times F$ and $E_\eta \times (-\delta + \eta, \delta - \eta)^2 \subset F_\eta$. Therefore, from~\eqref{eq:I-1}, Fubini's theorem and the definition of the fractional seminorm, we have that
    \begin{equation} \label{eq:inf2-in}
        [\phi_d\circ u_{\star}]^2_{H^{\frac{1+s}{2}}(F)}
        \ge
        \mathcal{H}^{n-2}(E_{\eta}) \int_{(-\delta + \eta, \delta - \eta)^2} \mathrm{I}_{s,\eta}(p) \, dp.
    \end{equation}
    
    Now, we fix $\theta \in (0,\kappa]$ and we notice that $\mathrm{I}_{s,\eta}(p) \ge \mathrm{I}_{s,\theta \abs{p}}(p)$ for any $\abs{p} \le \eta$. Therefore, from~\eqref{eq:I-2-2} and the fact that $\mathrm{D}_{\eta} \subset \mathrm{D}_{\delta - \eta} \subset (-\delta + \eta, \delta - \eta)^2$, we obtain
    \begin{equation*}
        \label{eq:inf3-in}
        \int_{(-\delta + \eta, \delta - \eta)^2} \mathrm{I}_{s,\eta}(p) \, dp
        \ge
        \int_{\mathrm{D}_{\eta}} \mathrm{I}_{s,\theta \abs{p}}(p) \, dp
        \ge
        (1-\ep) \frac{2 \pi \omega_{n-1}d^2}{n(1-s)^2} \cdot \frac{\theta^{1-s}}{(1+\theta)^2} \eta^{1-s}.
    \end{equation*}
    Moreover, combining the previous estimate with~\eqref{eq:inf2-in}, we deduce that
    \[
        \liminf_{s \to 1^-} \hspace{0.03cm} (1-s)^2
        [\phi_d\circ u_{\star}]^2_{H^{\frac{1+s}{2}}(F_{\eta})}
        \ge
        (1-\ep) \frac{2 \pi \omega_{n-1}d^2}{n} \cdot \frac{1}{(1+\theta)^2} \cdot \mathcal{H}^{n-2}(E_{\eta}).
    \]
    Letting first $\theta \to 0^+$, then $\eps \to 0^+$, and then $\eta \to 0^+$ in the previous estimate, we deduce~\eqref{eq:inf1-in}.
    

    \smallskip
    
    It remains to prove~\eqref{eq:vor-outside}. This follows easily from~\eqref{eq:vor-inside},~\eqref{eq:UBN-2}
    and the fact that
    \begin{equation*}
        \iint_{F \times F^c} 
        \frac{\abs{\phi_d (u_{\star}(x)) - \phi_d(u_{\star}(y))}^2}{\abs{x-y}^{n+1+s}} \, dx dy
        =
        \iint_{F \times \R^n} 
        \frac{\abs{\phi_d(u_{\star}(x)) - \phi_d(u_{\star}(y))}^2}{\abs{x-y}^{n+1+s}} \, dx dy
        -
        [u_{\star}]^2_{H^{\frac{1+s}{2}}(F)}.
    \end{equation*}
\end{proof}

We are ready to conclude the proof of the upper bound. 

\begin{proof}[Proof of Proposition~\ref{prop:gamma-limsup}]

Let $u_{\Sigma}$ be defined by~\eqref{eq:u-Sigma-2}. From the proof of Proposition~\ref{comp existence}, see~\eqref{eq:linking-condition-2}, we know that $u_{\Sigma} \in \mathfrak{F}_{\hspace{-.7pt} s}(\Sigma)$ for any $s \in (0,1)$. This is our candidate to show the validity of~\eqref{eq:limsup-ineq}.

We start with the case $\Omega = \R^n$, in this case we have $\mathcal{E}_{\alpha}(u, \R^n) = [u]_{H^\alpha(\R^n)}^2$, for any $\alpha \in (0,1)$. For any $i\in \{1,\dots, N\}$ and $\delta\in (0,\delta_0]$, let $\Phi_i$ and $\mathcal{U}_{i,\delta}$ be as at the beginning of Section~\ref{sec:main-results}, and let us set also $\mathcal{T}_\delta:=\mathcal{U}_{1,\delta}\cup\dots\cup \mathcal{U}_{m,\delta}$.


From the definition of the fractional seminorm, for any $\delta \le \delta_0$, we have 
    \begin{multline}\label{eq:split-total energy}
        [u_{\Sigma}]^2_{H^{\frac{1+s}{2}}(\R^n)}
        \leq
        \sum_{i=1} ^{m} [u_{\Sigma}]^2_{H^{\frac{1+s}{2}}(\mathcal{U}_{i,\delta})}
        + 2 \sum_{i\neq j} \iint_{\mathcal{U}_{i,\delta} \times \mathcal{U}_{j,\delta}} \frac{|u_{\Sigma}(x)-u_{\Sigma}(y)|^2}{|x-y|^{n+1+s}} \, dxdy \\
        +
        2 \iint_{\R^n \times \mathcal{T}_{\delta}^c} \frac{|u_{\Sigma}(x)-u_{\Sigma}(y)|^2}{|x-y|^{n+1+s}} \, dxdy.
    \end{multline}

We observe that for any $i \neq j$ the distance between $\mathcal{U}_{i,\delta}$ and $\mathcal{U}_{j,\delta}$ is greater than a certain fixed positive real number. Since $u_{\Sigma}$ takes value in the circle, this implies that
    \begin{equation}
        \label{eq:UBN-3}
        \limsup_{s \to 1^-} \hspace{0.03cm} (1-s)^2 \sum_{i\neq j} \iint_{\mathcal{U}_{i,\delta} \times \mathcal{U}_{j,\delta}} \frac{|u_{\Sigma}(x)-u_{\Sigma}(y)|^2}{|x-y|^{n+1+s}} \, dxdy = 0.
    \end{equation}

For what concerns the last integral in~\eqref{eq:split-total energy}, we have
    \begin{equation}
        \label{eq:UBN-4}
        \iint_{\R^n \times \mathcal{T}_{\delta}^c} \frac{|u_{\Sigma}(x)-u_{\Sigma}(y)|^2}{|x-y|^{n+1+s}} \, dxdy
        =
        [u_{\Sigma}]^2_{H^{\frac{1+s}{2}}(\mathcal{T}_{\delta}^c)}
        +
        \iint_{\mathcal{T}_{\delta} \times \mathcal{T}_{\delta}^c} \frac{|u_{\Sigma}(x)-u_{\Sigma}(y)|^2}{|x-y|^{n+1+s}} \, dxdy.
    \end{equation}
We notice that $u_{\Sigma} \in H^1_{\mathrm{loc}}(\mathcal{T}_{\delta}^c; \Sp^1)$ and it is constant outside a large ball, therefore we can use \cite[Remark 4.3]{HitGuide}, to infer that
    \begin{equation} \label{eq:zero-cent-A}
        \limsup_{s \to 1^-} \hspace{0.03cm} (1-s)^2 [u_{\Sigma}]^2_{H^{\frac{1+s}{2}}(\mathcal{T}_{\delta}^c)} = 0.
    \end{equation}
Moreover, we observe that
    \begin{align*} \notag
     \iint_{\mathcal{T}_{\delta} \times \mathcal{T}_{\delta}^c } \frac{|u_{\Sigma}(x)-u_{\Sigma}(y)|^2}{|x-y|^{n+1+s}} \, dxdy
     \le &
     \iint_{\mathcal{T}_{\delta/2} \times \mathcal{T}_{\delta}^c} \frac{|u_{\Sigma}(x)-u_{\Sigma}(y)|^2}{|x-y|^{n+1+s}} \, dxdy \\[1ex]
     & +
     \iint_{\mathcal{T}_{\delta/2}^c \times \mathcal{T}_{\delta/2}^c } \frac{|u_{\Sigma}(x)-u_{\Sigma}(y)|^2}{|x-y|^{n+1+s}} \, dxdy.
\end{align*}
For the same reasons used to prove~\eqref{eq:UBN-3} and~\eqref{eq:zero-cent-A} we deduce that
    \begin{equation}
        \label{eq:UBN-5}
        \limsup_{s \to 1^-} \hspace{0.03cm} (1-s)^2  \iint_{\mathcal{T}_{\delta} \times \mathcal{T}_{\delta}^c } \frac{|u_{\Sigma}(x)-u_{\Sigma}(y)|^2}{|x-y|^{n+1+s}} \, dxdy = 0.
    \end{equation}
Putting together~\eqref{eq:split-total energy},~\eqref{eq:UBN-3},~\eqref{eq:UBN-4},~\eqref{eq:zero-cent-A} and~\eqref{eq:UBN-5} we obtain
    \begin{equation}
        \label{eq:UBN-6}
        \limsup_{s \to 1^-} \hspace{0.03cm} (1-s)^2 [u_{\Sigma}]^2_{H^{\frac{1+s}{2}}(\R^n)}
        \le
        \limsup_{s \to 1^-} \hspace{0.03cm} (1-s)^2
        \sum_{i=1} ^{m} [u_{\Sigma}]^2_{H^{\frac{1+s}{2}}(\mathcal{U}_{i,\delta})}.
    \end{equation}
In particular, we have that $u_\Sigma=u_{\Sigma_i}^{d_i}$ in $\mathcal{U}_{i,\delta}$ and hence, if $\delta \le \delta_0/2$, it holds that $u_{\Sigma_i}^{d_i}(\Phi_i(\sigma, p))=\phi_{d_i}(u_{\star}(p))$ for every $(\sigma,p)\in \Sigma_i\times [-\delta,\delta]^2$. In view of~\eqref{eq:UBN-6}, it is enough to prove that
    \begin{equation}
        \label{eq:UBN-7}
        \limsup_{s \to 1^-} \hspace{0.03cm} (1-s)^2
        [u_{\Sigma}]^2_{H^{\frac{1+s}{2}}(\mathcal{U}_{i,\delta})} \le \frac{ 2 \pi \omega_{n-1}}{n} d_i^2 \cdot \mathcal{H}^{n-2}(\Sigma_i), \qquad \forall i \in \{1,\dots,m\}.
    \end{equation}
Therefore, we can drop the index $i$ and assume without loss of generality that $\Sigma$ is connected and with multiplicity $d$.

Combining the change of variable $(x,y)=(\Phi(\sigma,p), \Phi(\tau,q))$ with the definition of $u_{\Sigma}$ we obtain
    \begin{align} \label{eq:upper-1}
        [u_{\Sigma}]^2_{H^{\frac{1+s}{2}}(\mathcal{U}_{\delta})}
        =
        \iint_{\mathcal{V}_{\delta} \times \mathcal{V}_{\delta}} \frac{|\phi_d(u_{\star}(p))-\phi_d(u_{\star}(q))|^2}{|\Phi(\sigma,p)-\Phi(\tau,q)|^{n+1+s}} J \Phi (\sigma,p) J \Phi (\tau,q) \, d\sigma d\tau dp dq,
    \end{align}
where $\mathcal{V}_{\delta}:=\Sigma \times (-\delta, \delta)^2$ and $J \Phi(\sigma, p) := \abs{\det \mathrm{d}\Phi(\sigma, p)}$ is the Jacobian of $\Phi$ evaluated at the point $(\sigma, p)$.

We fix $\eta > 0$ and a positive real number $\delta(\eta) \le \delta_0/2$ for which~\eqref{eq:alm-isom-below} holds. We fix $\delta \le \delta(\eta)$ and we consider a collection $\mathcal{W}=\{ W_1,\dots, W_N \}$ associated to the parameters $\eta$ and $\delta$ such that properties a) and b) stated at the beginning of Section~\ref{sec:main-results} hold.
For any $i\in \{1,\dots, N\}$ we define the set $\mathcal{V}_{i,\delta} := W_i \times (-\delta,\delta)^2$. From~\eqref{eq:upper-1} and~\eqref{eq:jac-est-1} we deduce that
    \begin{align} \label{eq:upper-1B}
        [u_{\Sigma}]^2_{H^{\frac{1+s}{2}}(\mathcal{U}_{\delta})}
        \le
        (1+\eta)^{2n} 
        \sum_{i,j=1}^N \iint_{\mathcal{V}_{i,\delta} \times \mathcal{V}_{j,\delta}} \frac{|\phi_d(u_{\star}(p))-\phi_d(u_{\star}(q))|^2}{|\Phi(\sigma,p)-\Phi(\tau,q)|^{n+1+s}} \, d\sigma d\tau dp dq.
    \end{align}

We distinguish three cases:

i) if $i \neq j$ and the distance between $W_i$ and $W_j$ is greater than or equal to a positive number $r_0$, then from~\eqref{eq:alm-isom-below} and the fact that $\phi_d \circ u_{\star}$ has values in the circle we have
    \begin{align*} 
     \iint_{\mathcal{V}_{i,\delta} \times \mathcal{V}_{j,\delta}} \frac{|\phi_d(u_{\star}(p))-\phi_d(u_{\star}(q))|^2}{|\Phi(\sigma,p)-\Phi(\tau,q)|^{n+1+s}}
     & \le
     \frac{4}{(1-\eta)^{n+1+s}}
     \iint_{\mathcal{V}_{i,\delta} \times \mathcal{V}_{j,\delta}} \frac{1}{ (  r_0^2 + \abs{p-q}^2 )^{\frac{n+1+s}{2}}} \\[1ex]
     & \le
     \frac{\big( 8 \delta^2 \mathcal{H}^{n-2}(\Sigma) \big)^2}{\big((1-\eta)r_0\big)^{n+1+s}}.
    \end{align*}
In particular, from the previous estimate it follows immediately that
    \begin{equation}
        \label{eq:eI3}
        \lim_{s \to 1^-} \hspace{0.03cm} (1-s)^2 \!
        \iint_{\mathcal{V}_{i,\delta} \times \mathcal{V}_{j,\delta}} \frac{|\phi_d(u_{\star}(p))-\phi_d(u_{\star}(q))|^2}{|\Phi(\sigma,p)-\Phi(\tau,q)|^{n+1+s}} = 0.
    \end{equation}

ii) if $i \neq j$ and $W_i$ and $W_j$ are at zero distance, we can find an open set $U \subset \R^{n-2}$ and a bi-Lipschitz parameterization $\psi \colon U \to B$, where $B$ is some ball in $\Sigma$ which contains both $W_i$ and $W_j$. It is always possible to find such a map if $\delta$ is smaller than the injectivity radius of $\Sigma$.

We define the set $G_{i,\delta} := \psi^{-1}(W_i) \times (-\delta,\delta)^2 \subset \R^{n-2} \times \R^2$ and we consider the change of variable $(\sigma, \tau) =(\psi(z),\psi(w))$,
then it is clear that there exists $L > 0$, depending on the Lipschitz constants of $\psi$, $\Phi^{-1}$ and $\psi^{-1}$, such that
    \begin{equation} \label{eq:eI2-1}
        \iint_{\mathcal{V}_{i,\delta} \times \mathcal{V}_{j,\delta}} \frac{|\phi_d(u_{\star}(p))-\phi_d(u_{\star}(q))|^2}{|\Phi(\sigma,p)-\Phi(\tau,q)|^{n+1+s}}
        \le
        L
        \iint_{G_{i,\delta} \times G_{j,\delta}} \frac{|\phi_d(u_{\star}(p))-\phi_d(u_{\star}(q))|^2}{(\abs{z-w}^2 + \abs{p-q}^2)^{\frac{n+1+s}{2}}}.
    \end{equation}
We observe that $G_{j,\delta} \subset G_{i,\delta}^c$, therefore from~\eqref{eq:eI2-1} and~\eqref{eq:vor-outside} we have
    \begin{equation} \label{eq:eI2-3}
        \lim_{s \to 1^-} \hspace{0.03cm} (1-s)^2 \!
        \iint_{\mathcal{V}_{i,\delta} \times \mathcal{V}_{j,\delta}} \frac{|\phi_d(u_{\star}(p))-\phi_d(u_{\star}(q))|^2}{|\Phi(\sigma,p)-\Phi(\tau,q)|^{n+1+s}} = 0.
    \end{equation}

iii) if $i=j$, we consider the map $\varphi_i \times \mathrm{Id} \colon A_i \times \R^2 \to W_i \times \R^2$ and the open set
    \begin{equation} \label{eq:set-E}
        F_{i,\delta} :=(\varphi_i \times \mathrm{Id})^{-1}(\mathcal{V}_{i,\delta})
        =
        A_i \times (-\delta,\delta)^2,
    \end{equation}
where $\varphi_i$ is the bi-Lipschitz diffeomorphism appearing in~\eqref{eq:bi-Lip}.

From~\eqref{eq:upper-1} and~\eqref{eq:jacobians} and the change of variables $(\sigma, \tau)=(\varphi_i(z),\varphi_j(w))$, we obtain
    \begin{align} 
        \label{eq:upper-2}
        \iint_{\mathcal{V}_{i,\delta} \times \mathcal{V}_{i,\delta}} \frac{|\phi_d(u_{\star}(p))-\phi_d(u_{\star}(q))|^2}{|\Phi(\sigma,p)-\Phi(\tau,q)|^{n+1+s}}
        \le
        (1+\eta)^{2(n-2)} 
        \iint_{F_{i,\delta} \times F_{i,\delta}} \frac{|\phi_d(u_{\star}(p))-\phi_d(u_{\star}(q))|^2}{|\Phi(\varphi_i(z),p)-\Phi(\varphi_j(w),q)|^{n+1+s}}.
    \end{align}
Moreover, from~\eqref{eq:alm-isom-below} and~\eqref{eq:bi-Lip} we have
    \begin{equation} \label{eq:equal-index}
        \iint_{F_{i,\delta} \times F_{i,\delta}} \! \! \frac{|\phi_d(u_{\star}(p))-\phi_d(u_{\star}(q))|^2}{|\Phi(\varphi_i(z),p)-\Phi(\varphi_i(w),q)|^{n+1+s}}
        \!
        \le
        \!
        \frac{1}{(1-\eta)^{2(n+1+s)}} 
        \! 
        \iint_{F_{i,\delta} \times F_{i,\delta}} \! \! \frac{|\phi_d(u_{\star}(p))-\phi_d(u_{\star}(q))|^2}{(\abs{z-w}^2 + \abs{p-q}^2)^{\frac{n+1+s}{2}}}.
    \end{equation}
It follows immediately from~\eqref{eq:set-E} that
    \begin{equation}
        \label{eq:set-E2}
        \Phi\big((\varphi_i \times \mathrm{Id})(F_{i, \delta} \cap \R^{n-2} \times \{ 0 \}) \big) = W_i.
    \end{equation}
Therefore, combining~\eqref{eq:upper-2},~\eqref{eq:equal-index} and~\eqref{eq:limsup-in-1} we deduce that
    \begin{align} \notag
       \hspace{-2em} \limsup_{s \to 1^-} \hspace{0.02cm} (1-s)^2 \!  \iint_{\mathcal{V}_{i,\delta} \times \mathcal{V}_{i,\delta}} \frac{|\phi_d(u_{\star}(p))-\phi_d(u_{\star}(q))|^2}{|\Phi(\sigma,p)-\Phi(\tau,q)|^{n+1+s}}
        & \le
        \frac{(1+\eta)^{2(n-2)}}{(1-\eta)^{2(n+2)}}
        \frac{ 2 \pi \omega_{n-1}}{n} d^2 \cdot \mathcal{H}^{n-2}(A_i) \\[1ex]
         \label{eq:eI1}
        & \le
        \frac{(1+\eta)^{2(n-2)}}{(1-\eta)^{3n+2}}
        \frac{ 2 \pi \omega_{n-1}}{n} d^2 \cdot \mathcal{H}^{n-2}(W_i),
    \end{align}
where we used~\eqref{eq:set-E2} and again the estimates~\eqref{eq:alm-isom-below} and~\eqref{eq:bi-Lip}. 

Finally, from~\eqref{eq:upper-1B},~\eqref{eq:eI3},~\eqref{eq:eI2-3} and~\eqref{eq:eI1} we obtain
    \begin{align} \notag
        \limsup_{s \to 1^-} \hspace{0.03cm} (1-s)^2
        [u_{\Sigma}]^2_{H^{\frac{1+s}{2}}(\mathcal{U}_{\delta})}
        & \le
        \frac{(1+\eta)^{4(n-1)}}{(1-\eta)^{3n+2}}
        \frac{ 2 \pi \omega_{n-1}}{n} d^2 \cdot \sum_{i=1}^N \mathcal{H}^{n-2}(W_i) \\[1ex]
        \label{eq:eIT}
        & =
        \frac{(1+\eta)^{4(n-1)}}{(1-\eta)^{3n+2}}
        \frac{ 2 \pi \omega_{n-1}}{n} d^2 \cdot \mathcal{H}^{n-2}(\Sigma).
    \end{align}
Letting $\eta \to 0^+$, we get~\eqref{eq:UBN-7}.

We are left with the case of a general open set $\Omega$. By definition, we have
    \begin{equation}
        \label{eq:reduction-LB}
        \mathcal{E}_{\frac{1+s}{2}}(u_{\Sigma}, \Omega) \le [u_{\Sigma}]^2_{H^{\frac{1+s}{2}}(\R^n)} - [u_{\Sigma}]^2_{H^{\frac{1+s}{2}}(\R^n \setminus \overline{\Omega})}. 
    \end{equation}
Therefore, if we prove that for every open set $V \subset \R^n$ it holds that
    \begin{equation}
        \label{eq:new-lower-bnd-1}
        \liminf_{s \to 1^{-}} \hspace{0.03cm} (1-s)^2 [u_{\Sigma}]^2_{H^{\frac{1+s}{2}}(V)} \ge 
        \frac{2\pi\omega_{n-1}}{n} \sum_{i=1} ^{m} d_i ^2 \mathcal{H}^{n-2}(\Sigma_i \cap V),
    \end{equation}
then the conclusion will follow combining~\eqref{eq:reduction-LB}, the estimate on the whole space and the previous estimate with $V = \R^n \setminus \overline{\Omega}$.

To prove~\eqref{eq:new-lower-bnd-1}, we proceed as in the first part of the proof of Proposition~\ref{prop:Gamma-liminf}. We have that
    \begin{equation} \label{eq:new-LB-trivial}
        [u_{\Sigma}]^2_{H^{\frac{1+s}{2}}(V)}
        \geq
        \sum_{i=1}^{m} [u_{\Sigma}]^2_{H^{\frac{1+s}{2}}(V \cap \mathcal{U}_{i,\delta_0})},
    \end{equation}
so~\eqref{eq:new-lower-bnd-1} follows if we prove that
    \begin{equation} \label{est:new-low-2}
        \liminf_{s\to 1^{-}} \hspace{0.03cm} (1-s)^2 [u_{\Sigma}]^2_{H^{\frac{1+s}{2}}(V \cap \mathcal{U}_{i,\delta_0})}
        \geq
        \frac{2\pi\omega_{n-1}}{n} d_i^2 \cdot \mathcal{H}^{n-2}(\Sigma_i \cap V), \qquad \forall i \in \{1,\dots,m\}.
    \end{equation}
As before, since $u_\Sigma=u_{\Sigma_i}^{d_i}$ in $\mathcal{U}_{i,\delta_0}$, we can drop the index $i$ and assume without loss of generality that $\Sigma$ is connected and with multiplicity $d$.

 We let $\eta$, $\delta(\eta)$ and $\mathcal{W} = \{ W_1, \dots, W_N \}$ as in the first part of this proof. We fix $\delta \in (0, \delta(\eta))$, and for every $i \in \{1,\dots,N\}$ we consider the set
    \[
        W_i^{\delta}:=\Big\{\sigma \in W_i : \Phi \big(\{\sigma\} \times [-\delta,\delta]^2 \big) \subset V \cap \mathcal{U}_{\delta_0} \Big\}.
    \]
It is clear that
\begin{equation}\label{est:low2-B}
    [u_{\Sigma}]_{H^{\frac{1+s}{2}}(V \cap \mathcal{U}_{i,\delta_0})}^{2}
    \geq
    \sum_{i=1}^{N} [u_{\Sigma}]_{H^{\frac{1+s}{2}}(\Phi(W_i^{\delta} \times (-\delta,\delta)^2))} ^{2}.
\end{equation}
Combining the change of variables $(x,y)=(\Phi(\sigma,p), \Phi(\tau,q))$ with~\eqref{eq:alm-isom-below} and~\eqref{eq:jac-est-1}, and then the change of variables $(\sigma,\tau)=(\varphi_i(z), \varphi_i(w))$ with~\eqref{eq:bi-Lip} and~\eqref{eq:jacobians}, we deduce that
    \begin{equation}
         \label{est:low3-B}
        [u_\Sigma]_{H^{\frac{1+s}{2}}(\Phi(W_i^{\delta} \times (-\delta,\delta)^2))}^{2}
        \geq
        \frac{(1-\eta)^{2n+2(n-2)}}{(1+\eta)^{2(n+1+s)}}[\phi_d \circ u_{\star}]_{H^{\frac{1+s}{2}}(A_i^{\delta} \times (-\delta,\delta)^2)}^{2},
\end{equation}
where $A_i^{\delta}:=\varphi_i^{-1}(W_i^{\delta})$ and $\varphi_i$ is the bi-Lipschitz diffeomorphism appearing in~\eqref{eq:bi-Lip}. We remark that, if $\delta \le \delta_0/2$, then $u_{\Sigma}(\Phi(\sigma, p))=\phi_d(u_{\star}(p))$ for every $(\sigma,p) \in \Sigma \times [-\delta,\delta]^2$. 

From~\eqref{eq:bi-Lip},~\eqref{eq:vor-inside},~\eqref{est:low2-B} and~\eqref{est:low3-B} we obtain
    \begin{equation}
        \label{est:new-low-1}
        \liminf_{s \to 1^{-}} \hspace{0.03cm} (1-s)^2 [u_{\Sigma}]_{H^{\frac{1+s}{2}}(V \cap \mathcal{U}_{\delta_0})}^{2}
        \ge
        \frac{(1-\eta)^{4(n-1)}}{(1+\eta)^{3n+2}} \frac{2\pi \omega_{n-1}}{n} d^2
        \sum_{i=1}^N \mathcal{H}^{n-2}(W_i^{\delta}).
    \end{equation}
Moreover, since $W_i^{\delta} \to W_i \cap V$, as $\delta \to 0^+$, we have that
    \begin{equation}\label{eq:low5-B}
        \lim_{\delta\to 0^{+}} \sum_{i=1}^{N} \mathcal{H}^{n-2}(W_i^{\delta})
        =
        \sum_{i=1}^{N} \mathcal{H}^{n-2}(W_i \cap V)=\mathcal{H}^{n-2}(\Sigma \cap V).
    \end{equation}
From~\eqref{est:new-low-1},~\eqref{eq:low5-B} and letting $\eta \to 0^+$, we get~\eqref{est:new-low-2}. Finally, putting together~\eqref{eq:new-LB-trivial} and~\eqref{est:new-low-2} we deduce~\eqref{eq:new-lower-bnd-1}.

\end{proof}

At this point, Theorem~\ref{energy conv teo} follows from \eqref{eq:limsup-ineq} and the following density result (see also \cite[Proposition~8.6]{2005-Indiana-ABO}), together with a standard diagonal argument.

\begin{lemma}\label{Lemma:density}
Let $\Sigma$ be as in \eqref{defn:Sigma}. Then there exists a family $\{\Sigma_s\}$ of surfaces with unit multiplicity, namely
$$\Sigma_s:=\Sigma_{1,1,s}\cup \dots \cup \Sigma_{1,d_1,s} \cup\dots\dots\cup \Sigma_{m,1,s}\cup \dots \cup \Sigma_{m,d_m,s},$$
where, for every $s\in (0,1)$, $\Sigma_{i,j,s}$ are pairwise disjoint closed oriented surfaces, such that as $s\to 1^-$ it holds that $\mathcal{H}^{n-2}(\Sigma_s)\to \sum_{i=1}^{m} d_i\mathcal{H}^{n-2}(\Sigma_i)$ and $\Sigma_{i,j,s}\to \Sigma_{i}$ in the flat convergence of boundaries for every $i,j$.
\end{lemma}

\begin{proof}
For every $i\in\{1,\dots,m\}$, we fix a small positive number $\delta_i>0$ and a parametrization $\Phi_i:\Sigma_i\times [-\delta_i,\delta_i]^2\to \R^n$ of a tubular neighborhood of the surface $\Sigma_i$, so that the sets $\Phi_i(\Sigma_i\times [-\delta_i,\delta_i]^2)$ are pairwise disjoint. Then it is enough to set $\Sigma_{i,j,s}:=\Phi_i (\Sigma_i\times \{(1-s) v_{i,j}\})$, where $v_{i,j}\in [-\delta_i,\delta_i]^2$ are such that $v_{i,j}\neq v_{i,j'}$ when $j\neq j'$.

We observe that $\F(\Sigma_i - \Sigma_{i,j,s})\leq \mathcal{H}^{n-1}(\Phi_i(\Sigma_i\times \{tv_{i,j}:t\in [0,1-s]\}))\to 0$, and clearly $\mathcal{H}^{n-2}(\Sigma_{i,j,s})\to\mathcal{H}^{n-2}(\Sigma_i)$ as $s\to 1^-$.
\end{proof}


\section{Possible generalizations}\label{sec: poss generalizations}

In order to enlighten the main ideas in our proofs, avoiding as much as possible inessential technicalities, we stated and proved our main result in the case of codimension two submanifolds of the Euclidean space.

However, our method is quite robust, and could be generalized to more general settings. In this section, we briefly describe two possible generalizations, pointing out the what can be extended verbatim, and what requires some carefulness.

\subsection{Ambient Riemannian manifolds}\label{sbs ambient Riem}

The convergence results presented in this work (see Proposition \ref{prop:Gamma-liminf} and Proposition \ref{prop:gamma-limsup}) are local. Our main result Theorem \ref{energy conv teo} is stated globally essentially just because we do not know that the minimizer in \eqref{fraccontdef} is unique. Then, the localized fractional $s$-mass could depend on the choice of the minimizer, as explained in Remark \ref{rem: uniq storiella}.  

\vsp 
Apart from this technical point, our results are completely local. This means that our notion is well-suited to be extended to ambient Riemannian manifolds. The only point that requires care in doing so is that in an ambient $n$-dimensional Riemannian manifold $N$, even for a single $(n-2)$-dimensional closed, oriented surface (with multiplicity one) $\Sigma \subset N$, the family $\mathfrak{F}_{\hspace{-.7pt} s}(\Sigma)$ defined in \eqref{Flink} could be empty because of topological obstructions.

\vsp
Indeed, in our proof of the existence of a competitor in $\R^n$ (that is, Proposition~\ref{comp existence}), we exploit Theorem~\ref{teo:kirby}, for which it is essential that $H_{n-2}(\R^n)=\{0\}$ to say that every cycle is a boundary. On an ambient Riemannian manifold $N$ with $H_{n-2}(N) \neq \{0\}$, it is not true in general that every such cycle (the connected components of $\Sigma$) is a boundary, so a competitor might fail to exist. 

\vsp 
For example, take the flat three torus $T^3 $, which we consider as $[-1,1]^3$ with suitably identified faces. Then, for the closed curve $\Sigma_\circ=\{(0,0,t) \, : \, t\in [-1,1]\}$ no such $u \in \mathfrak{F}_{\hspace{-.7pt} s}(\Sigma_\circ)$ exists. If such $u$ would exist then some level set of $u$ would be an immersed surface spanning $\Sigma_\circ$, and this would imply $[\Sigma_\circ]=0$ in $H_1(T^3)$, which is not the case. 

\vsp 
This phenomenon is due to the fact that $\Sigma_\circ$ above does not lie in the natural class on which the $s$-mass is well-defined on ambient Riemannian manifolds. This natural class is the family of $(n-2)$-dimensional boundaries with integer coefficients. 

\vsp 
Precisely, for $m\geq 1$ positive integer and $d=(d_1,\dots,d_m)\in \N_+ ^m$ one can consider
\begin{equation}\label{eq: higher codim sigma def 2}
\Sigma = \partial(d_1 M_1 \cup \dots \cup d_m M_m) ,
\end{equation}
where $\{M_i\}_{i=1}^m \subset \R^n$ are $C^2$, connected, oriented hypersurfaces with boundary. Note that being an $(n-2)$-dimensional boundary puts some restrictions on the support, multiplicity and orientation of the connected components of $\Sigma$.

\vsp
On a closed Riemannian manifold, everything done in this manuscript works for $(n-2)$-dimensional boundaries instead of $(n-2)$-dimensional cycles. In particular, on every Riemannian manifold with trivial $(n-2)$-dimensional homology (where every $(n-2)$-dimensional cycle with integer coefficients is an $(n-2)$-dimensional boundary), our proofs apply verbatim. That is, on these manifolds, our notion of $s$-mass is well-defined and, with similar proofs, it $\Gamma$-converges to the $(n-2)$-Hausdorff measure with multiplicity. This is the case, for example, for Riemannian spheres $(\Sp^n,g)$ for $n\ge 3$. 

\subsection{Higher codimension}\label{sbs h codim}

Apart from being extended to ambient Riemannian manifolds, our notion is suited to be extended to higher codimension too. In what follows, we work in ambient $\R^n$ and $k\in\{2,\dots,n-1\}$ denotes the codimension of our surfaces. 

Similarly to \eqref{eq: higher codim sigma def 2}, the natural domain for the $s$-mass in codimension $k$ is the space of $k$-dimensional boundaries with integer coefficients and trivial normal bundle. That is, for $m\geq 1$ positive integer and $d=(d_1,\dots,d_m)\in \N_+ ^m$ let
\begin{equation}\label{eq: higher codim sigma def}
\Sigma = \partial(d_1 M_1 \cup \dots \cup d_m M_m) = \bigcup_{i=1}^m \bigcup_{j=1}^{m_i} d_i \Sigma_{i}^j ,
\end{equation}
where $\{M_i\}_{i=1}^m \subset \R^n$ are $C^2$, connected, oriented $(n-k+1)$-dimensional submanifolds with $\partial M_i = \cup_{j=1}^{m_i} \Sigma_{i}^j$, and each $\Sigma_{i}^j$ is connected and has trivial normal bundle (which is always true when $k=2$, but not necessarily if $k\ge 3$). Here, for every $ i\in \{1, \dots, m\}$ and $ j\in \{1, \dots, m_i\}$, each $\Sigma_{i}^j$ comes with the induced orientation from $M_i$. 

Denote also
\begin{equation}
L_k(\Sigma):=\big\{S \subset \R^n\setminus \Sigma : S \text{ is a bi-Lipschitz image of $\Sp^{k-1}$} \big\} .
\end{equation}

With this notation, we can consider the notion of fractional $s$-mass of codimension $k\in\{2,\dots,n-1\}$, for surfaces $\Sigma$ as in \eqref{eq: higher codim sigma def}, defined by 
$$\M_{k,s}(\Sigma):=\min \left\{[u]_{W^{\frac{k-1+s}{k},k}(\R^n)} ^{k} :
u\in \mathfrak{F}_{\hspace{-.7pt} s,k}^{\hspace{.9pt} w} (\Sigma) \right\},$$
where
   \begin{multline}\label{eq: high codim 1}
   \mathfrak{F} _{s,k} (\Sigma)  := \Big\{ u\in C^1(\R^n\setminus \Sigma;\Sp^{k-1})  \cap W^{\frac{k-1+s}{k}, k}(\R^n; \Sp^{k-1}) \ \text{such that}\\
    \abs{\deg(u,S)}= \bigg| \sum_{i=1}^m \sum_{j=1}^{m_i} d_i\link(S,\Sigma_i^j) \bigg| \mbox{ for every} \ S \in L_k(\Sigma)\Big\},
\end{multline}
and
    \begin{equation*}
        \mathfrak{F}_{\hspace{-.7pt} s,k}^{\hspace{.9pt} w}  (\Sigma)
        :=
        \biggl\{ u \colon \R^n \to \Sp^{k-1} \colon \exists \{u_h\}\subset \mathfrak{F}_{\hspace{-.7pt} s,k} (\Sigma),\ u_h\to u \sms \mbox{in } W^{\alpha,k}_{\mathrm{loc}}(\R^n;\Sp^{k-1}) \ \forall \alpha \in \biggl( \hspace{-0.05cm} 0, \frac{k-1+s}{2}\biggr)
        \biggr\}.
    \end{equation*}

With this definition, it should be possible to extend Theorem~\ref{energy conv teo} to any codimension. 

\vsp 

Indeed, for the upper bound, at least in the case in which $m=1$, since the normal bundle of each $\Sigma_1 ^j$ is trivial and all the multiplicities agree with the one of $M_1$, one can still use the standard vortex of $\R^k$ (or a fixed map of higher degree) in each fiber of a tubular neighborhood of $\Sigma_1^j$, and extend this map to an admissible competitor that is defined on $\R^n$ and is constant outside a neighborhood of $M_1$. The case in which $m>1$ here is more delicate, since we cannot use the complex multiplication to build a single map out of the ones relative to each component. Nevertheless, this can still be done trivially when the manifolds $M_i$ are disjoint and probably, with some more effort, also if they are not.

\vsp 
On the other hand, in order to obtain the lower bound, one can still reduce the fractional energy to a discrete one and bound it from below with the following version of the Ginzburg-Landau energy
$$GL_{k,\ep,\beta}(v,\Omega):=\int_{\Omega} \left(\frac{1-\beta}{k\abs{\log \ep}} |\nabla v|^k + \frac{\beta c_k}{\ep^2 \abs{\log\ep}} (1-|v|^2)^k \right) dx,$$
for which we still have an analogue of Lemma~\ref{lemma:ABO} (see \cite[Lemma~3.10]{2005-Indiana-ABO}) that provides an estimate for the energy on $k$-dimensional cubes. Moreover, for every $s\in (0,1)$ we have the embedding $W^{\frac{k-1+s}{k},k}(\R^{k-1})\hookrightarrow C^0(\R^{k-1})$, which provides an analogue of Lemma~\ref{lemma:frac_embedd} that we can apply on each $(k-1)$-face of the cubes in $\mathscr{C}_{\ep,\nu,z}(F)$.

\vsp
Very similarly, Serra in \cite[Section 5]{SerraSurv} proposed a notion of $(n-k)$-dimensional fractional volume $V_{n-k,s}$ (we keep the notation $V_{d,s}$ to match the one in \cite{SerraSurv}, where $d=n-k$) for boundaries that are level sets of (regular values of) smooth functions. Precisely, for $s\in (0,1)$ Serra defined the $(n-k)$-dimensional fractional volume as 
\begin{equation*}
    V_{n-k,s}(\Sigma) := \inf\left\{[u]_{H^{\frac{k-1+s}{2}}(\R^n)} ^{2} :
u\in \mathfrak{F} _{s,k} (\Sigma) \right\} ,
\end{equation*}
where the family $\mathfrak{F} _{s,k} (\Sigma)$ is defined as in \eqref{eq: high codim 1} without the multiplicities, and with the space $H^{\frac{k-1+s}{2}}(\R^n; \Sp^{k-1})$ in place of $W^{\frac{k-1+s}{k},k}(\R^n; \Sp^{k-1})$.

\vsp 
We remark that on a closed Riemannian manifold $M$, the spaces $H^\alpha(M)$ for $\alpha\ge 1$ still have few (all equivalent) canonical definitions corresponding to the natural extensions of the ones in \cite[Section 4.1]{SerraSurv}. For example, see \cite{Yang2013OnHO} for the Caffarelli-Silvestre extension of higher-order fractional Laplacian. Hence, even in this setting of general codimension, this notion of fractional mass seems to be canonical also if $\R^n$ is replaced by a closed Riemannian manifold without boundary. 

\vsp 
   Note that the family of level-set surfaces in \cite{SerraSurv} satisfies our general hypothesis above. Indeed, in \cite{SerraSurv} the $(n-k)$-dimensional fractional volume $V_{n-k,s}$ is defined for surfaces $\Sigma=F^{-1}(y)$, where $F:\R^n\to \R^k$ is smooth and $y$ is a regular value of $F$. These surfaces have trivial normal bundle, since in this setting, being $y$ a regular value of $F$, $\{ \nabla F^i(y) \}_{i=1}^k $ are $k$ independent, nonzero sections of the normal bundle of $\Sigma$. Hence, they trivialize each connected component of $\Sigma$. 
  Moreover, $\Sigma=F^{-1}(y)$ is the boundary of $M=(F/\abs{F})^{-1}(\theta)$, which is an embedded $(n-k+1)$-surface when $\theta\in\Sp^{k-1}$ is a regular value of $F/\abs{F}$. This means that the class \eqref{eq: higher codim sigma def} contains all regular level sets of smooth functions.

\appendix

\section{Appendix: energy asymptotics of the planar vortex}\label{sec:Appendix}

In the following result, we show the correct asymptotic of the fractional energy in the model case of the standard vortex in $\R^2$. This has already been done in Lemma \ref{l:est-std-vortex} in the case $n\geq 3$.

Here, using a slicing procedure, we give a simple proof of this fact for $n=2$ and when the domain is a ball. Since this result is not essential for our purpose, we just sketch the proof, leaving some details to the reader.

\begin{lemma}\label{lem: vortex slicing} Let $u_\star : \mathrm{D_1}\setminus \{0\} \subset \R^2 \to \Sp^1$ be the standard vortex in two dimensions $u_\star(p)=p/|p|$. Then 
\begin{equation*}
      \lim_{s \to 1^-} (1-s)^2 [u_\star]^2_{H^{\frac{1+s}{2}}(\mathrm{D_1})} = 2 \pi^2 .
\end{equation*}
\end{lemma}
\begin{proof} 
    By the slicing formulas \eqref{eq: slic for}-\eqref{eq:planar_slicing_squares} we have 
    \begin{equation*}
        [u_\star]^2_{H^{\frac{1+s}{2}}(\mathrm{D_1})} =
        \frac{1}{2} \int_{ \Sp^1} d\theta  \int_{-1} ^1 [u_\star]^2_{H^\frac{1+s}{2}(S_{\theta, r})} \, dr ,
    \end{equation*}
    where $S_{\theta, r}:=\{r\theta^\perp + \xi\theta:\xi\in (-\sqrt{1-r^2},\sqrt{1+r^2})\}$. By the rotational invariance of the vortex we see that $[u_\star]^2_{H^\frac{1+s}{2}(S_{\theta, r})}$ does not depend on $\theta \in \Sp^1$, but only on $|r|$. Hence 
    \begin{equation*}
         [u_\star]^2_{H^{\frac{1+s}{2}}(B_1)} = 2\pi \int_0^1 [u_\star]^2_{H^\frac{1+s}{2}(S_{e_1, r})} \, dr ,
    \end{equation*}
    where $e_1=(1,0)$ and $S_{e_1, r} = \{ (x,r) \in B_1  \mid  -\sqrt{1-r^2} \le x \le \sqrt{1-r^2} \}$. Moreover
    \begin{align*}
        [u_\star]^2_{H^\frac{1+s}{2}(S_{e_1, r})} & = \int_{-\sqrt{1-r^2}}^{\sqrt{1-r^2}} \int_{-\sqrt{1-r^2}}^{\sqrt{1-r^2}} \frac{\Big|\frac{(x,r)}{\sqrt{x^2+r^2}} - \frac{(y,r)}{\sqrt{y^2+r^2}} \Big|^2}{|x-y|^{2+s}} \, dxdy \\ &= r^{ -s} \iint_{I_r \times I_r}  \frac{ \Big|\frac{(x,1)}{\sqrt{x^2+1}} - \frac{(y,1)}{\sqrt{y^2+1}} \Big|^2}{|x-y|^{2+s}} \, dxdy ,
    \end{align*}
    where we have substituted in the integral $x=r\widetilde x$,  $y=r\widetilde y$ (without renaming the variables) and we have set $I_r = \{|x|\le \sqrt{r^{-2}-1}\}$. Note that $I_r \nearrow \R$ as $r \to 0^+$. Hence, setting $u_1(x):= \frac{(x,1)}{\sqrt{x^2+1}} : \R \to \R^2 $ we can summarize the above computations as
    \begin{equation*}
         [u_\star]^2_{H^{\frac{1+s}{2}}(\mathrm{D_1})} = 2\pi \int_0 ^1 r^{-s} [u_1]^2_{H^{\frac{1+s}{2}}(I_r)} \, dr .
    \end{equation*} 
    
From \cite[Remark 4.3]{HitGuide} we deduce that
$$\lim_{s\to 1^-} (1-s)   [u_1]_{H^{\frac{1+s}{2}}(I_r)}^2 = 2 \int_{I_r} | u_1' (x)|^2\,dx,$$
and it is not difficult to check that the convergence is uniform with respect to $r\in (0,1)$.

Since $(1-s)r^{-s}\mres(0,1] \wconv{} \delta_{\{0\}}$ in duality with $C^0([0,1])$ and
\begin{equation*}
    \int_{\R} |u_1 '(x)|^2\,dx = \int_\R \left| \partial_x\left( \frac{x}{\sqrt{1+x^2}} \right)\right|^2 + \left| \partial_x\left( \frac{1}{\sqrt{1+x^2}} \right)\right|^2  dx  = \int_\R \frac{1}{(1+x^2)^2} dx = \frac{\pi}{2},
\end{equation*}
we conclude that
\begin{equation*}
    \lim_{s \to 1^-} (1-s)^2 [u_\star]^2_{H^{\frac{1+s}{2}}(\mathrm{D_1})}= 2\pi \int_0 ^1 (1-s)r^{-s}\cdot  (1-s)[u_1]^2_{H^{\frac{1+s}{2}}(I_r)} \, dr  = 2\pi \cdot 2\cdot \frac{\pi}{2} = 2\pi^2.
\end{equation*}
\end{proof}

\begin{rem}
Note that the constant $2\pi^2$ is in accordance with the one in Theorem~\ref{energy conv teo} for $n=2$. 
\end{rem}

\bigskip \noindent
\textbf{Acknowledgements.} We are very grateful to Joaquim Serra for proposing this problem to us, and to Andrea Malchiodi and Daniel Stern for many valuable comments on a preliminary version of this work. 

We also thank Clara Antonucci, Federica Bertolotti, Giovanni Italiano, Matteo Migliorini, and Diego Santoro for several fruitful discussions.

M.~C. is funded by NSF grant DMS-2304432 and is grateful to Stanford University for its kind hospitality during the realization of part of this work.

M.~F. is a member of the PRIN Project 2022AKNSE4 {\em Variational and Analytical aspects of Geometric PDEs}.

M.~F. and N.~P. are members of the \selectlanguage{italian}{``Gruppo Nazionale per l'Analisi Matematica, la Probabilità e le loro Applicazioni''} (GNAMPA) of the \selectlanguage{italian}{``Istituto Nazionale di Alta Matematica''} (INdAM).

N.~P. acknowledges the MIUR Excellence Department Project awarded to the Department of Mathematics, University of Pisa, CUP I57G22000700001.

\begin{center}
    \textbf{Statements and Declarations}
\end{center}

\noindent \textbf{Conflict of interest.} The authors have no relevant financial or non-financial interests to disclose.

\smallskip

\noindent \textbf{Data availability.} Data sharing is not applicable, since no data were used for this research.


\newcommand{\etalchar}[1]{$^{#1}$}
\renewcommand\refname{References}

\end{document}